\documentclass{amsart}
\usepackage{hyperref}
\usepackage{graphicx}

\usepackage{appendix}
\usepackage{amssymb}
\usepackage{amsmath}
\usepackage{amsthm}
\usepackage{wrapfig}
\usepackage{subfig}
\usepackage{enumerate}
\usepackage{fancybox}
\usepackage{framed}
\usepackage{xcolor}

\usepackage{ esint }
\usepackage{float}
\usepackage{color}
\usepackage[top=2.5cm,bottom=2.5cm,left=2.5cm,right=2.5cm]{geometry}

 \numberwithin{equation}{section}

\title[A study of 1--D TGV$^{2}$ minimisation problem]{A study of the one dimensional  total generalised variation regularisation problem. }

\author{Konstantinos Papafitsoros}
\address{Department of Applied Mathematics and Theoretical Physics, Cambridge Centre for Analysis, University of Cambridge, United Kingdom}
\email{k.papafitsoros@maths.cam.ac.uk}

\author{Kristian Bredies}
\address{Institute for Mathematics and Scientific Computing, Karl-Franzens University of Graz, Austria}
\email{kristian.bredies@uni-graz.at}

\date{\today}
\begin{document}

\begin{abstract}
In this paper we study the one dimensional second order total generalised variation regularisation (TGV) problem with $L^{2}$ data fitting term.  We examine some properties of this model and we calculate exact solutions using simple piecewise affine functions  as data terms. We investigate how these solutions behave with respect to the TGV parameters and we verify our results using numerical experiments.
\end{abstract}

\maketitle
\tableofcontents
\section{Introduction}

\subsection{Context and related work}
During the last years, the seek for appropriate regularisation techniques has become a major issue in the field of inverse problems and more particularly, in the area of mathematical imaging. 
 A regulariser of good quality is one that provides reconstructed images that are close to the desired result quantitatively and are aesthetically pleasing as well. A particularly efficient regulariser of this kind has been the total generalised variation of second order, $\mathrm{TGV}_{\beta,\alpha}^{2}$, introduced recently in \cite{TGV}. For an open and bounded domain $\Omega\subseteq\mathbb{R}^{d}$ and positive parameters $\alpha$ and $\beta$ the $\mathrm{TGV}_{\beta,\alpha}^{2}$ functional reads
\begin{equation}\label{tgvdef1}
\mathrm{TGV}_{\beta,\alpha}^{2}(u):= \sup\left\{\int_{\Omega} u\,\mathrm{div}^{2}v~dx:\; v\in \mathcal{C}_{c}^{2}(\Omega,S^{d\times d}),\;\|v\|_{\infty}\le \beta,\; \|\mathrm{div}v\|_{\infty}\le \alpha\right\},
\end{equation}
where $S^{d\times d}$ is the set of the $d\times d$ symmetric matrices and $u\in L^{1}(\Omega)$. 

The total generalised variation--based regularisation has the ability to adapt to the regularity of the data and images restored with this method are typically piecewise smooth that is to say, not only the discontinuities but also  the affine structures are preserved. This is in contrast to total variation--based reconstructed images \cite{rudin1992nonlinear, ChambolleLions} which exhibit an undesirable piecewise constant structure (staircasing effect). Recall the definition of the total variation (TV) functional:
\begin{equation}\label{tvdef}
\mathrm{TV}(u):=\sup\left \{\int_{\Omega}u\,\mathrm{div}v~dx:\;v\in \mathcal{C}_{c}^{1}(\Omega,\mathbb{R}^{d}),\;\|v\|_{\infty}\le 1 \right \}.
\end{equation}

Some successful applications of total generalised variation in image reconstruction and in related tasks have been done in image denoising \cite{TGV, tgvcolour}, image deblurring \cite{BredValk, tgvcolour}, reconstruction of MRI images \cite{knoll2011second}, diffusion tensor imaging \cite{doi:10.1137/120867172} and JPEG decompression \cite{bredies2012artifact}.
In most of the tasks above the reconstructed image is obtained as a minimiser of a variational problem of the type
\begin{equation}\label{general_variational}
\min_{u}\; \frac{1}{s}\int_{\Omega}|Tu-f|^{s}dx+\mathrm{TGV}_{\beta,\alpha}^{2}(u),
\end{equation}
where $s=1$ or $2$, $T$ is a linear operator and $f$ are the given corrupted data. In the task of denoising images corrupted by Gaussian noise, we have $s=2$ and $T$ is the identity operator. This is the case we are focusing on this paper.

Before the introduction of the total generalised variation, total variation had been one of the most popular choices for imaging tasks. As a consequence, much work has been done in order to investigate the mathematical properties of TV regularisation. The purpose of the present paper is to do the same for TGV regularisation. So far, most emphasis has been given to the study of the exact solutions of the following family of variational problems 
\begin{equation}\label{TV_variational}
\min_{u}\frac{1}{s}\int_{\Omega}|u-f|^{s}+\alpha\, \mathrm{TV}(u),
\end{equation}
where $\alpha$ is a positive parameter, $\Omega\subseteq \mathbb{R}^{d}$ is an open domain with $d=1$ or $2$ and $s=1$ or $2$.

The model with $s=2$ is typically used to denoise images corrupted by Gaussian noise. 
 In \cite{meyer2001oscillating}, Meyer
showed that if $f$ is a characteristic function of a circle then the solution of \eqref{TV_variational}  is equal to the characteristic function of the same circle with a decreased height. This loss of contrast, which is more intense as the value of the parameter $\alpha$ increases, is a characteristic feature of the total variation regularisation. In \cite{caselles2007discontinuity}, Caselles, Chambolle and Novaga showed that the jump set of the solution $u$ is contained in the jump set of the data $f$. In \cite{strong2003edge}, Strong and Chan provided exact solutions of TV regularisation for simple 1D and rotationally invariant 2D data and similar results are obtained through techniques analysed by Ring in \cite{ring2000structural}. In \cite{grasmair2007equivalence}, Grasmair proved the equivalence of TV regularisation and the taut-string algorithm, see also the work of Hinterberger et al. \cite{hinterberger2003tube}.

For results regarding exact solutions of the TV regularisation problem \eqref{TV_variational} with $s=1$ we refer the reader to \cite{chanL1, duvalL1, nikolova2002minimizers} among others.

Related work regarding the study of the analytical properties of some higher order regularisation methods can also be found in the literature. In \cite{poschl2008characterization} P\"{o}schl and Scherzer studied exact solutions in the case where the regulariser is the total variation of an arbitrary order derivative of $u$. In \cite{dal2009higher} a higher-order non-convex model is investigated by Dal Maso et al.. Regarding exact solutions of TGV regularisation, properties of analytical solutions for the one dimensional case for the $L^{1}$ data fitting case are studied in \cite{BrediesL1}. Finally in \cite{TGVbregman}, Benning et al. investigate the capability of infimal convolution, second order TV and TGV regularisation to recover certain data exactly apart from a loss of contrast.  We are also aware
of an unpublished work of P\"oschl and Scherzer \cite{scherzertgv} which was prepared at the
same time, independently of our work and also addresses exact solutions of $L^2$--TGV$^2$ in dimension
one. Their focus is on determining how to choose the parameters
such that the solutions do not coincide with the respective $L^2$--TV and
$L^2$--TV$^2$ solutions.

In this paper we are studying further the one dimensional second order TGV$_{\beta,\alpha}^{2}$ regularisation problem with $L^{2}$ data fitting term. The motivation is to understand deeper how this kind of regularisation behaves, by computing exact solutions for simple data functions and to investigate how these solutions change with respect to the values of the parameters $\alpha$ and $\beta$. \\

\subsection{Outline of the paper}
In Section \ref{section:preliminaries}, we fix our notation and we recall some basic notions regarding Radon measures and functions of bounded variation. We also introduce the TGV functional along with some of its basic properties.

In Section \ref{section:optimality}, we formulate our problem, i.e., one dimensional second order TGV regularisation with $L^{2}$ data fitting term and we derive the corresponding optimality conditions in the spirit of \cite{BrediesL1}, using Fenchel duality.

In Section \ref{section:properties}, we examine some of the basic properties of the exact solutions, e.g., behaviour near and away from the boundary, preservation of discontinuities and facts about the $L^{2}$--linear regression. We also show that  at least for even data, the TGV and TV regularisations coincide under some conditions. 

In Section \ref{section:computation}, which is the main section of the paper, we compute exact solutions of the $\mathrm{TGV}_{\beta,\alpha}^{2}$ regularisation problem for three different data functions: a piecewise constant function with a single jump, a piecewise affine function with a single jump and a hat function. Emphasis is given on how the characteristic features  of the solutions (discontinuities, piecewise affinity) are affected by the parameters $\alpha$ and $\beta$.

Finally, in Section \ref{section:numerics}, our computations are verified using numerical experiments.

\section{Preliminaries}\label{section:preliminaries}
In this section we recall some  definitions and results that we are going to use and we also fix our notation.
\subsection{Radon  measures and functions of bounded variation}
 We denote with $\mathcal{M}(\Omega)$ the space of real valued finite Radon measures. If $\mu$ is a Radon measure then $|\mu|$ denotes its total variation measure. From the Radon-Nikod\'{y}m theorem we have that there exists a $|\mu|$--integrable function, denoted with $\mathrm{sgn}(\mu)$ such that $\mu=\mathrm{sgn}(\mu)|\mu|$ and has the property that $\mathrm{sgn}(\mu)=1$, $|\mu|$--a.e.. The one dimensional Lebesgue measure is denoted with $\mathcal{L}$.
 
Recall that for an open $\Omega\subseteq \mathbb{R}^{d}$ we say that a function $u\in L^{1}(\Omega)$ is of bounded variation if its distributional derivative can be represented by a finite Radon measure $Du$, i.e.,
\[\int_{\Omega}uv'~dx=-\int_{\Omega} v~dDu,\quad \forall v\in C_{c}^{\infty}(\Omega).\]
Equivalently, $u$ is a function of bounded variation if it has finite total variation $\mathrm{TV}(u)$, where 
\[\mathrm{TV}(u):=\sup\left \{\int_{\Omega}u\,\mathrm{div}v~dx:\;v\in \mathcal{C}_{c}^{1}(\Omega,\mathbb{R}^{d}),\;\|v\|_{\infty}\le 1 \right \},
\]
and in that case $\mathrm{TV}(u)$ is equal to $|Du|(\Omega)$ which is the total variation measure of $Du$ evaluated in $\Omega$, see for example \cite{AmbrosioBV}. The space of functions of bounded variation is denoted with $\mathrm{BV}(\Omega)$ which is a Banach space under the norm $\|u\|_{\mathrm{BV}(\Omega)}:=\|u\|_{L^{1}(\Omega)}+|Du|(\Omega)$. The distributional derivative $Du$ of a BV function can be decomposed into the absolutely continuous $D^{a}u$  and the singular part $D^{s}u$ with respect to $\mathcal{L}$, i.e., $Du=D^{a}u+D^{s}u$ with $D^{a}u=u'\mathcal{L}$. Here $u'$ denotes the density function of $D^{\alpha}u$ with respect to $\mathcal{L}$.

 We say that a sequence of BV functions $(u_{n})_{n\in \mathbb{N}}$ converges to $u$ weakly$^{\ast}$ in $\mathrm{BV}(\Omega)$ if it converges in $L^{1}$ and the sequence of measures $(Du_{n})_{n\in \mathbb{N}}$ converges to $Du$ weakly$^{\ast}$ in the sense of Radon measures, i.e., $\lim_{n\to\infty}\int_{\Omega}v~dDu_{n}=\int_{\Omega}v~dDu$ for all $v\in \mathcal{C}_{0}(\Omega)$. If $\Omega$ has Lipschitz boundary then every sequence which is bounded with respect to $\|\cdot\|_{\mathrm{BV}(\Omega)}$ has a weakly$^{\ast}$ converging subsequence, see \cite{AmbrosioBV}.

For the rest of the paper $\Omega$ will be a bounded open interval of the real line, i.e., $\Omega=(a,b)$.

In the one dimensional case the notion of the precise representative is a useful one. For a function $u\in \mathrm{BV}(a,b)$ the pointwise variation of $u$ in $(a,b)$ is defined as 
\[
\mathrm{pV}(u,(a,b))=\sup\left \{\sum_{i=1}^{n-1}|u(x_{i+1})-u(x_{i})|:\; n\ge 2,\;a<x_{1}<\cdots<x_{n}<b \right \}.
\]
There exist representatives $\tilde{u}$ of $u$ that have the property
\[
\mathrm{pV}(\tilde{u},(a,b))=\inf\left \{\mathrm{pV}(v):\; v=u\;\; a.e. \right \}=|Du|(\Omega),
\]
and those are called good representatives of $u$. There exists a unique $c\in \mathbb{R}$ such that the functions
\[u^{l}(x)=c+Du((a,x)),\quad u^{r}=c+Du((a,x])\]
are good representatives where $u^{l}$ and $u^{r}$ are left and right continuous respectively.
We define the functions
\[\overline{u}(x)=\max \{u^{l}(x),u^{r}(x)\},\quad \underline{u}(x)=\min\{u^{l}(x),u^{r}(x)\},\]
which are also good representatives of $u$. The jump set of $u$, $J_{u}$, is defined as the set of atoms of $u$, i.e., $J_{u}=\{x\in (a,b):\; Du(\{x\})\ne 0\}$. We refer again the reader to \cite {AmbrosioBV}.

Another useful concept is the one of the Radon norm $\|\cdot\|_{\mathcal{M}}$. For a distribution $\mathcal{T}$ in $\Omega$ we define
\[\|\mathcal{T}\|_{\mathcal{M}}=\sup\left\{\langle \mathcal{T},v\rangle:\; v\in \mathcal{C}_{c}^{\infty}(\Omega),\; \|v\|_{\infty}\le 1 \right \}.\]
We have that $\mathcal{T}$ is represented by a finite Radon measure, say $\mu$, if and only if $\|\mathcal{T}\|_{\mathcal{M}}$ is finite and in that case it is equal to $|\mu|(\Omega)$. If $\mathcal{T}$ is represented by an integrable function $u$ then $\|u\|_{\mathcal{M}}=\|u\|_{L^{1}(\Omega)}$. If $u\in \mathrm{BV}(\Omega)$ then $\|Du\|_{\mathcal{M}}=|Du|(\Omega)$.

\subsection{Convex analysis}
If $X$, $X^{\ast}$ are two vector spaces placed in duality and $F$ is a real convex function defined on $X$, then $F^{\ast}$ denotes the convex conjugate of $F$:
\[F^{\ast}(x^{\ast})=\sup_{x\in X}\; \langle x^{\ast},x\rangle-F(x).\] 
If $A\subseteq X$ then $\mathcal{I}_{A}$ denotes the indicator function of $A$:
\[\mathcal{I}(x)=
\begin{cases}
0 & \text{if }\;x\in A,\\
+\infty& \text{if }\;x\notin A.
\end{cases}
\]

\subsection{Basic facts about TGV} In the arbitrary dimension, the second order total generalised variation functional is defined as 
\[\mathrm{TGV}_{\beta,\alpha}^{2}(u):= \sup\left\{\int_{\Omega} u\,\mathrm{div}^{2}v~dx:\; v\in \mathcal{C}_{c}^{2}(\Omega,S^{d\times d}),\;\|v\|_{\infty}\le \beta,\; \|\mathrm{div}v\|_{\infty}\le \alpha\right\},
\]
where $S^{d\times d}$ is the set of the $d\times d$ symmetric matrices and $u\in L^{1}(\Omega)$, see \cite{TGV}. This is a proper, convex, rotationally invariant functional, lower semicontinuous on each $L^{p}(\Omega)$, $1\le p<\infty$.
 In \cite{BredValk}, an equivalent formulation was proved:
\[\mathrm{TGV}_{\beta,\alpha}^{2}(u)=\min_{w\in\mathrm{BD}(\Omega)}\; \alpha \|Du-w\|_{\mathcal{M}}+\beta\|Dw\|_{\mathcal{M}},\]
where $\mathrm{BD}(\Omega)$ is the space of functions of bounded deformation \cite{temam1985mathematical}.
 In the one dimensional case the two formulations read
 \begin{eqnarray}
\mathrm{TGV}_{\beta,\alpha}^{2}(u)&=& \sup\left\{\int_{\Omega} uv''~dx:\; v\in \mathcal{C}_{c}^{2}(\Omega),\;\|v\|_{\infty}\le \beta,\; \|v'\|_{\infty}\le \alpha\right\}\label{tgv1d_sup}\\
\mathrm{TGV}_{\beta,\alpha}^{2}(u)&=&\min_{w\in\mathrm{BV}(\Omega)}\; \alpha \|Du-w\|_{\mathcal{M}}+\beta\|Dw\|_{\mathcal{M}}.\label{tgv1d_min}\
\end{eqnarray}
In \cite{BredValk}, it was also proved that there exist positive constants $c<C$ that depend only at the size of the domain such that for every $u$ with finite total generalised variation we have
\[c\|u\|_{\mathrm{BV}(\Omega)}\le \|u\|_{L^{1}(\Omega)}+\mathrm{TGV}_{\beta,\alpha}^{2}(u)\le C\|u\|_{\mathrm{BV}(\Omega)}.\]

\section{Formulation of the problem and optimality conditions}\label{section:optimality}
The problem we are interested to study is the one dimensional second order TGV minimisation problem with $L^{2}$ data fitting term, i.e,
\begin{equation}\label{basic_problem_gen}
\min_{u\in\mathrm{BV}(\Omega)}\; \frac{1}{2}\int_{\Omega}(u-f)^{2}dx+\mathrm{TGV}_{\beta,\alpha}^{2}(u),
\end{equation}
for positive parameters $\alpha$, $\beta$ and $f\in L^{2}(\Omega)$. Using the alternative formulation of TGV, \eqref{tgv1d_min}, we can see that problem \eqref{basic_problem_gen} is equivalent to 
\begin{equation}
\tag{$\mathcal{P}$}
\min_{\substack{u\in\text{BV}(\Omega)\\w\in\text{BV}(\Omega)}} \;\frac{1}{2}\int_{\Omega}(u-f)^{2}dx+ \alpha\|Du-w\|_{\mathcal{M}}+\beta\|Dw\|_{\mathcal{M}}.
\label{P}
\end{equation}
Let us note here that a solution $(u,w)$ for \eqref{P} exists, see \cite{BredValk}. Uniqueness is guaranteed for $u$ but not for $w$.
In fact $w$ is a solution to an $L^{1}$--$\mathrm{TV}$ minimisation problem:
\begin{eqnarray*}
w&=&\underset{w\in\mathrm{BV}(\Omega)}{\operatorname{argmin}}\; \alpha \|Du-w\|_{\mathcal{M}}+\beta\|Dw\|_{\mathcal{M}}\iff\\
w&=&\underset{w\in\mathrm{BV}(\Omega)}{\operatorname{argmin}}\; \|D^{a}u+D^{s}u-w\|_{\mathcal{M}}+\frac{\beta}{\alpha}|Dw|(\Omega)\iff\\
w&=&\underset{w\in\mathrm{BV}(\Omega)}{\operatorname{argmin}}\;\int_{\Omega}|u'-w|~dx+\frac{\beta}{\alpha}\mathrm{TV}(w).
\end{eqnarray*}

In order to study exact solutions of the problem \eqref{P} we essentially follow \cite{BrediesL1} where the corresponding $L^{1}$--data fitting term case is examined. We identify the predual problem of \eqref{P} and derive the optimality conditions using Fenchel--Rockafellar duality.

Consider the following problem
\begin{equation}
\tag{$\mathcal{P}'$}
\sup\left \{\int_{\Omega}fv''~dx-\frac{1}{2}\int_{\Omega}(v'')^{2}:\; v\in H_{0}^{2}(\Omega),\;\|v\|_{\infty}\le \beta,\; \|v'\|_{\infty}\le \alpha \right\}.
\label{Pprime}
\end{equation}
We shall prove that \eqref{Pprime} is the predual problem of \eqref{P}. Firstly, we show that \eqref{Pprime} has a solution indeed.

\newtheorem{first}{Proposition}[section]
\begin{first}
The problem \eqref{P} admits a unique solution in $H_{0}^{2}(\Omega)$.
\end{first}
\begin{proof}
Because of the estimate $\|v\|_{\infty}+\|v'\|_{\infty}\le C\|v\|_{H_{0}^{2}(\Omega)}$, for all $v\in H_{0}^{2}(\Omega)$, \cite{EvansPDEs}, we have that the set $K=\left\{v\in H_{0}^{2}(\Omega):\;\|v\|_{\infty}\le \beta,\; \|v'\|_{\infty}\le \alpha  \right\}$ is normed-closed and  since it is convex, it is also weakly closed. Moreover since $(\frac{1}{2}\|\cdot\|_{L^{2}(\Omega)}^{2})^{\ast}=\frac{1}{2}\|\cdot\|_{L^{2}(\Omega)}^{2}$ we have that
 \[\sup_{v\in L^{2}(\Omega)} \int_{\Omega}f v ~dx -\frac{1}{2}\int_{\Omega}v^{2}~dx=\Big(\frac{1}{2}\|\cdot\|_{L^{2}(\Omega)}^{2}\Big)^{\ast}(f)=\frac{1}{2}\|f\|_{L^{2}(\Omega)}^{2}.\]
Thus the supremum in \eqref{Pprime} is finite and we denote it with $\sup\mathcal{P}'$. Consider now a maximising sequence $(v_{n})_{n\in\mathbb{N}}$ such that 
\[\int_{\Omega} f v_{n}''~dx-\frac{1}{2}\int_{\Omega}(v_{n}'')^{2}~dx\to \sup\mathcal{P}'.\]
Then there exists a positive constant $M$ such that
\begin{equation}\label{boundpredual}
\left |\int_{\Omega} f v_{n}''~dx-\frac{1}{2}\int_{\Omega}(v_{n}'')^{2}~dx\right|\le M,\quad \forall n\in\mathbb{N}.
\end{equation}
We will show that the sequence $\left (\|v_{n}''\|_{L^{2}(\Omega)} \right)_{n\in\mathbb{N}}$ is bounded as well.
Suppose not, then $\limsup_{n} \int_{\Omega}(v_{n}'')^{2}~dx=\infty$. Thus, passing to a subsequence if necessary we have
\[
\int_{\Omega} f v_{n}''~dx-\frac{1}{2}\int_{\Omega}(v_{n}'')^{2}~dx\le \|f\|_{L^{2}(\Omega)}\|v_{n}''\|_{L^{2}(\Omega)}-\frac{1}{2}\|v_{n}''\|_{L^{2}(\Omega)}^{2}\to -\infty, 
\]
which is a contradiction from \eqref{boundpredual}. Hence, we have that the sequence $(v_{n})_{n\in\mathbb{N}}$ is bounded in $H_{0}^{2}(\Omega)$ and from the reflexivity of that space we get the existence of a subsequence $(v_{n_{k}})_{k\in\mathbb{N}}$ and a function $v\in H_{0}^{2}(\Omega)$ such that $v_{n_{k}}\to v$ weakly. Since $K$ is weakly closed we have that $v\in K$. Moreover the functional to maximise is weakly upper semicontinuous since $\int_{\Omega}f(\cdot)''~dx$ is continuous in $H_{0}^{2}(\Omega)$ and $-\frac{1}{2}\|\cdot\|_{L^{2}(\Omega)}^{2}$ is weakly upper semicontinuous. Thus
\[\sup \mathcal{P}' \ge \int_{\Omega} f v''~dx-\frac{1}{2}\int_{\Omega}(v'')^{2}~dx \ge\limsup_{k} \int_{\Omega} f v_{n_{k}}''~dx-\frac{1}{2}\int_{\Omega}(v_{n_{k}}'')^{2}~dx=\sup \mathcal{P}'\]
which means that $v$ is a solution to \eqref{Pprime}. This solution is unique as the maximising functional is strictly concave, defined on a convex domain.
\end{proof}

Define now $X=H_{0}^{2}(\Omega)\times H_{0}^{1}(\Omega)$, $Y=H_{0}^{1}(\Omega)\times L^{2}(\Omega)$, $\Lambda:X\to Y$, $F_{1}:X\to (-\infty,\infty]$, $F_{2}:Y\to (-\infty,\infty]$ with 
\begin{eqnarray}
\Lambda(v,\omega)&=&(\omega+v',\omega'),\\
F_{1}(v,\omega)&=& \mathcal{I}_{\{\|\cdot\|_{\infty}\le \beta\}}(v)+\mathcal{I}_{\{\|\cdot\|_{\infty}\le \alpha\}}(\omega),\\
F_{2}(\phi,\psi)&=&\mathcal{I}_{0}(\phi)+\int_{\Omega} f\psi~dx+\frac{1}{2}\int_{\Omega}\psi^{2}~dx.
\end{eqnarray}

It is easy to check that with these definitions the problem \eqref{Pprime} is equivalent to
\begin{equation}\label{temam_primal}
-\min_{(v,\omega)\in X} F_{1}(v,\omega)+F_{2}(\Lambda(v,\omega)).
\end{equation}
The dual problem of \eqref{temam_primal} is 
\begin{equation}\label{temam_dual}
\min_{(w,u)\in Y^{\ast}}F_{1}^{\ast}(-\Lambda^{\ast}(w,u))+F_{2}^{\ast}(w,u),
\end{equation}
see for example \cite{ekeland1976convex}. Moreover, $X$, $Y$ are Banach spaces, $F_{1}$, $F_{2}$ are proper lower semicontinuous functions and the following condition holds:
\begin{equation}\label{A.B.}
Y=\bigcup_{\lambda\ge 0} \lambda (\mathrm{dom}(F_{2})-\Lambda(\mathrm{dom}(F_{1}))).
\end{equation}
Then, see \cite{attouch1986duality}, we have that no duality gap occurs, i.e.,
\[\left ( \min_{(v,\omega)\in X} F_{1}(v,\omega)+F_{2}(\Lambda(v,\omega))\right ) +\left (\min_{(w,u)\in Y^{\ast}}F_{1}^{\ast}(-\Lambda^{\ast}(w,u))+F_{2}^{\ast}(w,u) \right )=0. \]
The fact that condition \eqref{A.B.} holds follows from the corresponding theorem in \cite{BrediesL1} where the same condition was proved for the same  $X$, $Y$, $F_{1}$, $\Lambda$ and for a $F_{2}$ with a smaller domain.
The next step is to identify the dual problem \eqref{temam_dual} with the problem \eqref{P}.

\newtheorem{Identify_dual}[first]{Proposition}
\begin{Identify_dual}
The problem
\begin{equation}\label{temam_dual_2}
\min_{(w,u)\in Y^{\ast}}F_{1}^{\ast}(-\Lambda^{\ast}(w,u))+F_{2}^{\ast}(w,u),
\end{equation}
is equivalent to \eqref{P} in the sense that $(w,u)$ solve \eqref{temam_dual_2} if and only if $(w,u)\in \mathrm{BV}(\Omega)^{2}$ and they solve \eqref{P}.
\end{Identify_dual}
\begin{proof}
The proof follows closely the proof of the corresponding theorem in \cite{BrediesL1}. Firstly, we look at $F_{1}^{\ast}$. For a pair $(\sigma,\tau)\in H_{0}^{2}(\Omega)^{\ast}\times H_{0}^{1}(\Omega)^{\ast}$ we have that 
\[F_{1}^{\ast}(\sigma,\tau)=\sup_{\substack{(v,\omega)\in H_{0}^{2}(\Omega)\times H_{0}^{1}(\Omega)\\ \|v\|_{\infty}\le \beta \\ \|\omega\|_{\infty}\le \alpha}}  \langle\sigma,v \rangle+\langle\tau,\omega \rangle=\beta\sup_{\substack{v\in H_{0}^{2}(\Omega)\\ \|v\|_{\infty}\le 1 }}  \langle\sigma,v \rangle+\alpha\sup_{\substack{w\in H_{0}^{1}(\Omega)\\  \|\omega\|_{\infty}\le 1}}  \langle\tau,\omega \rangle.\]
Using density arguments, one can check that
\begin{equation}\label{F1_ast}
F_{1}^{\ast}(\sigma,\tau)=\beta\sup_{\substack{v\in \mathcal{C}_{c}^{\infty}(\Omega)\\ \|v\|_{\infty}\le 1 }}  \langle\sigma,v \rangle+\alpha\sup_{\substack{w\in \mathcal{C}_{c}^{\infty}(\Omega)\\  \|\omega\|_{\infty}\le 1}}  \langle\tau,\omega \rangle.
\end{equation}
We also have for every $(w,u)\in Y^{\ast}$ and $(v,\omega)\in X$,
\begin{equation}\label{Lambda_ast}
-({\underbrace{\Lambda^{\ast}({\underbrace {w,u}_{\in Y^{\ast}}})}_{\in X^{\ast}}})({\underbrace{v,\omega}_{\in X}})=-(w,u)(\Lambda(v,\omega))=-\langle w, v'+\omega \rangle-\langle u,\omega' \rangle.
\end{equation}
Combining \eqref{F1_ast} and \eqref{Lambda_ast} we have that for every $(w,u)\in Y^{\ast}$ 
\begin{equation}\label{beforedstr}
F_{1}^{\ast}(-(\Lambda^{\ast}(w,u)))=\beta\sup_{\substack{v\in \mathcal{C}_{c}^{\infty}(\Omega) \\ \|v\|_{\infty}\le 1 }}  -\langle w,v' \rangle+\alpha\sup_{\substack{\omega\in  \mathcal{C}_{c}^{\infty}(\Omega)\\  \|\omega\|_{\infty}\le 1}}  -\langle w,\omega \rangle -\langle u,\omega'\rangle.
\end{equation}
Since $w\in H_{0}^{1}(\Omega)^{\ast}$ and $u\in L^{2}(\Omega)$, they can be considered as distributions and thus \eqref{beforedstr} can be written
\begin{equation}\label{afterdstr}
F_{1}^{\ast}(-(\Lambda^{\ast}(w,u)))=\beta\sup_{\substack{v\in \mathcal{C}_{c}^{\infty}(\Omega) \\ \|v\|_{\infty}\le 1 }}  \langle Dw,v \rangle+\alpha\sup_{\substack{\omega\in  \mathcal{C}_{c}^{\infty}(\Omega)\\  \|\omega\|_{\infty}\le 1}}  \langle Du-w,\omega\rangle.
\end{equation}
Since we are considering the minimisation problem \eqref{temam_dual_2}, both terms in \eqref{afterdstr} should be finite and thus $Dw$ is a Radon measure, $w$ an $L^{1}$ function (thus a BV function) and $Du$ a Radon measure, i.e., $u$ is a BV function as well. Thus if $(w,u)$ is a pair of minimisers of \eqref{temam_dual_2} we have that $(w,u)\in \mathrm{BV}(\Omega)^{2}$ and 
\[F_{1}^{\ast}(-(\Lambda^{\ast}(w,u)))=\beta \|Dw\|_{\mathcal{M}}+\alpha\|Du-w\|_{\mathcal{M}}.\]
Finally we have 
\begin{eqnarray*}
F_{2}^{\ast}(w,u)&=&\sup_{\substack{(\phi,\psi)\in Y \\ \phi=0}} \langle w,\phi \rangle +\int_{\Omega}u\psi~dx-\int_{\Omega}f\psi ~dx-\frac{1}{2}\int_{\Omega}\psi^{2}~dx\\
&=&\sup_{\psi\in L^{2}(\Omega)}\int_{\Omega}(u-f)\psi~dx-\frac{1}{2}\int_{\Omega}\psi^{2}~dx\\
&=&\left (\frac{1}{2}\|\cdot\|_{L^{2}(\Omega)}^{2} \right )^{\ast}(u-f)\\&=&\frac{1}{2}\int_{\Omega}(u-f)^{2}~dx.
\end{eqnarray*}
\end{proof}

We are now ready to derive the optimality conditions that link the solutions of the problems \eqref{P} and \eqref{Pprime}. We need the following definition and lemma from \cite{BrediesL1}:
\newtheorem{Sgn_def}[first]{Definition}
\begin{Sgn_def}
Let $\mu\in\mathcal{M}(\Omega)$. We define the set-valued sign, $\mathrm{Sgn}(\mu)$ as
\[\mathrm{Sgn}(\mu)=\{v\in L^{\infty}(\Omega)\cap L^{\infty}(\Omega,|\mu|):\; \|v\|_{\infty}\le 1,\;v=\mathrm{sgn}(\mu),\;|\mu|-a.e.\}\]
\end{Sgn_def}

\newtheorem{Sgn_lemma}[first]{Lemma}
\begin{Sgn_lemma}[\cite{BrediesL1}]
If $\mu\in\mathcal{M}(\Omega)$ then
\[\mathrm{Sgn}(\mu)\cap\mathcal{C}_{0}(\Omega)= \partial \|\cdot\|_{\mathcal{M}}(\mu)\cap\mathcal{C}_{0}(\Omega).\]
\end{Sgn_lemma}

The following proposition characterises minimisers of the problem \eqref{P}.
\newtheorem{Optimality_conditions}[first]{Proposition}
\begin{Optimality_conditions}[Optimality conditions]
A pair $(u,w)\in \mathrm{BV}(\Omega)^{2}$ is a minimiser for \eqref{P} if and only if there exists a function $v\in H_{0}^{2}(\Omega)$, such that
\begin{equation}\label{C1}
\tag{C$_{f}$}
v''=f-u,
\end{equation} 
\begin{equation}\label{C2}
\tag{C$_{\alpha}$}
-v'\in \alpha\,\mathrm{Sgn}(Du-w),
\end{equation}
\begin{equation}\label{C3}
\tag{C$_{\beta}$}
v\in \beta\,\mathrm{Sgn}(Dw).
\end{equation}
\end{Optimality_conditions}
\begin{proof}
Since there is no duality gap we have that the solutions $(v,\omega)$ and   $(w,u)$ of \eqref{Pprime} and \eqref{P} respectively (equivalently \eqref{temam_primal} and \eqref{temam_dual}) are linked through the optimality conditions
\begin{eqnarray}
(v,\omega)&\in&\partial F_{1}^{\ast}(-\Lambda^{\ast}(w,u)),\label{opt1}\\
\Lambda(v,\omega)&\in&\partial F_{2}^{\ast}(w,u),\label{opt2}
\end{eqnarray}
see for example \cite{ekeland1976convex}. Moreover, $(w,u)$ is a solution for \eqref{P} if there exists a $v\in H_{0}^{2}(\Omega)$ and $\omega=-v'$ such that \eqref{opt1}--\eqref{opt2} hold.
We have that condition \eqref{opt1} is equivalent to
\begin{eqnarray*}
&\Leftrightarrow& F_{1}^{\ast}(-\Lambda^{\ast}(w,u))+\langle (v,\omega),(\sigma,\tau)+\Lambda^{\ast}(w,u) \rangle \le F_{1}^{\ast}(\sigma,\tau),\;\;\forall (\sigma,\tau)\in X^{\ast}\\
&\Leftrightarrow& \alpha \|Du-w\|_{\mathcal{M}}+\beta\|Dw\|_{\mathcal{M}}+\langle (v,\omega),(\sigma-Dw,\tau-(Du-w)) \rangle \le \alpha \|\tau\|_{\mathcal{M}}+\beta\|\sigma\|_{\mathcal{M}} ,\;\;\forall (\sigma,\tau)\in X^{\ast}\\
&\Leftrightarrow& \alpha \|Du-w\|_{\mathcal{M}}+\beta\|Dw\|_{\mathcal{M}}+\langle v,\sigma-Dw\rangle + \langle \omega, \tau-(Du-w) \rangle \le \alpha \|\tau\|_{\mathcal{M}}+\beta\|\sigma\|_{\mathcal{M}} ,\;\;\forall (\sigma,\tau)\in X^{\ast}\\
&\Leftrightarrow& \begin{cases}  \alpha \|Du-w\|_{\mathcal{M}}+ \langle \omega, \tau-(Du-w) \rangle \le \alpha \|\tau\|_{\mathcal{M}}, \;\;\forall \tau\in H_{0}^{1}(\Omega)^{\ast}\\
\beta\|Dw\|_{\mathcal{M}}+\langle v,\sigma-Dw\rangle \le \beta \|\sigma\|_{\mathcal{M}}, \;\;\forall \sigma\in H_{0}^{2}(\Omega)^{\ast}
  \end{cases}\\
  &\Leftrightarrow& \begin{cases}  \alpha \|Du-w\|_{\mathcal{M}}+ \langle \omega, \psi-(Du-w) \rangle \le \alpha \|\tau\|_{\mathcal{M}}, \;\;\forall \tau\in \mathcal{M}(\Omega)\\
\beta\|Dw\|_{\mathcal{M}}+\langle v,\sigma-Dw\rangle \le \beta \|\sigma\|_{\mathcal{M}}, \;\;\forall \tau\in \mathcal{M}(\Omega)  \end{cases}\\
&\Leftrightarrow&\begin{cases}\omega\in\partial \|\cdot\|_{\mathcal{M}}(Du-w)\\ v\in\partial \beta \|\cdot\|_{\mathcal{M}}(Dw) \end{cases}
\end{eqnarray*}
which is equivalent to the following:
\begin{equation}\label{oa}
\omega\in \alpha\, \text{Sgn}(Du-w),
\end{equation}
\begin{equation}\label{ob}
v\in \beta\,\text{Sgn}(Dw).
\end{equation}
Now the second condition is equivalent to 
\begin{eqnarray}
&\Leftrightarrow& F_{2}^{\ast}(w,u)+\langle \Lambda(v,\omega),(\hat{w},\hat{u})-(w,u) \rangle \le F_{2}^{\ast}(\hat{w},\hat{u}),\;\;\forall (\hat{w},\hat{u})\in Y^{\ast}\nonumber\\
&\Leftrightarrow& \frac{1}{2}\int_{\Omega}(u-f)^{2}dx+\langle w+v',\hat{w}-w \rangle+\langle w',\hat{u}-u \rangle\le \frac{1}{2}\int_{\Omega}(\hat{u}-f)^{2}dx,\;\;\forall (\hat{w},\hat{u})\in Y^{\ast} \nonumber\\
&\Leftrightarrow&
\begin{cases}
0+\langle w+v',\hat{w}-w \rangle\le 0,\;\;\forall \hat{w}\in H_{0}^{1}(\Omega)^{\ast}\\
\frac{1}{2}\int_{\Omega}(u-f)^{2}dx+\langle w',\hat{u}-u \rangle\le \frac{1}{2}\int_{\Omega}(\hat{u}-f)^{2}dx,\;\;\forall \hat{u}\in L^{2}(\Omega)^{\ast}
\end{cases}\nonumber\\
&\Leftrightarrow&
\begin{cases}
w=-v'\\
w'\in \partial \left (\frac{1}{2}\|\cdot -f\|_{L^{2}(\Omega)}^{2}\right)(u)
\end{cases}\nonumber\\
&\Leftrightarrow&
\begin{cases}
w=-v'\\
w'=u-f
\end{cases}\label{wminusv}
\end{eqnarray}
Combining \eqref{wminusv} with \eqref{oa} and \eqref{ob} we deduce \eqref{C1}--\eqref{C3}.
\end{proof}
We note here that the corresponding optimality conditions for the $L^{1}$--TGV$_{\beta,\alpha}^{2}$ model \cite{BrediesL1} are
\begin{eqnarray*}
v''&\in& \mathrm{Sgn}(f-u),\\
-v'&\in&\alpha\,\mathrm{Sgn}(Du-w),\\
v&\in&\beta\,\mathrm{Sgn}(Dw),
\end{eqnarray*}
In the $L^{2}$--TV case with regularising parameter $\alpha$, the optimality conditions read
\begin{eqnarray*}
v'=f-u,\\
-v\in \alpha\,\mathrm{Sgn}(Du),
\end{eqnarray*}
where $v\in H_{0}^{1}(\Omega)$, see \cite{ring2000structural}.

\section{Properties of the solutions}\label{section:properties}
Before we proceed to computations of exact solutions for simple data functions, we firstly show some properties of the one dimensional $L^{2}$--TGV$_{\beta,\alpha}^{2}$ regularisation. Recall that $\Omega$ is an open interval $(a,b)$. 
The following two propositions were proved in \cite{BrediesL1} for the one dimensional $L^{1}$--TGV$_{\beta,\alpha}^{2}$ model and they hold here as well. Their proofs are minor adjustments of the corresponding proofs for the $L^{1}$ case and we omit them.

\newtheorem{u_smaller_f}[first]{Proposition}
\begin{u_smaller_f}\label{u_smaller_fprop}
Let $f\in\mathrm{BV}(\Omega)$ and suppose that $u$, $w\in \mathrm{BV}(\Omega)$ solve \eqref{P}. Suppose that $\overline{u}<\underline{f}$ on an open interval $I\subseteq (a,b)$. Then the following hold:
\begin{enumerate}
\item $(Du-w)\lfloor I=0$, that is to say $u'=w$ on $I$ and $|D^{s}u|(I)=0$.
\item $w'=0$ on $I$ and $0\le -Dw\lfloor I\ll \delta_{x}$ for some $x\in I$.
\item The function $w=u'$ is non-increasing on $I$.
\end{enumerate}
Similarly, if $\underline{u}>\overline{f}$ on $I$ then we have
\begin{enumerate}
\item $(Du-w)\lfloor I=0$, that is to say $u'=w$ on $I$ and $|D^{s}u|(I)=0$.
\item $w'=0$ on $I$ and $0\le Dw\lfloor I\ll \delta_{x}$ for some $x\in I$.
\item The function $w=u'$ is non-decreasing on $I$.
\end{enumerate}
\end{u_smaller_f}
\begin{proof}
See corresponding proof in \cite{BrediesL1}.
\end{proof}

The following proposition tells us that the jump set of the solution $u$ is contained in the jump set of the data $f$.
\newtheorem{jumpset}[first]{Proposition}
\begin{jumpset}\label{jumpsetprop}
Let $f\in \mathrm{BV}(\Omega)$ and let
\[G_{f}:=\left\{(x,t)\in\Omega\times\mathbb{R}:\;x\in J_{f}, \;t\in[\underline{f}(x),\overline{f}(x)] \right \}.\]
If $u$ solves the $L^{2}$--$\mathrm{TGV}_{\beta,\alpha}^{2}$ minimisation problem with data function $f$, then
\begin{equation}\label{inclusion_dcts}
G_{u}\subseteq G_{f}.
\end{equation}
\end{jumpset}
\begin{proof}
See corresponding proof in \cite{BrediesL1}.
\end{proof}
 
 The following proposition says that at least away from the boundary the solution $u$ can be bounded pointwise by the data $f$.
 \newtheorem{away_from_boundary}[first]{Proposition} 
\begin{away_from_boundary}[Behaviour away from the boundary]\label{awayboundary}
Let $u$ be the solution to the problem \ref{P} with data $f$ and let $(c,d)\subseteq \{x\in (a,b):\; \overline{f}(x)<\underline{u}(x)\}$ with the property that it is maximal and $a<c<d<b$. Then 
\begin{eqnarray}
\overline{u}(c)&\in&[\underline{f}(c),\overline{f}(c)],\label{u_plus_bound}\\
\overline{u}(a)&\in&[\underline{f}(d),\overline{f}(d)],\label{u_minus_bound}
\end{eqnarray}
and 
\begin{equation}\label{u_inside_bound}
\inf_{x\in (c,d)}\overline{f}(x)\le \min_{x\in [c,d]} \underline{u}(x)\le \max_{x\in [c,d]}\overline{u}(x)\le \max\{\overline{f}(c),\overline{f}(d)\}.
\end{equation}
Analogue results hold for $\overline{u}<\underline{f}$.
\end{away_from_boundary}

\begin{proof}
Let us note first that the set $\{x\in (a,b):\; \overline{f}(x)<\underline{u}(x)\}$ is open, see \cite{BrediesL1}. We show \eqref{u_plus_bound} and \eqref{u_minus_bound} can be shown in similar fashion.

 Suppose first that $\underline{f}(c)=\overline{f}(c)$. From \eqref{inclusion_dcts} we have that $\underline{u}(c)=\overline{u}(c)=u^{+}(c)$. Suppose that $\underline{u}(c)>\overline{f}(c)$ (we work similarly for the case $\underline{u}(c)<\overline{f}(c)$). In that case we claim that there exists $\delta>0$ such that for every $x\in (c-\delta,c]$ we have $\underline{u}(x)>\overline{f}(x)$ which would be a contradiction from the maximality of $(c,d)$. Indeed, if the claim is not true there exists a sequence $(x_{n})_{n\in\mathbb{N}}$, with $x_{n}<c$, such that $x_{n}\to c$ and $\underline{u}(x_{n})\le \overline{f}(x_{n})$ for every $n$. Since both $\underline{u}$ and $\overline{f}$ are continuous on $c$ that would mean that $\underline{u}(c)\le \overline{f}(c)$.

In the case where $\underline{f}(c)<\overline{f}(c)$, i.e., $c\in J_{f}$, \eqref{u_plus_bound} follows directly from \eqref{inclusion_dcts}.

In order to prove \eqref{u_inside_bound}, notice first that Proposition \ref{u_smaller_fprop} implies that $u$ is continuous and convex at $(c,d)$, thus $\overline{u}=\underline{u}$ there. From convexity we have for every $\lambda\in [0,1]$
\begin{eqnarray}
\overline{u}(\lambda c+(1-\lambda)d)&\le& \lambda \overline{u}(c)+(1-\lambda)\overline{u}(d)\nonumber\\
								&\le&\lambda \max\{\overline{f}(c),\overline{f}(d)\}+(1-\lambda) \max\{\overline{f}(c),\overline{f}(d)\}\nonumber\\
								&=& \max\{\overline{f}(c),\overline{f}(d)\},\label{right_part}
\end{eqnarray}
Finally we have that $\overline{f}(x)<\underline{u}(x)$ for all $x\in (c,d)$. From the existence of the side limits we have that $\underline{u}$ can be extended continuously to $[c,d]$ and thus 
\begin{equation}\label{left_part}
\inf_{x\in(c,d)}\overline{f}(x)\le \min_{x\in [c,d]}\underline{u}(x).
\end{equation}
Combining \eqref{right_part} and \eqref{left_part} we get \eqref{u_inside_bound}.
\end{proof}

Let us note here that the above pointwise estimates hold only away from the boundary of $\Omega$. As we will see in the following sections, there are examples where $u<\inf f$ and $u>\sup f$ in the boundary. In the following proposition we investigate further the structure of the solution $u$ near the boundary.
\newtheorem{linear_boundary}[first]{Proposition}
\begin{linear_boundary}[Behaviour near the boundary]\label{nearboundary}
The following two statements hold:
\begin{enumerate}
\item If $(a,c)$ is a maximal interval subset of $\{x\in (a,b):\; \overline{u}<\underline{f}\}$, then $\overline{u}$ is an affine function there.
\item Suppose that $u=f$ a.e. in a maximal set of a form $(a,c)$ and suppose that $\overline{u}<\underline{f}$ in a set of the form $(c,d)$. Then $u$ is linear in $(a,d)$.
\end{enumerate}
Analogue results hold for the case $\underline{u}>\overline{f}$ and also near $b$.
\end{linear_boundary}

\begin{proof}
(i) We know from Proposition \ref{u_smaller_fprop} that $u$ will be continuous and concave on $(a,c)$. We claim that $Dw=0$ and thus $w$ is constant there, and thus again from the fact that $u'=w$ we get that $u$ is linear. Indeed we have that the function $v$ is strictly convex on $(a,c)$ with $v(a)=v'(a)=0$. This means that $v$ is strictly increasing on $(a,c)$. If $Dw\lfloor I=-\delta_{x}$ for some $x\in (a,c)$ we would have that $v(x)=-\beta$ and $v(x')<-\beta$ for every $a<x'<x$ which is a contradiction.\\
(ii) Because of the fact that $u=f$ a.e. in $(a,c)$, we have that $v$ is an affine function there and since $v\in H_{0}^{2}(a,b)$ we have that $v=0$ in $(a,c)$. Condition \eqref{C3} forces $w$ to be a constant and condition \ref{C2} forces $Du=w$ and thus $u$ is an affine function in $(a,c)$ say with derivative equal to $\lambda_{1}$. Since $\overline{u}<\underline{f}$ in $(c,d)$ we have that $u$ is also an affine function there say with derivative equal to $\lambda_{2}$. Moreover $u$ must be continuous in $c$ because otherwise from condition \eqref{C2} we would have that $|v'(c)|=\alpha$ and $v'$ would be discontinuous at $c$. Also, we have that $\lambda_{1}=\lambda_{2}$. This is because of the fact that $w\lfloor(a,c)=\lambda_{1}$, $w\lfloor(c,d)=\lambda_{2}$ and if $\lambda_{1}\ne \lambda_{2}$, from condition \eqref{C3} we would have that $|v(c)|=\beta$ and thus $v$ would be discontinuous at $c$. Note finally that $u$ cannot have a gradient change in $(c,d)$ because then $v$ would have local minimum in $(c,d)$ which is again impossible.
\end{proof}

The second part of Proposition \ref{nearboundary} tells us that the case where $u=f$ near the boundary can possibly happen only if $f$ is linear there.

\subsection{$L^{2}$--linear regression}
We continue with some results concerning the $L^{2}$--linear regression of the data $f$.
In \cite{BrediesL1}, it was proved that in the one dimensional $L^{1}-$TGV$_{\beta,\alpha}^{2}$ case there exist thresholds $\alpha^{\ast}>0$, $\beta^{\ast}>0$ such that if $\alpha\ge \alpha^{\ast}$ and $\beta>\beta^{\ast}$ then the solution $u$ of \eqref{P} is the $L^{1}$--linear regression $f^{\ast}$ of $f$, where
\[f^{\ast}=\underset{v \text{ affine}}{\operatorname{argmin}}\,\|f-v\|_{L^{1}(\Omega)}.\] 
There, the values of $\alpha^{\ast}$ and $\beta^{\ast}$ are independent of the function $f$ and depend only on the size of the domain $\Omega$. Here we show that these thresholds exist for the $L^{2}$ case as well, but they depend on the function $f$. 

For a function $f\in\mathrm{BV}(\Omega)$, we define $f^{\star}$ to be the $L^{2}$--linear regression, i.e.,
\begin{equation}\label{fstar}
f^{\star}=\underset{v \text{ affine}}{\operatorname{argmin}}\, \|f-v\|_{L^2(\Omega)}^{2}.
\end{equation}
We have the following proposition:
\newtheorem{L2regression}[first]{Theorem}
\begin{L2regression}[Thresholds for $L^{2}$--linear regression]\label{l2regression_thm}
Let $u$ be the solution to \eqref{P} with data function $f$.
Then $u$ is an affine solution to \eqref{P} if and only if $u=f^{\star}$. Moreover, if
\begin{eqnarray}
\alpha&\ge&\frac{L}{2}\|f-f^{\star}\|_{\infty},\label{regression1}\\
\beta &\ge&\frac{L^{2}}{4}\|f-f^{\star}\|_{\infty},\label{regression2}
\end{eqnarray}
we have that the solution $u$ of \eqref{P} is given by $f^{\star}$.
\end{L2regression}

\begin{proof}
We first show that if $u$ is an affine function and solves \eqref{P}, then $u=f^{\star}$. Since $u$ is affine we have that $D^{2}u=0$, thus the optimum $w$ in \eqref{P} is $w=Du$. Then,  we obviously have
\[u=\underset{\tilde{u} \text{ affine}}{\operatorname{argmin}}\;\frac{1}{2}\|f-\tilde{u}\|_{L^{2}(\Omega)}^{2},\]
 i.e, $u=f^{\star}$. For the second part, since $Df^{\star}-w=0$ and $Dw=0$, in order for the solution to be  equal to $f^{\star}$ it suffices to find a function $v\in H_{0}^{2}(\Omega)$, such that
 \begin{equation}
v''=f-f^{\star},\quad  \|v\|_{\infty}\le \beta \quad \mathrm{and} \quad \|v'\|_{\infty}\le \alpha.
\end{equation}
It is easy to see that for every function $v\in W^{1,\infty}(\Omega)$ with $v(a)=v(b)=0$ (recall that $v$ is necessarily continuous) we have the estimate
\[\|v\|_{\infty}\le \frac{L}{2}\|v'\|_{\infty}.\]
Hence, a function $v\in H_{0}^{2}(\Omega)$ that satisfies $v''=f-f^{\star}$, it  satisfies
\begin{eqnarray}\
\|v'\|_{\infty}&\le& \frac{L}{2}\|f-f^{\star}\|_{\infty},\label{boundvprime}\\
\|v\|_{\infty}&\le&\frac{L^{2}}{4}\|f-f^{\star}\|_{\infty}   \label{boundv}.
\end{eqnarray}
Thus, in order to have this kind of solution, it suffices to choose $\alpha$ and $\beta$ as in \eqref{regression1}-\eqref{regression2}.
\end{proof}
The next proposition tells us that the solution to the $L^{2}$--$\mathrm{TGV}_{\beta,\alpha}^{2}$ problem, has the same $L^{2}$--linear regression with the data function.

\newtheorem{sameregression}[first]{Proposition}
\begin{sameregression}
Let $u_{f,\beta,\alpha}$ be the solution to $L^{2}$--$\mathrm{TGV}_{\beta,\alpha}^{2}$ minimisation problem
 with data $f$. Then $f^{\star}=u_{f,\beta,\alpha}^{\star}$, i.e.,
 \[\underset{v \text{ affine}}{\operatorname{argmin}}\,\|f-v\|_{L^{2}(\Omega)}^{2}=\underset{v \text{ affine}}{\operatorname{argmin}}\,\|u_{f,\beta,\alpha}-v\|_{L^{2}(\Omega)}^{2}. \]
\end{sameregression}
\begin{proof}
Without loss of generality we can assume that $\Omega=(a,b)$ is an interval of the form $(0,L)$. We have that $v''=f-u$ with $v\in H_{0}^{2}(\Omega)$. Thus, we have that $\int_{\Omega}v''~dx=0$ and $\int_{\Omega}x v''~dx=0$ which implies
\[\int_{\Omega}f~dx=\int_{\Omega}u_{f,\beta,\alpha}~dx\quad \text{and}\quad \int_{\Omega}x f~dx=\int_{\Omega}x u_{f,\beta,\alpha}~dx.\]
Then the proof is finished by simply observing that for a function $g$ we have
\begin{eqnarray*}
\underset{v \text{ affine}}{\operatorname{argmin}}\,\|g-v\|_{L^{2}(\Omega)}^{2}&=&\underset{v=\lambda x+\mu}{\operatorname{argmin}}\,\|g-\lambda x-\mu\|_{L^{2}(\Omega)}^{2}\\
&=&\underset{v=\lambda x+\mu}{\operatorname{argmin}}\,\frac{1}{2}\lambda^{2}L^{2}+\lambda L^{2}b+L\mu^{2}-2\lambda\int_{\Omega}xg~dx-2\mu\int_{\Omega}g~dx.
\end{eqnarray*}

\end{proof}

\subsection{Even and odd functions}
Before we proceed to the computation of exact solutions for some simple examples, we firstly point out some facts concerning odd and even data. For this section, we assume that $\Omega$ is an interval of the type $(-L,L)$.
\newtheorem{symmetrypredual}[first]{Proposition}
\begin{symmetrypredual}\label{symmetrypredual_prop}
Let $f\in \mathrm{BV}(-L,L)$ be an odd (even) function. Then the solution $v\in H_{0}^{2}(-L,L)$ to the problem \eqref{Pprime} is also odd (even).
\end{symmetrypredual}
\begin{proof}
We prove the odd case, the even case is proved analogously. Suppose that $f(x)=-f(-x)$ a.e.. We will show that $v(x)=-v(-x)$.
Let $\hat{f}(x)=-f(-x)$ and $\hat{v}$ to be the solution to \ref{Pprime} with data function $\hat{f}$ and set $\tilde{v}(x)=-v(-x)$. 
Obviously 
\[\|v\|_{\infty}\le \beta,\; \|v'\|_{\infty}\le \alpha \iff \|\tilde{v}\|_{\infty}\le \beta,\; \|\tilde{v}'\|_{\infty}\le \alpha.\]
We have
\begin{eqnarray*}
\int_{-L}^{L}\hat{f}(x)\tilde{v}(x)~dx-\frac{1}{2}\int_{-L}^{L}(\tilde{v}''(x))^{2}~dx&\le& \int_{-L}^{L}\hat{f}(x)\hat{v}(x)~dx-\frac{1}{2}\int_{-L}^{L}(\hat{v}''(x))^{2}~dx \Rightarrow\\
\int_{-L}^{L}f(x)v(x)~dx-\frac{1}{2}\int_{-L}^{L}v(x)^{2}~dx&\le&\int_{-L}^{L}f(x)(-\hat{v}(-x))~dx-\frac{1}{2}\int_{-L}^{L}((-\hat{v}(-x))'')^{2}~dx
\end{eqnarray*}
thus $v(x)=-\hat{v}(-x)$ but since $\hat{f}=f$ we have that $\hat{v}=v$.
\end{proof}

\newtheorem{symmetrydual}[first]{Proposition}
\begin{symmetrydual}\label{symmetrydual_label}
Let $f\in \mathrm{BV}(-L,L)$ be an odd (even) function. Then the problem \eqref{P} has a solution $u$ that is also odd (even).
\end{symmetrydual}

\begin{proof}
Again we only prove the odd case as the even case can be proved in a similar fashion. Set $\hat{f}(x)=-f(-x)$, $\hat{u}$ the solution of \eqref{P} with data function $\hat{f}$ and set $\tilde{u}(x)=-u(-x)$. One can notice easily that $\mathrm{TGV}_{\beta,\alpha}^{2}(u)=\mathrm{TGV}_{\beta.\alpha}^{2}(\tilde{u})$. We have
 \begin{eqnarray*}
 \frac{1}{2}\int_{-L}^{L}(\hat{u}(x)-\hat{f}(x))^{2}dx+\text{TGV}_{\beta,\alpha}^{2}(\hat{u})&\le&\frac{1}{2}\int_{-L}^{L}(\tilde{u}(x)-\hat{f}(x))^{2}dx+\text{TGV}_{\beta,\alpha}^{2}(\tilde{u})\Rightarrow\\
 \frac{1}{2}\int_{-L}^{L}(-\hat{u}(-x)-f(x))^{2}dx+\text{TGV}_{\beta,\alpha}^{2}(-\hat{u}(-\cdot))&\le&\frac{1}{2}\int_{-L}^{L}(u(x)-f(x))^{2}dx+\text{TGV}_{\beta,\alpha}^{2}(u)
 \end{eqnarray*}
Thus we have $u(x)=-\hat{u}(-x)$. Since $\hat{f}=f$ we have $\hat{u}=u$ and thus $u(x)=-u(-x)$.
\end{proof}

Finally, the following theorem tells us that at least for the case of even data functions $f$, if the ratio $\beta/\alpha$ of the TGV$_{\beta,\alpha}^{2}$ parameters is large enough then TGV regularisation is equivalent to TV regularisation with parameter $\alpha$. We need firstly the following Lemma:

\newtheorem{even_plus_linear}[first]{Lemma}
\begin{even_plus_linear}\label{Lemma_even_plus_linear}
Let $u\in \mathrm{BV}(-L,L)$ be an even function. Then for every $c\in\mathbb{R}$ we have
\begin{equation}
\|Du\|_{\mathcal{M}}\le \|Du+c\|_{\mathcal{M}}.
\end{equation}
\end{even_plus_linear}

\begin{proof}
We can assume that $c>0$.
It suffices to show that 
\begin{equation}\label{abs_part_even}
\int\limits_{(-L,L)}|u'|~dx\le \int\limits_{(-L,L)}|u'+c|~dx.
\end{equation}
Since $u$ is even we have that $u'$ is even and thus up to null sets we have
\begin{eqnarray*}
\{u'>0\}\cap(-L,0)&=&-\{u'<0\}\cap (0,L),\\
\{u'<0\}\cap(-L,0)&=&-\{u'>0\}\cap (0,L).
\end{eqnarray*}
Then we estimate
\begin{eqnarray*}
\hspace{-0.4 cm}\int\limits_{(-L,L)}|u'|~dx&=&\int\limits_{\{u'>0\}\cap (-L,0)}\hspace{-0.4cm}u' ~dx \;\;-\hspace{-0.3cm}\int\limits_{\{u'<0\}\cap(0,L)}\hspace{-0.4cm} u'~dx\;\;+\hspace{-0.3cm}\int\limits_{\{u'>0\}\cap (0,L)}\hspace{-0.4cm}u' ~dx\;\; -\hspace{-0.3cm}\int\limits_{\{u'<0\}\cap(-L,0)}\hspace{-0.4cm} u'~dx\\
&=&\int\limits_{\{u'>0\}\cap (-L,0)}\hspace{-0.4cm}c+u' ~dx \;\;-\hspace{-0.3cm}\int\limits_{\{u'<0\}\cap(0,L)}\hspace{-0.4cm} c+u'~dx\;\;+\hspace{-0.3cm}\int\limits_{\{u'>0\}\cap (0,L)}\hspace{-0.4cm}c+u' ~dx \;\;-\hspace{-0.3cm}\int\limits_{\{u'<0\}\cap(-L,0)} \hspace{-0.4cm}c+u'~dx\\
&\le&\int\limits_{\{u'>0\}\cap (-L,0)}\hspace{-0.4cm}c+u' ~dx \;\;+\hspace{-0.3cm}\int\limits_{\{u'<0\}\cap(0,L)}\hspace{-0.4cm} |c+u'|~dx\;\;+\hspace{-0.3cm}\int\limits_{\{u'>0\}\cap (0,L)}\hspace{-0.4cm}c+u' ~dx \;\;+\hspace{-0.3cm}\int\limits_{\{u'<0\}\cap(-L,0)} \hspace{-0.4cm}|c+u'|~dx+\hspace{-0.3cm}\int\limits_{\{u'=0\}\cap (-L,L)}\hspace{-0.4cm}|c|~dx \\
&=&\int\limits_{(-L,L)}|u'+c|~dx.
\end{eqnarray*}
\end{proof}

\begin{figure}[h!]
\begin{center}
\hspace{-0.5 cm}
\includegraphics[scale=0.40]{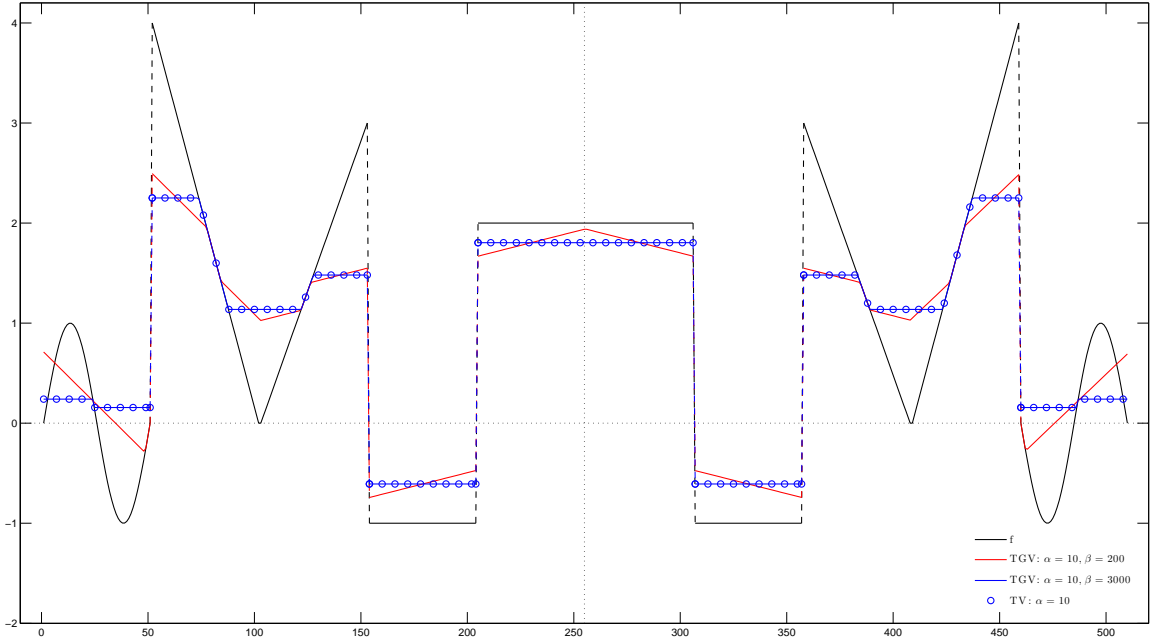}
\caption{ Numerical verification of Theorem \ref{evenTV_label}. Using the even function $f$ shown above as a data function and setting $\beta/\alpha\ge L$, the  solutions of the $\mathrm{TGV}_{\beta,\alpha}^{2}$ and $\alpha\,\mathrm{TV}$ regularisation problems coincide.}
\label{evenplot}
\end{center}
\end{figure}

\newtheorem{evenTV}[first]{Theorem}
\begin{evenTV}\label{evenTV_label}
Let $f\in \mathrm{BV}(-L,L)$ be an even function. If 
\begin{equation}\label{a_b_even}
\frac{\beta}{\alpha}\ge L,
\end{equation}
then the solution to the $L^{2}$--$\mathrm{TGV}_{\beta,\alpha}^{2}$ regularisation problem is the same with the solution to the following $\mathrm{TV}$ regularisation problem:
\[\min_{u}\;\frac{1}{2}\int_{\Omega}(u-f)^{2}dx+\alpha\,\mathrm{TV}(u).\]
\end{evenTV}

\begin{proof}
Recall that for every $u\in W^{1,1}(\Omega)$ the following Poincar\'{e} type inequality holds,
\[\|u-u_{\Omega}\|_{L^{1}(\Omega)}\le L \|\nabla u\|_{L^{1}(\Omega)},\]
where $u_{\Omega}=\frac{1}{\mathcal{L}(\Omega)}\int_{\Omega}u ~dx$, see for example \cite{AmbrosioBV}. Using density arguments,  we can prove that the same inequality holds for BV with the same constant, i.e., for every $u\in \mathrm{BV}(\Omega)$
\begin{equation}\label{BVpoincare}
\|u-u_{\Omega}\|_{L^{1}(\Omega)}\le L \|Du\|_{\mathcal{M}}.
\end{equation}
Since $f$ is even, from Proposition \ref{symmetrydual_label} we have that the solution of 
\begin{equation}\label{TGV_before_even}
\min_{u\in\mathrm{BV}(\Omega)}\; \frac{1}{2}\int_{\Omega}(u-f)^{2}dx+\mathrm{TGV}_{\beta,\alpha}^{2}(u),
\end{equation}
is an even function as well. Thus, the problem \eqref{TGV_before_even} is equivalent to
\begin{equation}\label{TGV_after_even}
\min_{\substack{u\in\mathrm{BV}(\Omega)\\ u \text{ even}}}\; \frac{1}{2}\int_{\Omega}(u-f)^{2}dx+\mathrm{TGV}_{\beta,\alpha}^{2}(u).
\end{equation}
Similarly we can prove that the outcome of TV regularisation for even data, is even as well.
Thus, it suffices to prove that for an even function $u$, we have $\mathrm{TGV}_{\beta,\alpha}^{2}(u)=\alpha\,\mathrm{TV}(u):=\alpha \|Du\|_{\mathcal{M}}$, provided \eqref{a_b_even} holds.
We calculate successively:
\begin{align*}
\mathrm{TGV}_{\beta,\alpha}^{2}(u)&=\min_{w\in\mathrm{BV}(\Omega)}\; \alpha \|Du-w\|_{M}+\beta\|Dw\|_{\mathcal{M}}&\\
								 &\le\alpha \|Du\|_{\mathcal{M}}  		&\text{(choosing }w=0),\\
								 &\le \alpha \|Du-w_{\Omega}\|_{\mathcal{M}}&\forall w\in\mathrm{BV}(\Omega), (\text{from Lemma \ref{Lemma_even_plus_linear}}),\\
								 &\le\alpha \|Du-w\|_{\mathcal{M}} +\alpha\|w-w_{\Omega}\|_{L^{1}(\Omega)}&\\
								 &\le\alpha \|Du-w\|_{\mathcal{M}} +\frac{\beta \alpha L}{\beta}\|Dw\|_{\mathcal{M}} &\text{(Poincar\'{e} inequality),}\\
								 &\le\max\left \{1,\frac{aL}{\beta} \right \} \left (\alpha \|Du-w\|_{\mathcal{M}}+\beta \|Dw\|_{\mathcal{M}} \right)&\\
								 &=\alpha \|Du-w\|_{\mathcal{M}}+\beta \|Dw\|_{\mathcal{M}} &(\text{since \eqref{a_b_even} holds}).
\end{align*}
Finally, we have
\begin{align*}
\alpha \|Du\|_{\mathcal{M}}&\le \alpha \|Du-w\|_{\mathcal{M}}+\beta \|Dw\|_{\mathcal{M}}\quad \forall w\in\mathrm{BV}(\Omega) \Rightarrow\\
\alpha \|Du\|_{\mathcal{M}}&\le\min_{w\in\mathrm{BV}(\Omega) }\alpha \|Du-w\|_{\mathcal{M}}+\beta \|Dw\|_{\mathcal{M}}\\						&=\mathrm{TGV}_{\beta,\alpha}^{2}(u), 
\end{align*}
and the proof is complete.
\end{proof}

In Figure \ref{evenplot} we verify numerically the above result. Both TGV and TV minimisation problems are solved with the Chambolle-Pock primal dual method \cite{chambolle2011first} as it is described in \cite{tgvcolour}.

\section{Computation of exact solutions}\label{section:computation}
In this section we compute exact solutions for the $\mathrm{TGV}_{\beta,\alpha}^{2}$ regularisation problem for several data functions $f$. In particular we calculate exact solutions for simple piecewise constant, piecewise affine and hat functions. We are focusing on the relation between the structure of solutions and the parameters $\alpha$ and $\beta$.

\subsection{Piecewise constant function with a single jump}
For convenience in the calculation, for this section, $\Omega$ will be an interval of the form $(0,2L)$. We define $f$ to be a piecewise function with a single jump, i.e.,
\begin{equation}\label{f_1}
f(x)=
\begin{cases}
0 & \text{if }\;x\in (0,L),\\
h& \text{if }\;x\in [L,2L),
\end{cases}
\end{equation}
for some $h>0$, see Figure \ref{f_jump}.

\begin{figure}[h!]
\begin{center}
\includegraphics[scale=0.40]{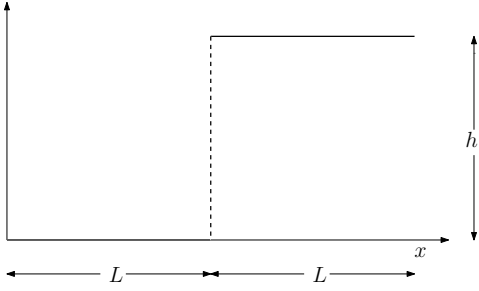}
\caption{Piecewise constant function $f$ with a jump discontinuity at $x=L$. }
\label{f_jump}
\end{center}
\end{figure}

In order to compute exact solutions for $f$ our strategy is as  follows: Firstly, we investigate what are the possible types of solutions, describing them in a qualitative way and secondly we calculate them explicitly.

Since after a simple translation $f$ is an even function and $\mathrm{TGV}_{\beta,\alpha}^{2}$ is translation invariant \cite{TGV}, from Proposition \ref{symmetrydual_label} we have that the solution $u$ will be symmetric with respect to the point $(L,h/2)$. That is to say
\begin{equation}
u(x)=-u(2L-x)+h,\quad \text{for all } x\in (L,2L).
\end{equation}
Thus, we only need to describe the solution in $(0,L)$. The following theorem states which kinds of solutions are allowed to occur.
\newtheorem{f_allowed_qualitative}[first]{Theorem}
\begin{f_allowed_qualitative}\label{f_allowed_qualitative_thm}
Let  $u$ be the solution to the $L^{2}$--$\mathrm{TGV}_{\beta,\alpha}^{2}$ regularisation problem
for the piecewise constant function defined in \eqref{f_1}. Then $u$ can only have the following form on $(0,L)$:
\begin{itemize}
\item There exists $0<x_{1}<L$ such that $u$ is strictly negative and affine in $(0,x_{1})$ with $u(x_{1})=0$. 
\item There exists potentially $0<x_{1}\le\tilde{x}_{1}<L$ such that $u=0$ in $[x_{1},\tilde{x}_{1}]$.
\item $u>0$  in $(\tilde{x}_{1},L)$ consisting of at most two affine parts of increasing gradient in $(\tilde{x}_{1},x_{2})$ and $(x_{2},L)$ respectively with $\tilde{x}_{1}<x_{2}<L$. Moreover, in the case $\tilde{x}_{1}=x_{1}$, $u$ has the same gradient in $(0,x_{1})$ and $(x_{1},x_{2})$. 
\end{itemize}
\end{f_allowed_qualitative}

\newtheorem*{note1}{Note}
\begin{note1}
\emph{In fact, as we will see later, $\tilde{x}_{1}$ must be equal to $x_{1}$, i.e., the allowed solutions will be strictly negative in $(0,x_{1})$ and strictly positive in $(x_{1},L)$. However this is not obvious by simply interpreting the optimality conditions \ref{C1}, \ref{C2} and \ref{C3} (which is how Theorem \ref{f_allowed_qualitative_thm} is proved) but it is a result of subsequent calculations.}
\end{note1}

\begin{proof}[Proof of Theorem \ref{f_allowed_qualitative_thm}]
Let us note firstly that from Propositions \ref{jumpsetprop} and \ref{awayboundary} we have that $u$ is continuous on $(0,L)$ with $u^{-}(L)\ge 0$. Also, from Proposition \ref{nearboundary} we have that if $u<0$  on a set of the form $(0,x_{1})$ then $u$ must be affine there. Moreover Proposition \ref{u_smaller_fprop} implies that there is not an interval $(c,d)\subseteq (0,L)$ such that $u<0$  ($u>0$) on $(c,d)$ with $u^{+}(c)=u^{-}(d)=0$. Indeed, for the negative case, there would be a point $x\in (c,d)$ such that $u$ is affine on $(c,x)$ and $(x,d)$ with strictly negative and strictly positive gradient respectively on these intervals. But then $u'$ would be strictly increasing in an interval where $\overline{u}<\underline{f}$ a contradiction. We work similarly for the positive case. We conclude that $u$ can be strictly negative only in an interval of the type $(0,c)$, $0\le c\le L$ and strictly positive only in an interval of the type $(d,L)$, $0\le d\le L$ possibly consisting of two affine parts of increasing gradient. In particular $u$ would be increasing in $(0,L)$. The proof will be complete if we prove that the following situations cannot happen:
\begin{enumerate}
\item $u>0$ in $(0,L)$.
\item $u=0$ in $(0,x_{1})$, $0<x_{1}<L$ and $u>0$ in $(x_{1},L)$.
\item $u=0$ in $(0,L)$.
\item $u<0$ in $(0,x_{1})$, $0<x_{1}<L$ and $u=0$ in $x_{1},L$.
\item $u<0$ in $(0,L)$. 
\end{enumerate}

For (i) suppose that $u>0$ in $(0,L)$. By symmetry we would have that $u<h$ in $(L,2L)$. This means that the dual function $v$ would be strictly convex in $(0,L)$, strictly concave in $(L,2L)$ with $v^{+}(0)=v'^{+}(0)=v^{-}(2L)=v'^{-}(2L)=0$ something that cannot happen. 

For (ii) suppose that there exists $0<x_{1}<L$ such that $u=0$ in $(0,x_{1})$ and $u>0$ in $(x_{1},L)$. Then $v$ will be an affine function in $(0,x_{1})$ (condition \eqref{C1}) but since $v\in H_{0}^{2}(0,2L)$ we have that $v=0$ in $(0,x_{1})$. Moreover $Du=0$  and thus $w=0$ a.e. in $(0,x_{1})$ because otherwise from condition \eqref{C2} we would have that $|v'|=\alpha$ on a set of positive measure in $(0,x_{1})$. Now, since $u>0$ in $(x_{1},L)$ (assume w.l.o.g. that $u$ is affine there) from Proposition \ref{u_smaller_fprop} we have that $w=Du=c>0$ something that forces $w$ to have a jump discontinuity at $x_{1}$ and thus from \eqref{C3} we get that $v(x_{1})=\beta>0$. However, this contradicts to the continuity of $v$.

For (iii) suppose that $u=0$ in $(0,L)$. Then by symmetry we would have that $u=h$ in $(L,2L)$. Thus, from condition \eqref{C1} again, $v$ will be affine  in those intervals and again from its boundary conditions and its continuity we get that $v=0$ on $(0,L)$. But since $u$ has a jump discontinuity at $x=L$, according to condition \eqref{C2} we must have $v'=-\alpha$, a contradiction.

For (iv) suppose that there exists $0<x_{1}<L$ such that $u<0$ in $(0,x_{1})$ and $u=0$ in $(x_{1},L)$. Then $v$ would be strictly convex and thus have a strictly increasing derivative in $(0,x_{1})$. Since $v'^{+}(0)=0$ we have that $v'(x_{1})>0$. Moreover,  $v$ would be affine in $(x_{1},L)$ with positive gradient equal to $v'(x_{1})$, something that forces $v(L)>0$. However, as after a simple translation $f$ is an odd function, and according to Proposition  \ref{symmetrypredual_prop} the same will be true for the continuous function $v$, then we must have  $v(L)=0$. The case (v) is proved in a similar way.

Finally, for the last statement of the proposition, suppose that $u<0$ in $(0,x_{1})$, $u>0$ in $(x_{1},L)$ and the derivative of $u$ has a jump discontinuity at $x_{1}$. Since $Du=w$, in $(0,x_{1})$ and $(x_{1},L)$ we have that $w$ will have a jump discontinuity at $x_{1}$ and condition \eqref{C3} imposes that $|v(x_{1})|=\beta$. The case $v(x_{1})=-\beta$ is impossible as $v$ is strictly increasing in $(0,x_{1})$ and also if $v(x_{1})=\beta$, from the fact that $v'(x_{1})>0$ and $v'$ is continuous as $x_{1}$ we have that $v$ is increasing in $(x_{1},x_{1}+\delta)$ for some $\delta>0$, thus $v(x_{1}+\delta)>\beta$, contradicting \eqref{C3}.
\end{proof}

Notice how we take  advantage of the symmetries of $u$ and $v$  in order to prove Theorem \ref{f_allowed_qualitative_thm}. It remains to rule out the possibility that $x_{1}<\tilde{x}_{1}$. In particular, we have to show that the four following situations cannot occur \footnote{N: negative, 0: zero, P1: positive one affine part, P2: positive two affine parts, J: jump discontinuity, C: continuous.}:
\begin{align}\label{N0P1J}
\tag{N-0-P1-J}
&u<0 \text{ in } (0,x_{1}),\; u=0 \text{ in } (x_{1},\tilde{x}_{1}),\; u>0 \text{ with 1 affine part in } (\tilde{x}_{1},L), \text{ jump at } x=L,\\
\label{N0P2J}\tag{N-0-P2-J}
&u<0 \text{ in } (0,x_{1}),\; u=0 \text{ in } (x_{1},\tilde{x}_{1}),\; u>0 \text{ with 2 affine parts in } \text{ in } (\tilde{x}_{1},L), \text{ jump at } x=L,\\
\label{N0P1C}
\tag{N-0-P1-C}
&u<0 \text{ in } (0,x_{1}),\; u=0 \text{ in } (x_{1},\tilde{x}_{1}),\; u>0 \text{ with 1 affine part in } (\tilde{x}_{1},L), \text{ continuous at } x=L,\\
\label{N0P2C}
\tag{N-0-P2-C}
&u<0 \text{ in } (0,x_{1}),\; u=0 \text{ in } (x_{1},\tilde{x}_{1}),\; u>0 \text{ with 2 affine parts in } (\tilde{x}_{1},L), \text{ continuous at } x=L,
\end{align}
where $0<x_{1}<\tilde{x}_{1}<L$.
In the following we will only show that for \eqref{N0P1J}, the proof of the rest three cases is done similarly.

\newtheorem{elimination1}[first]{Proposition}
\begin{elimination1}
The case \eqref{N0P1J} cannot occur.
\end{elimination1}

\begin{proof}
Suppose that $u<0$ in $(0,x_{1})$, $u=0$ in  $(x_{1},\tilde{x}_{1})$ and $u>0$ and affine in $(\tilde{x}_{1},L)$ with $0<x_{1}<\tilde{x}_{1}<L$. Then we claim that in that case, $u'$ is constant, say equal to $c>0$, in $(0,x_{1})\cup(\tilde{x}_{1},L)$, $w=c$ in $(0,L)$ and $v$ is affine in $(x_{1},\tilde{x}_{1})$ with gradient equal to $\alpha$ there. Indeed, let $c$ be the value of the derivative of $u$ in $(0,x_{1})$. We know already from Proposition \ref{u_smaller_fprop} that $w=u'$ there. Now, if $w\ne c$ on a set of positive measure in $(x_{1},\tilde{x}_{1})$, that will mean that $Dw\ne 0$ somewhere in $(x_{1},\tilde{x}_{1})$ and thus from \eqref{C3}, $v$ must be $\pm \beta$ somewhere in $(x_{1},\tilde{x}_{1})$. But this is not possible since $v$ is strictly convex in $(0,x_{1})$, $v'^{+}(0)=0$, affine in $(x_{1},\tilde{x}_{1})$ and thus $v$ is strictly increasing in $(0,\tilde{x}_{1})$. Moreover, $w=v'=c$ in $(\tilde{x}_{1},L)$ because otherwise $w$ would have a jump discontinuity on $\tilde{x}_{1}$, something that forces $v(\tilde{x}_{1})=\beta$ but this cannot happen since $v'(\tilde{x}_{1})>0$. Finally, since $Du-w=-c$ in $(x_{1},\tilde{x}_{1})$, from condition \eqref{C2} we have that $v'=\alpha$ there. 

Also note that since $u$ has a jump discontinuity at $x=L$, condition \eqref{C2} forces $v'(L)=-\alpha$ and from the symmetry of $v$ we must have $v(L)=0$. 

For convenience we set 
\[\ell_{1}=x_{1}>0,\quad \ell_{2}=\tilde{x}_{1}-x_{1}>0, \quad \ell_{3}=L-\tilde{x}_{1}>0 \quad \text{with }\ell_{1}+\ell_{2}+\ell_{3}=L.\]
and 
\[v_{1}=v\lfloor(0,x_{1}),\quad v_{2}=v\lfloor(x_{1},\tilde{x}_{1}),\quad v_{3}=v\lfloor(\tilde{x}_{1},L).\]
Since $u$ is an affine function on $(0,x_{1})$ we have that $v_{1}$ is a cubic function
\[v_{1}(x)=a_{1}x^{3}+b_{1}x^{2}+c_{1}x+d_{1},\quad x\in (0,x_{1}),\]
that satisfies the conditions 
\begin{equation}\label{v1_cond}
v_{1}(0)=0,\quad v_{1}'(0)=0,\quad v_{1}'(\ell_{1})=\alpha,\quad v_{1}''(\ell_{1})=0.
\end{equation}
Thus $v_{1}$ will  have the form
\begin{equation}\label{v1}
v_{1}(x)=-\frac{\alpha}{3\ell_{1}^{2}}x^{3}+\frac{\alpha}{\ell_{1}}x^{2},\quad x\in (0,x_{1}).
\end{equation}
Since $v_{2}$ is an affine function on $(x_{1},x_{2})$ with gradient $\alpha$ and $v$ is continuous at $x_{1}$, we have that $v_{2}$ will be of the form
\begin{equation}\label{v2}
v_{2}(x)=\alpha(x-\ell_{1})+\frac{2\alpha \ell_{1}}{3},\quad x\in (x_{1},\tilde{x}_{1}).
\end{equation}
Since $u$ is a strictly positive affine function on $(\tilde{x}_{1},L)$ (with the same gradient as in $(0,x_{1})$), we have that $v_{3}$ must be a cubic function
\[v_{3}(x)=a_{3}x^{3}+b_{3}x^{2}+c_{3}x+d_{3},\quad x\in (\tilde{x}_{1},L),\]
that satisfies the following conditions at $x=\tilde{x}_{1}=\ell_{1}+\ell_{2}$:
\begin{equation}\label{v3_cond1}
v_{3}(\ell_{1}+\ell_{2})=\alpha\ell_{2}+\frac{2\alpha\ell_{1}}{3},\quad v_{3}'(\ell_{1}+\ell_{2})=\alpha,\quad v_{3}''(\ell_{1}+\ell_{2})=0,\quad v_{3}'''(\ell_{1})(\ell_{1}+\ell_{2})=-\frac{2\alpha}{\ell_{1}^{2}},
\end{equation}
and also the following conditions at $x=L=\ell_{1}+\ell_{2}+\ell_{3}$:
\begin{equation}\label{v3_cond2}
v_{3}(\ell_{1}+\ell_{2}+\ell_{3})=0,
\end{equation}
\begin{equation}\label{v3_cond3}
 v_{3}'(\ell_{1}+\ell_{2}+\ell_{3})=-\alpha. 
\end{equation}
The conditions \eqref{v3_cond1} give
\begin{equation}\label{v1_1}
v_{3}(x)=-\frac{\alpha}{3\ell_{1}^{2}}(x-\ell_{1}-\ell_{2})^{3}+\alpha(x-\ell_{1}-\ell_{2})+\alpha\ell_{2}+\frac{2\alpha\ell_{1}}{3}, \quad x\in (\tilde{x}_{1},L).
\end{equation}
What is left is to find the relationship among $\ell_{1}$, $\ell_{2}$ and $\ell_{3}$.
Condition \eqref{v3_cond3} gives
\begin{equation}\label{el3el1}
\ell_{3}=\sqrt{2}\ell_{1},
\end{equation}
and the condition \eqref{v3_cond2} together with \eqref{el3el1} gives
\begin{equation}\label{el2el1}
\frac{2+\sqrt{2}}{3}\ell_{1}+\ell_{2}=0,
\end{equation}
which is a contradiction since both $\ell_{1}$ and $\ell_{2}$ are strictly positive. Thus, this kind of solution cannot occur.
\end{proof}

From all the previous results, it follows that only the following situations can occur:
\begin{align}
\label{NP1J}\tag{N-P1-J}
&u<0 \text{ in } (0,x_{1}),\; u>0 \text{ with 1 affine part in } (x_{1},L), \text{ jump at } x=L,\\
\label{NP2J}\tag{N-P2-J}
&u<0 \text{ in } (0,x_{1}),\; u>0 \text{ with 2 affine parts in } (x_{1},L), \text{ jump at } x=L,\\
\label{NP1C}\tag{N-P1-C}
&u<0 \text{ in } (0,x_{1}),\; u>0 \text{ with 1 affine part in } (x_{1},L), \text{ continuous at } x=L,\\
\label{NP2C}\tag{N-P2-C}
&u<0 \text{ in } (0,x_{1}),\; u>0 \text{ with 2 affine parts in } (x_{1},L), \text{ continuous at } x=L,
\end{align}
where $0<x_{1}<L$.

\begin{figure}
\begin{center}
\subfloat[Solution of the type \newline\ref{NP1J}.]
{
\hspace{-0.6 cm}\includegraphics[scale=0.25]{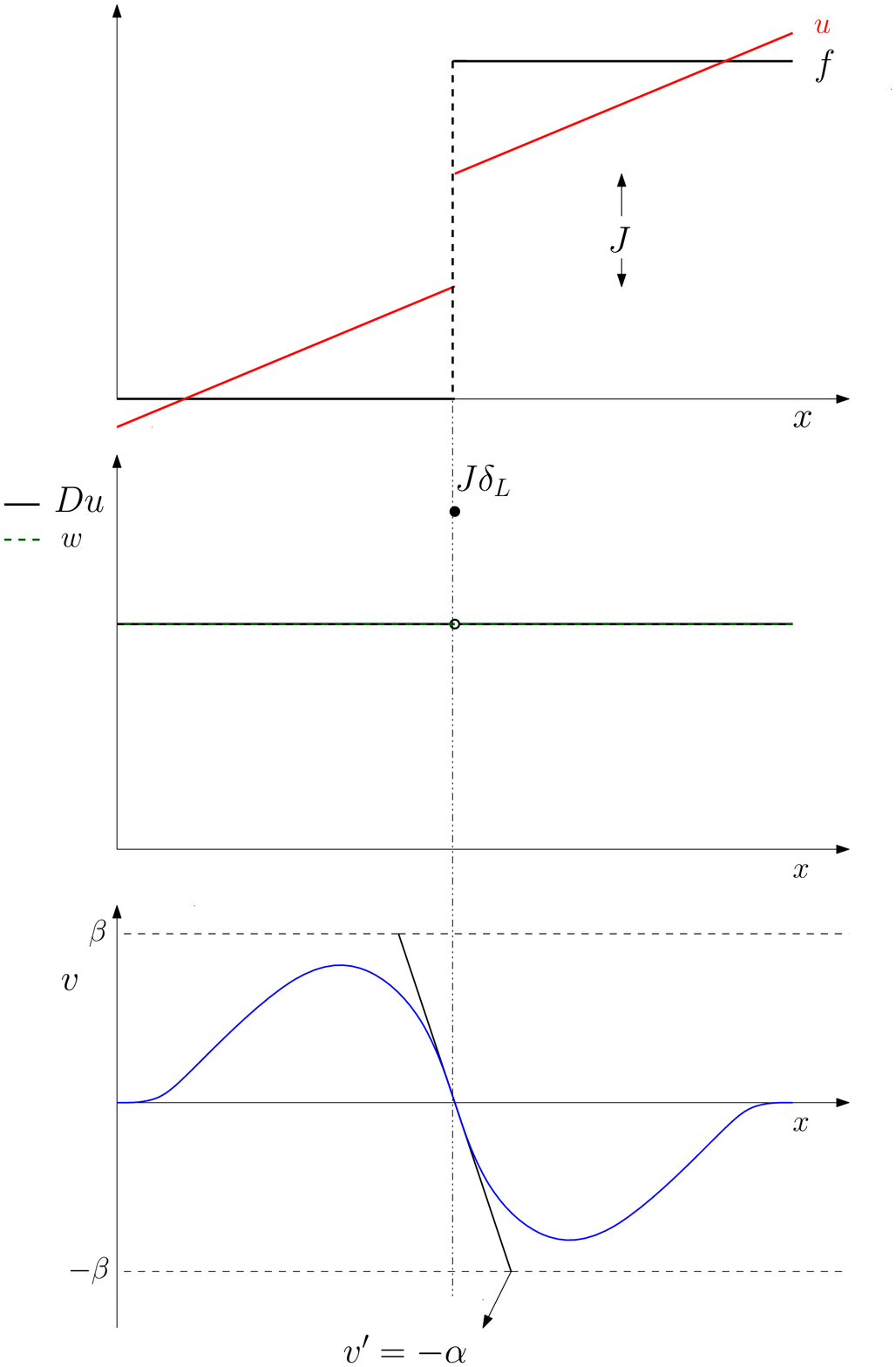}
}
\subfloat[Solution of the type \newline\ref{NP2J}.]{
\includegraphics[scale=0.25]{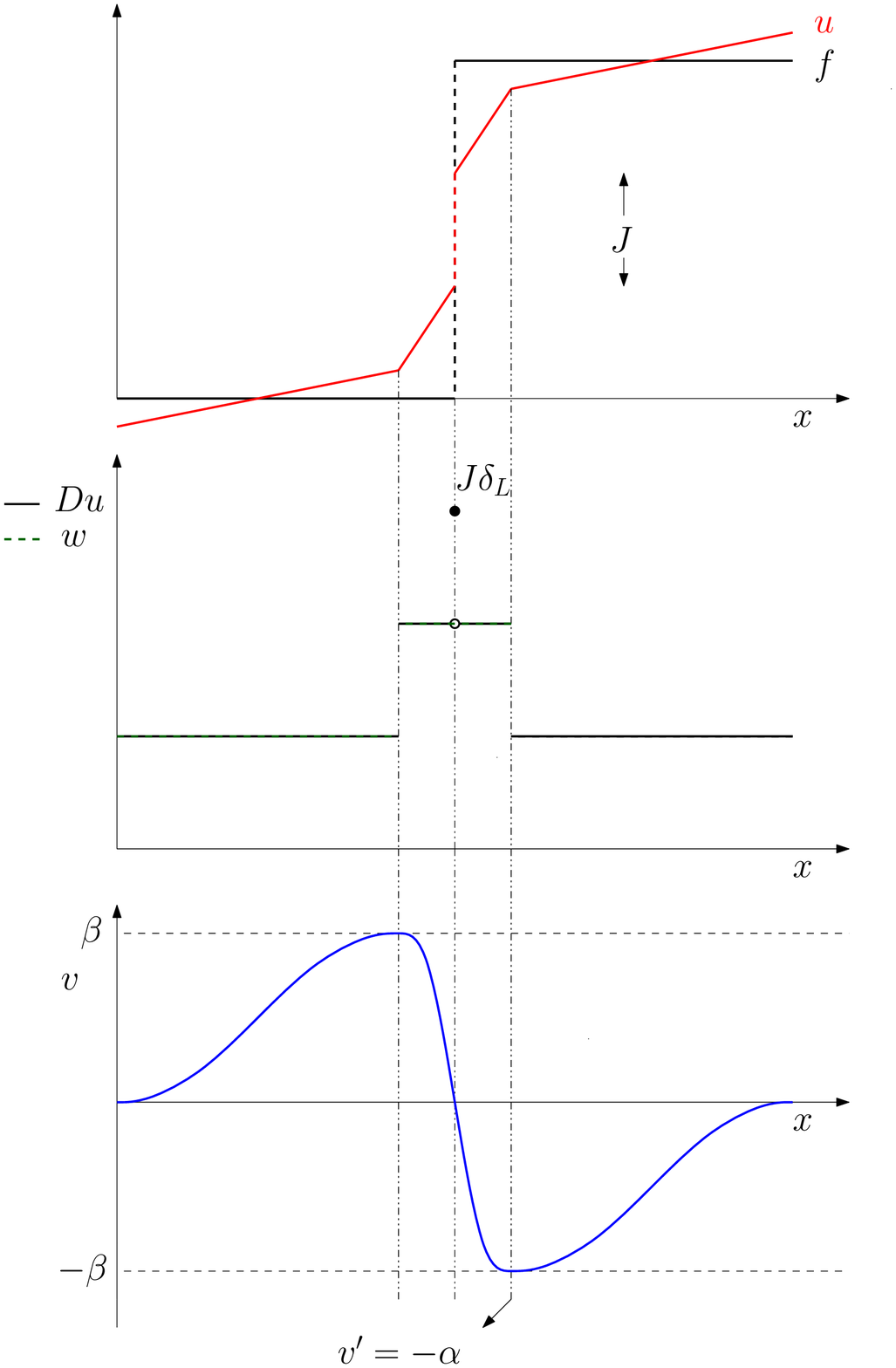}
}
\subfloat[Solution of the type \newline\ref{NP1C}.]{
\includegraphics[scale=0.25]{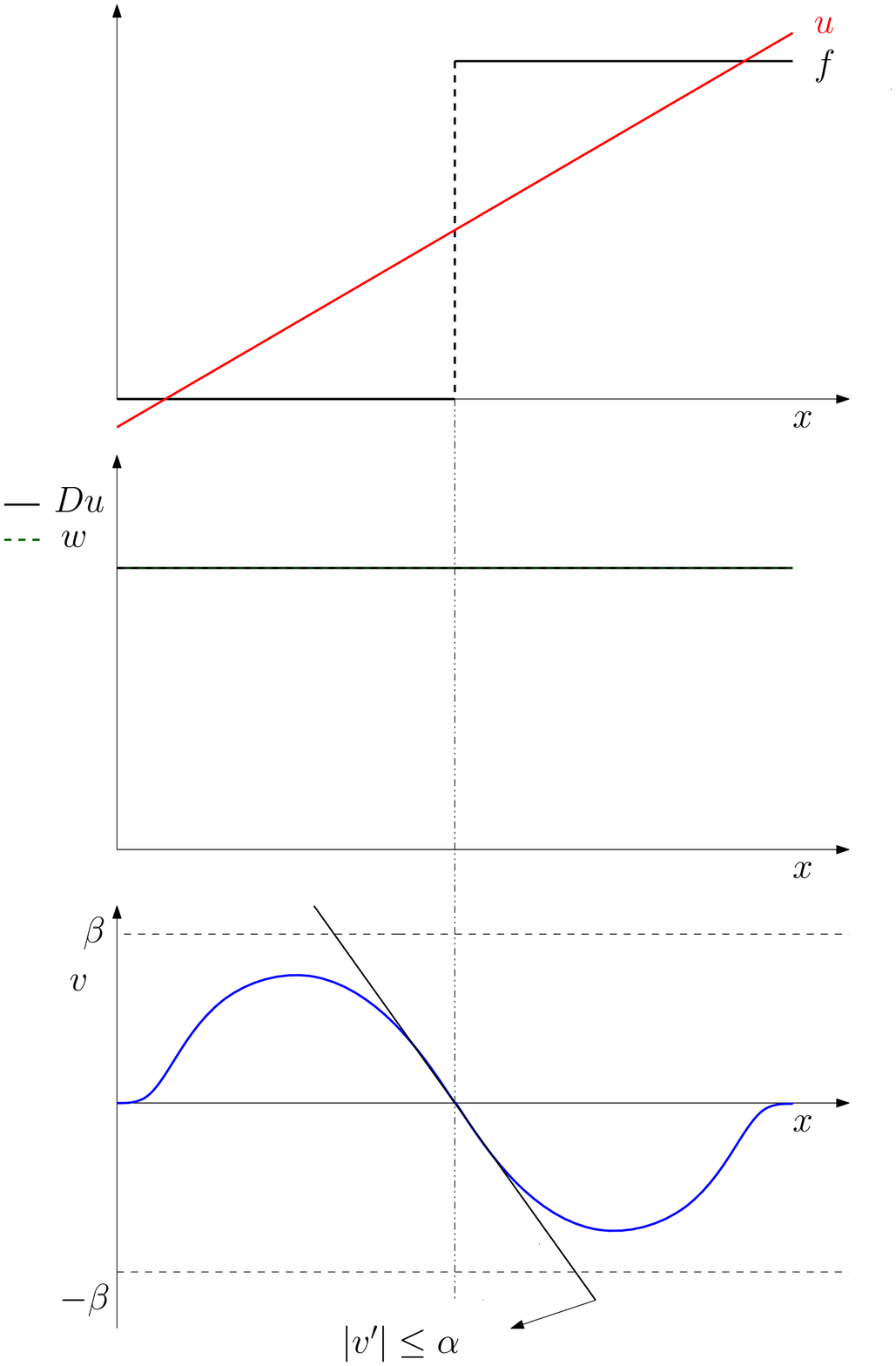}
}
\subfloat[Solution of the type \newline\ref{NP2C}.]{
\includegraphics[scale=0.25]{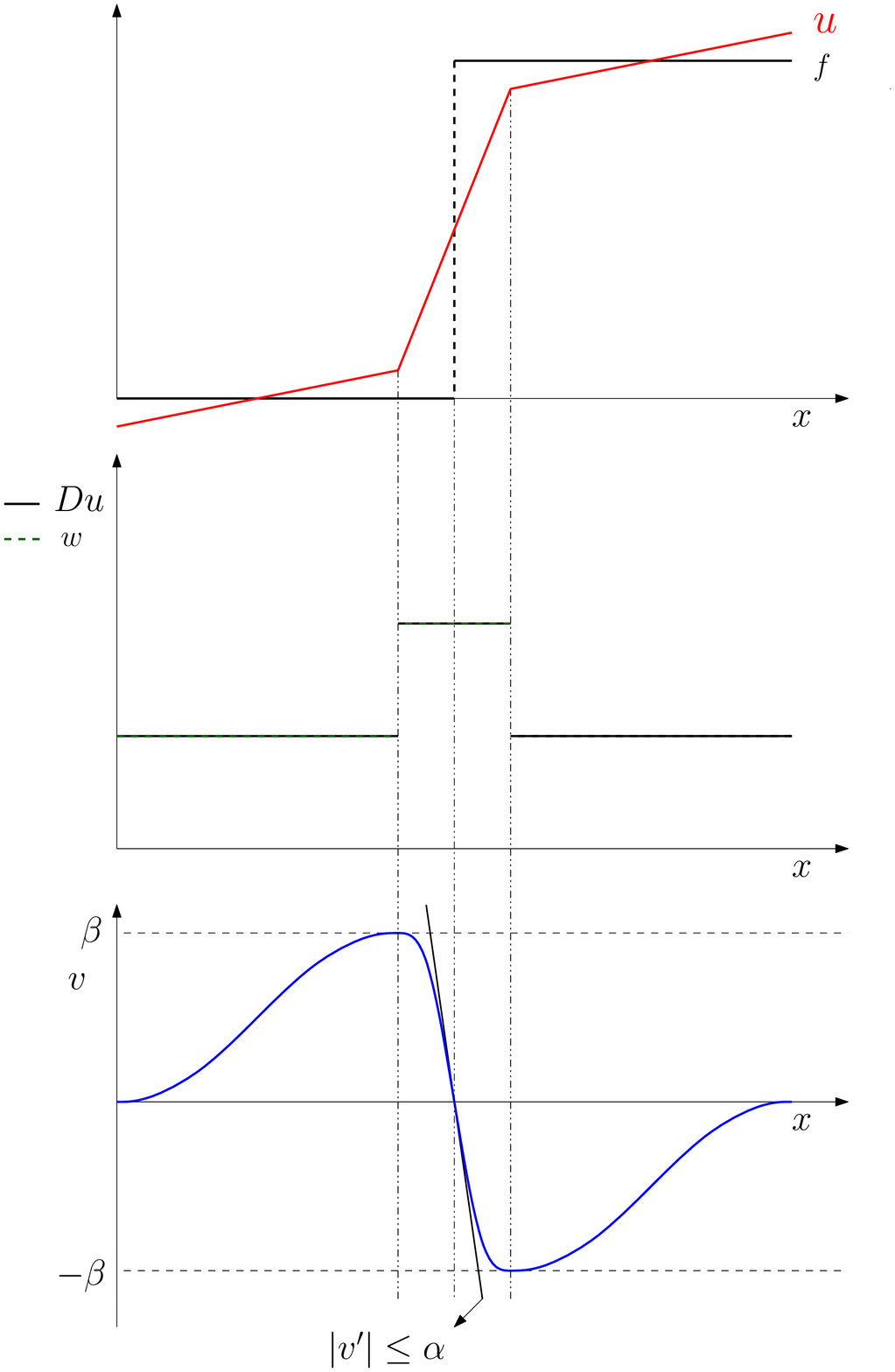}
}
\caption{All the possible types of solution $u$ to the minimisation problem \eqref{P} for the data function $f$ defined in \eqref{f_1}. We also show the forms of the corresponding variables $w$ and $v$.}
\label{allowed}
\end{center}
\end{figure}

In Figure \ref{allowed} we provide a qualitative description on how these four allowed solutions look like, along with the corresponding variables $w$ and $v$. Notice that the solution of the type \ref{NP1C} will be in fact the $L^{2}$--linear regression of $f$. Our next step is to identify the combinations of the parameters $\alpha$ and $\beta$ that lead to each type of solution.

\newtheorem{prop_np1j}[first]{Proposition}
\begin{prop_np1j}[\textbf{\ref{NP1J}}]\label{prop_np1j_label}
The solution $u$ of the problem \eqref{P} with the data function $f$ defined in \eqref{f_1} is of the type \emph{\ref{NP1J}}, see Figure \ref{allowed}(a), if and only if $\alpha$ and $\beta$ satisfy the conditions
\begin{equation}\label{np1j_conditions}
\frac{\beta}{\alpha}\ge \frac{4L}{27}\quad \text{ and }\quad \alpha< \frac{hL}{8}.
\end{equation}
More specifically in that case:
\begin{equation}
\label{u_np1j}
u=
\begin{cases}
\frac{6\alpha}{L^{2}}x-\frac{2\alpha}{L},\quad &\text{if}\quad x\in(0,L),\\
\frac{6\alpha}{L^{2}}(x-L)+h-\frac{4\alpha}{L},\quad &\text{if}\quad x\in(L,2L),
\end{cases}
\end{equation}
i.e., $u$ has a jump discontinuity of size $h-8\alpha/L$ at $x=L$.
\end{prop_np1j}

\begin{proof}
It is easy to check that the solution $u$ will be of the type \ref{NP1J} if and only if we can find a function $v$ in $[0,L]$
 that satisfies the following conditions:
\begin{align}
\label{v_cubic} \text{$v$ is a cubic polynomial,}\qquad     & (\text{from condition } \eqref{C1}, v''=-u, \text{ is an affine function}),\\
\label{v_bc} v(0)=0,\; v'(0)=0, \qquad& \text{(boundary conditions for $v$),}\\
\label{v_symmetry} v(L)=0,	\qquad			& \text{($v$ is an odd function),}\\
\label{v_jump}  v'(L)=-\alpha,\qquad 			& (\text{$u$ jumps at $L$, condition \eqref{C2}}),\\
\label{v_u}      0<-v''(L)<\frac{h}{2},\qquad	& (\text{$u$ jumps at $L$ $\&$ symmetry}),\\
\label{v_c2} |v'|\le \alpha,		\qquad 		&(\text{condition \eqref{C2}}),\\
\label{v_c3} |v|\le \beta,		\qquad 		&(\text{condition \eqref{C3}}).
\end{align}
It is now easy to check that the function
\[v(x)=-\frac{\alpha}{L^{2}}x^{3}+\frac{\alpha}{L}x^{2},\quad x\in [0,L],\]
satisfies conditions \eqref{v_cubic}--\eqref{v_jump} and also the condition \eqref{v_c2}. Observe that $v$ obtains it maximum in $[0,L]$ at $x=2L/3$, with $v(2L/3)=4\alpha L/27$, thus \eqref{v_c3} can be written as
\[\frac{4\alpha L}{27}\le \beta.\]
Finally, condition \eqref{v_u} can be written as
\[0<\frac{4\alpha}{L}<\frac{h}{2}.\]
The last two inequalities form the conditions in \eqref{np1j_conditions}.
Using the fact that $u=-v''$ in $(0,L)$ and taking advantage of the symmetry of u, we have \eqref{u_np1j}.
\end{proof}

We now proceed to the study of the case \ref{NP2J}. As we will see, it is computationally inaccessible to find the exact conditions on $\alpha$ and $\beta$ that result to this kind of solution but we are able to provide some sufficient conditions that are not far away from being necessary as well.

\newtheorem{prop_np2j}[first]{Proposition}
\begin{prop_np2j}[\textbf{\ref{NP2J}}]\label{prop_np2j_label}
If $\alpha$ and $\beta$ satisfy the conditions
\begin{equation}\label{np2j_conditions}
\frac{\beta}{\alpha}< \frac{4L}{27}\quad \text{ and }\quad \beta>\frac{4\alpha^{2}}{3h},
\end{equation}
then the solution $u$ of the problem \eqref{P} with the data function $f$ defined in \eqref{f_1} is of the type \emph{\ref{NP2J}}, see Figure \ref{allowed}(b). More specifically, there exist  points $0<x_{1}<x_{2}<0$
with
\[x_{1}=\frac{x_{2}}{2}\quad \text{ and }\quad L-\frac{3\beta}{\alpha}<x_{2}<L-\frac{9\beta}{4\alpha},\]
such that $u<0$ and $u>0$ in $(0,x_{1})$ and $(x_{1},L)$ respectively with $u'$ having a positive jump at $x_{2}$.
\end{prop_np2j}

\begin{proof}
Again, it is easily verified  that $u$ will be a solution of the type \ref{NP2J} if and only if we can find $\ell_{1}>0$, $\ell_{2}>0$ with $\ell_{1}+\ell_{2}=L$ and two functions  $v_{1}$, $v_{2}$ defined on $[0,\ell_{1}]$ and $[\ell_{1},\ell_{1}+\ell_{2}]$ respectively, such that the following conditions are satisfied:
\begin{align}
\label{v1v2_cubic} \text{$v_{1}$, $v_{2}$ are cubic polynomials},\qquad & (\text{from condition } \eqref{C1}),\\
\label{v1_bc} v_{1}(0)=0,\; v_{1}'(0)=0, \qquad 							& (\text{boundary conditions for $v$}),\\
\label{v1_ell1} v_{1}(\ell_{1})=\beta,\;v_{1}'(\ell_{1})=0,\qquad 			& (\text{$v$ has a maximum at $\ell_{1}$, condition \eqref{C2}}),\\
\label{v2_ell1}v_{2}(\ell_{1})=\beta,\;v_{2}'(\ell_{1})=0,\;v_{2}''(\ell_{1})=v_{1}''(\ell_{1}),\qquad & (\text{continuity of $v$, $v'$ and u}),\\
\label{v2prime_jump} v_{2}'''(\ell_{1})<v_{1}'''(\ell_{1}),\qquad & \text{($u'$ has a positive jump at $\ell_{1}$)}\\
\label{v2_symmetry} v_{2}(L)=0,	\qquad			& \text{($v$ is an odd function),}\\
\label{v2_jump}  v_{2}'(L)=-\alpha,\qquad 			& (\text{$u$ jumps at $L$, condition \eqref{C2}}),\\
\label{v2_u}      0<-v_{2}''(L)<\frac{h}{2},\qquad	& (\text{$u$ jumps at $L$ $\&$ symmetry}),\\
\label{v12_c2} |v_{i}'|\le \alpha,\;i=1,2,		\qquad 		&(\text{condition \eqref{C2}}),\\
\label{v12_c3} |v_{i}|\le \beta, \;i=1,2,		\qquad 		&(\text{condition \eqref{C3}}).
\end{align}
From conditions \eqref{v1v2_cubic}--\eqref{v1_ell1}, we get that
\begin{equation}\label{v1_equi}
v_{1}(x)=-\frac{2\beta}{\ell_{1}^{3}}x^{3}+\frac{3\beta}{\ell_{1}^{2}}x^{2},\quad x\in [0,\ell_{1}].
\end{equation}
We can also check that condition \eqref{v12_c2} for $v_{2}$ is equivalent to
\begin{equation}\label{v12_c2_equi}
\frac{3\beta}{2\alpha}\le \ell_{1},
\end{equation}
while condition \eqref{v12_c3} is satisfied. After some computations, we have that conditions \eqref{v2_ell1}, \eqref{v2_symmetry}, \eqref{v2_jump} give that $v_{2}$ will be of the form
\begin{equation}\label{v2_equi}
v_{2}(x)=\left (\frac{2\beta}{\ell_{1}^{2}\ell_{2}}-\frac{\alpha}{3\ell_{2}^{2}} \right )(x-\ell_{1})^{3}-\frac{3\beta}{\ell_{1}^{2}}(x-\ell_{1})^{2}+\beta.
\end{equation}
We also get the following relationship between $\ell_{1}$ and $\ell_{2}$:
\begin{equation}\label{ell1_ell2}
\gamma\ell_{2}^{2}+\ell_{2}\ell_{1}^{2}-\gamma\ell_{1}^{2}=0,\quad \text{ with }\quad\gamma:=\frac{3\beta}{\alpha},
\end{equation}
where one can check that $\ell_{2}<\gamma$ independently of the value of $\ell_{1}$. Notice now, that if the condition \eqref{v2prime_jump} is satisfied, then $v_{2}$ will be decreasing, with decreasing derivative as well something that would imply that conditions \eqref{v12_c2}--\eqref{v12_c3} hold for $v_{2}$. With the help of \eqref{v2_equi}, \eqref{ell1_ell2} and condition $\ell_{1}+\ell_{2}=L$, we get that condition \eqref{v2prime_jump} is equivalent to
\begin{equation}\label{v2prime_jump_equi}
\ell_{1}^{2}-2L\ell_{1}+2\gamma L<0.
\end{equation}
The inequality \eqref{v2prime_jump_equi} is true if and only if
\begin{equation}\label{v2prime_jump_equi2}
\gamma <\frac{L}{2}\quad \text{and} \quad L-\sqrt{L^{2}-2\gamma L}<\ell_{1}\iff \ell_{2}<\sqrt{L^{2}-2\gamma L}.
\end{equation}
From the fact that $\ell_{1}+\ell_{2}=L$ and \eqref{ell1_ell2} we get that $\ell_{2}$ is the unique solution of
\begin{equation}\label{phi_ell2}
\phi(\ell_{2}):=\sqrt{\frac{\gamma\ell_{2}^{2}}{\gamma-\ell_{2}}}+\ell_{2}-L=0.
\end{equation}
In view of \eqref{phi_ell2} we have that $\ell_{2}=\sqrt{L^{2}-2\gamma L}$ if and only if $\gamma=4L/9$. Since $\phi$ is strictly increasing, one can check that \eqref{v2prime_jump_equi2} is equivalent to the very simple expression
\begin{equation}\label{v2prime_jump_equi3}
\frac{\beta}{\alpha}<\frac{4L}{27}.
\end{equation}
Notice that in that case \eqref{v12_c2_equi} is also satisfied. Finally, the last condition that has to be satisfied is \eqref{v2_u}. Using \eqref{v2_equi} and \eqref{ell1_ell2} it can be checked that it is equivalent to
\begin{equation}\label{v2_u_equi}
\frac{4\ell_{2}-2\gamma}{\ell_{2}^{2}}<\frac{h}{2\alpha}.
\end{equation}
Ideally, one would like to obtain an explicit expression for $\ell_{2}$ from \eqref{phi_ell2} and obtain an inequality involving $\alpha$ and $\beta$ using \eqref{v2_u_equi}. However, this is practically impossible as one would have to solve a cubic equation for $\ell_{2}$. This is why we are giving some estimates instead. One can check that from \eqref{phi_ell2} and \eqref{v2prime_jump_equi3} we can derive that
\begin{equation}\label{estimate_ell2}
\frac{3}{4}\gamma<\ell_{2}<\gamma.
\end{equation}
We have that the expression $(4\ell_{2}-2\gamma)/\ell_{2}^{2}$ in \eqref{v2_u_equi} is a strictly increasing function of $\ell_{2}$ provided that $\ell_{2}<\gamma$ which is true in our case. Thus, in order to satisfy \eqref{v2_u_equi} a sufficient (but not necessary) condition is
\[\frac{4\gamma-2\gamma}{\gamma^{2}}<\frac{h}{2\alpha} \iff \beta>\frac{4\alpha^{2}}{3h}.\]
\end{proof}

\newtheorem*{remark1}{Remark}
\begin{remark1}
\emph{Let us also point out the following fact: Suppose that $\alpha^{\ast}$ and $\beta^{\ast}$ are such, so that conditions \eqref{np2j_conditions} hold, i.e., the solution is of the type \ref{NP2J}. Then there exists a $0<c<\beta^{\ast}$ such that for every $\beta\le c$ the condition \eqref{v2_u_equi} is violated, so we do not have this type of solution. Indeed from \eqref{v2_u_equi} and \eqref{estimate_ell2} we have that
\[\frac{16}{9\gamma}<\frac{4\ell_{2}-2\gamma}{\ell_{2}^{2}}<\frac{h}{2\alpha}.\]
Thus keeping $\alpha$ fixed and choosing $\beta$ such that
\begin{equation}\label{final_condition}
\frac{16}{\gamma}>\frac{h}{2\alpha} \iff \beta<\frac{32\alpha^{2}}{27h}
\end{equation}
we cannot have the solution of the type \ref{NP2J} any more.
}
\end{remark1}

We now turn our attention to the solution of the type \ref{NP1C}. As we mentioned earlier, in that case $u$ is an affine function and thus also the $L^{2}$--linear regression of $f$. In Theorem \ref{l2regression_thm} we gave some thresholds for $\alpha$ and $\beta$ in the general case but we can be more explicit in this specific example.
\newtheorem{prop_np1c}[first]{Proposition}
\begin{prop_np1c}[\textbf{\ref{NP1C}}]
The solution $u$ of the problem \eqref{P} with the data function defined in \eqref{f_1} is of the type \emph{\ref{NP1C}}, see Figure \ref{allowed}(c), if and only if $\alpha$ and $\beta$ satisfy the conditions
\begin{equation}\label{np1c_conditions}
\alpha\ge \frac{hL}{8} \quad \text{ and }\quad \beta\ge \frac{hL^{2}}{54}.
\end{equation}
Moreover, in that case $u$ will be the $L^{2}$--linear regression of $f$ and equal to
\begin{equation}\label{f_l2regression}
u(x)=\frac{3h}{4L}x-\frac{h}{4},\quad x\in (0,2L).
\end{equation}
\end{prop_np1c}

\begin{proof}
Again we can check that the solution will be of the  type \ref{NP1C} if and only if we can find  a function $v$ defined on $[0,L]$ such that the following conditions hold:
\begin{align}
\label{vreg_cubic}\text{$v$ is a cubic polynomial},\qquad &\\
\label{vreg_bc} v(0)=0,\;v'(0)=0,\qquad &\\
\label{vreg_odd} v(L)=0,\qquad &\\
\label{vreg_ctn}-v''(L)=\frac{h}{2}, \qquad &\text{(condition \eqref{C1}, continuity at L \& symmetry)},\\
\label{vreg_a} |v'|\le \alpha, \qquad &\\
\label{vreg_b} |v|\le \beta. \qquad &
\end{align}
We can easily check that conditions \eqref{vreg_cubic}--\eqref{vreg_ctn} give
\begin{equation}\label{vreg_explicit}
v(x)=-\frac{h}{8L}x^{3}+\frac{h}{8}x^{2},\quad x\in [0,L].
\end{equation}
Moreover, the maximum values of $|v|$ and $|v'|$ are $\frac{hL^{2}}{54}$ and $\frac{hL}{8}$ respectively. Thus, conditions \eqref{vreg_a}--\eqref{vreg_b} will be satisfied if and only if \eqref{np1c_conditions} holds.
\end{proof}

Finally, we investigate under which combinations of $\alpha$ and $\beta$ we have solutions of the type \ref{NP2C}. As in the case \ref{NP2J}, it is computationally inaccessible to provide exact conditions, thus we provide again some sufficient but not necessary conditions.

\newtheorem{prop_np2c}[first]{Proposition}
\begin{prop_np2c}[\textbf{\ref{NP2C}}]\label{np2c_prop}
If $\alpha$ and $\beta$ satisfy the conditions
\begin{equation}\label{np2c_conditions}
\beta<\frac{hL^{2}}{54}\quad \text{ and }\beta\le \frac{4L\alpha}{9}-\frac{L^{2}}{27},
\end{equation}
then the solution $u$ of the problem \eqref{P} with the data function $f$ defined in \eqref{f_1} is of the type \emph{\ref{NP2C}}, see Figure \ref{allowed}(d), and $u'$ has a positive jump at a point $x_{2}$ where $\frac{2L}{3}<x_{2}<L$.
\end{prop_np2c}

\begin{proof}
As in the \ref{NP2J} case, $u$ will be a solution of the type \ref{NP2C} if and only if we can find $\ell_{1}>0$, $\ell_{2}>0$ with $\ell_{1}+\ell_{2}=L$ and two functions defined on $[0,\ell_{1}]$ and $[\ell_{1},\ell_{1}+\ell_{2}]$ respectively, such that the following conditions are satisfied:
\begin{align}
\label{c_v1v2_cubic} \text{$v_{1}$, $v_{2}$ are cubic polynomials},\qquad &\\
\label{c_v1_bc} v_{1}(0)=0,\; v_{1}'(0)=0, \qquad 							& \\
\label{c_v1_ell1} v_{1}(\ell_{1})=\beta,\;v_{1}'(\ell_{1})=0,\qquad 			& \\
\label{c_v2_ell1}v_{2}(\ell_{1})=\beta,\;v_{2}'(\ell_{1})=0,\;v_{2}''(\ell_{1})=v_{1}''(\ell_{1}),\qquad &\\
\label{c_v2prime_jump} v_{2}'''(\ell_{1})<v_{1}'''(\ell_{1}),\qquad & \\
\label{c_v2_symmetry} v_{2}(L)=0,	\qquad			&\\
\label{c_v2_u}      -v''(L)=\frac{h}{2}, \qquad &\text{(condition \eqref{C1}, continuity at L \& symmetry)},\\
\label{c_v12_c2} |v_{i}'|\le \alpha,\;i=1,2,		\qquad 		&\\
\label{c_v12_c3} |v_{i}|\le \beta, \;i=1,2.		\qquad 		&
\end{align}
Conditions \eqref{c_v1v2_cubic}--\eqref{c_v1_ell1} together with \eqref{c_v12_c2}--\eqref{c_v12_c3} yield
\begin{equation}\label{c_v1_explicit}
v_{1}(x)=-\frac{2\beta}{\ell_{1}^{3}}x^{3}+\frac{3\beta}{\ell_{1}^{2}}x^{2},\quad x\in [0,\ell_{1}],
\end{equation}
with
\begin{equation}\label{v1primecond}
\frac{3\beta}{2\alpha}\le \ell_{1}.
\end{equation}
Supposing that $v_{2}$ is of the form
\[v_{2}(x)=a_{3}(x-\ell_{1})^{3}+a_{2}(x-\ell_{1})^{2}+a_{1}(x-\ell_{1})+a_{0},\quad x\in [\ell_{1},\ell_{1}+\ell_{2}],\]
conditions \eqref{c_v2_ell1}, \eqref{c_v2_symmetry}, \eqref{c_v2_u} give
\begin{equation}\label{c_v2_explicit}
v_{2}(x)=a_{3}(x-\ell_{1})^{3}-\frac{3\beta}{\ell_{1}^{2}}(x-\ell_{1})^{2}+\beta,\quad x\in  [\ell_{1},\ell_{1}+\ell_{2}],
\end{equation}
\begin{eqnarray}
a_{3}\ell_{2}^{3}-\frac{3\beta\ell_{2}^{2}}{\ell_{1}^{2}}+\beta&=&0,\label{a3_1}\\
-6a_{3}\ell_{2}+\frac{6\beta}{\ell_{1}^{2}}&=&\frac{h}{2}.\label{a3_2}
\end{eqnarray}
One can check that conditions \eqref{c_v2_ell1} and \eqref{c_v2_u} impose $v_{2}'$ to be decreasing , so in order to satisfy \eqref{c_v12_c2} and \eqref{c_v12_c3} it suffices to impose $|v_{2}'(\ell_{1}+\ell_{2})|\le \alpha$, which, with the help of \eqref{c_v2_explicit} and \eqref{a3_2}, is equivalent to
\begin{equation}\label{ell_1_2_ineq}
\frac{3\beta\ell_{2}}{\ell_{1}^{2}}+\frac{h\ell_{2}}{4}\le \alpha.
\end{equation}
Moreover combining \eqref{a3_1}, \eqref{a3_2} and the fact that $\ell_{1}+\ell_{2}=L$, we can get a relationship between $\ell_{1}$ and $\ell_{1}$ and an equation for $\ell_{2}$:
\begin{equation}\label{ell1_ell2_c}
 \ell_{2}=\sqrt{\frac{12\beta}{h+\frac{24\beta}{\ell_{1}^{2}}}},
\end{equation}
\begin{equation}\label{ell2_L}
\phi(\ell_{1}):=\ell_{1}+\sqrt{\frac{12\beta}{h+\frac{24\beta}{\ell_{1}^{2}}}}-L=0.
\end{equation}
It is easy to see that the equation \eqref{ell2_L} has a unique solution in $(0,L)$ but one has to solve a quartic in order to express it explicitly. Thus, as in the case \ref{NP2J} we give some estimates. Observe firstly that using \eqref{c_v1_explicit}, \eqref{c_v2_explicit}, \eqref{a3_2} and \eqref{ell1_ell2_c} condition \eqref{c_v2prime_jump} is satisfied if and only if
\begin{equation}\label{c_v2prime_jump_equi}
\frac{18\beta\ell_{1}}{\ell_{2}}-\frac{6\beta\ell_{1}^{3}}{\ell_{2}^{3}}<-12\beta,
\end{equation}
which is satisfied if and only if $\ell_{2}<\frac{\ell_{1}}{2}$. However, from \eqref{ell1_ell2_c}, this is true if and only if $\ell_{1}>\sqrt{\frac{24\beta}{h}}$. From the fact that the function $\phi$ is strictly increasing, the last inequality is true if and only if 
\begin{equation}\label{np2c_cond1}
\beta<\frac{hL^{2}}{54}.
\end{equation}
It remains to identify some sufficient conditions for \eqref{ell_1_2_ineq}. Observe that under \eqref{np2c_cond1} we have $\phi(2L/3)<0$ and $\phi(L>0)$, which means that
\begin{equation}\label{final_el12}
\frac{2L}{3}<\ell_{1}<L \quad \text{ and }\quad 0<\ell_{2}<\frac{L}{3}.
\end{equation}
Using \eqref{final_el12}, we find that a sufficient (but not necessary) condition for \eqref{ell_1_2_ineq}
is
\begin{equation}\label{np2c_cond2}
\frac{3\beta}{(\frac{2L}{3})^{2}}\frac{L}{3}+\frac{hL}{12}<\alpha \iff \beta\le \frac{4L\alpha}{9}-\frac{L^{2}h}{27}.
\end{equation}
We can easily verify that under \eqref{np2c_cond2} the condition \eqref{v1primecond} is satisfied as well.
\end{proof}

\begin{figure}[h!]
\begin{center}
\includegraphics[scale=0.45]{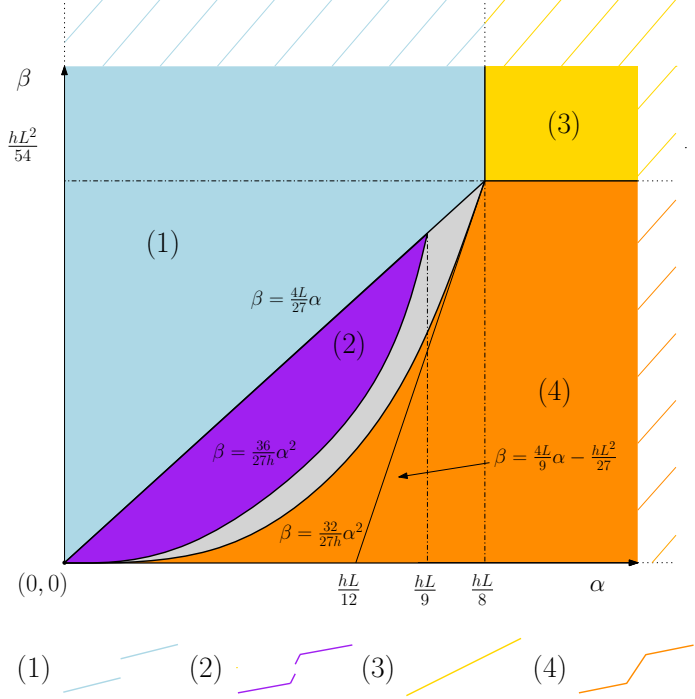}
\caption{Illustration of the four different types of solutions of \eqref{P} for the piecewise constant function $f$ that result for different combinations of the parameters $\alpha$ and $\beta$.}
\label{f_alpha_beta}
\end{center}
\end{figure}

We now summarise how the solutions of the problem \eqref{P} with the data function $f$ are affected by the different choices of the parameters $\alpha$ and $\beta$. In Figure \ref{f_alpha_beta} we have partitioned the set $\{\alpha>0,\;\beta>0\}$ into different areas that correspond to the four different possible solutions. These areas are:
\begin{enumerate}[(1)]
\item Blue colour: $\left\{\alpha<\frac{hL}{8},\; \beta\ge \frac{4L}{27}\alpha\right\}$, (necessary and sufficient condition), \textbf{\ref{NP1J}}, see Figure \ref{allowed}(a).
\item Purple colour: $\left\{\beta<\frac{4L}{27}\alpha,\;\beta\ge \frac{36}{27h}\alpha^{2}\right\}$, (sufficient condition), \textbf{\ref{NP2J}}, see Figure \ref{allowed}(b).
\item Yellow colour: $\left\{\alpha\ge \frac{hL}{8},\; \beta\ge \frac{hL^{2}}{54}\right\}$, (necessary and sufficient condition), \textbf{\ref{NP1C}}, see Figure \ref{allowed}(c).
\item Orange colour: $\left\{\beta\le \frac{32}{27h}\alpha^{2},\; \beta<\frac{hL^{2}}{54}\right\}$, (sufficient condition), \textbf{\ref{NP2C}}, see Figure \ref{allowed}(d).
\end{enumerate}
\vspace{0.2 cm}

Notice that according to Proposition \ref{np2c_prop} our initial sufficient conditions for the solution of the type \ref{NP2C} to happen were $\beta<\frac{hL^{2}}{54}$ and $\beta\le \frac{4L}{9}\alpha-\frac{L^{2}h}{27}$. However, according to the Remark after Proposition \ref{prop_np2j_label} when condition $\beta\le \frac{32}{27h}\alpha^{2}$ holds then the solution of the type \ref{NP2J} cannot happen and since the conditions for \ref{NP1C} and \ref{NP2C} are necessary and sufficient we conclude that when $\beta\le \frac{32}{27h}\alpha^{2}$ holds (together with $\alpha\le\frac{hL}{8}$) then the solution of the type \ref{NP2C} occur. Notice that 
\[\{\alpha>0,\;\beta>0:\;\beta\le\frac{4L}{9}\alpha-\frac{L^{2}h}{27},\; \alpha\le\frac{hL}{8}\}\subseteq \{\alpha>0,\;\beta>0:\;\beta\le \frac{32}{27h}\alpha^{2},\;\alpha\le\frac{hL}{8}\},\] see Figure \ref{f_alpha_beta}.

As we mentioned, due to computational issues it is difficult to derive necessary and sufficient conditions for the solutions of the type \ref{NP2J} and  \ref{NP2C}. However, the estimates we have provided are not far away from being sharp as the unknown grey area in Figure \ref{f_alpha_beta}, $\{\beta>\frac{32}{27h}\alpha^{2},\;\beta<\frac{36}{27h}\alpha^{2},\;\beta<\frac{4L}{27}\alpha\}$ is relatively small.

\subsection{Piecewise affine function with a single jump}

In this section we are choosing the data function to be a simple piecewise affine function $g$, see Figure \ref{g_jump}. However, as we will see, we do not have to perform the previous computations to identify solutions for \eqref{P} as the solutions have a close connection with the ones that correspond to the piecewise constant function $f$.

\begin{figure}[h!]
\begin{center}
\includegraphics[scale=0.40]{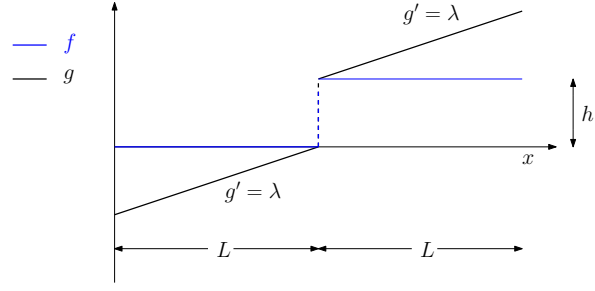}
\caption{The piecewise affine function $g$. It has a jump discontinuity at $x=L$ and the gradient in the affine parts is equal to $\lambda$.}
\label{g_jump}
\end{center}
\end{figure}

\newtheorem{connection}[first]{Proposition}
\begin{connection}\label{connectionlabel}
Let  $f$ and $g$ be the following functions on $(0,2L)$:
\begin{equation*}\label{gjump}
f(x)=
\begin{cases}
0 & \text{ if }\;\;\; x\in (0,L),\\
h & \text{ if }\;\;\; x\in (L,2L),
\end{cases}
\qquad\
g(x)=
\begin{cases}
\lambda(x-L) &  \text{if }\;x\in (0,L),\\
\lambda(x-L)+h &  \text{if }\;x\in (L,2L).
\end{cases}
\end{equation*}
Then $u_{f}$ is a solution to the minimisation problem \ref{P} with data function $f$ if and only if $u_{g}$ is a solution with data function $g$, where
\[u_{g}(x)=u_{f}(x)+\lambda(x-L).\]
\end{connection}

\begin{proof}
Suppose that $u_{f}$ is a solution of \eqref{P} with data $f$, for some combination of the parameters $\alpha$ and $\beta$ and let $v_{f}$, $w_{f}$ to be the corresponding dual variables. We will show that $u_{g}(x)=u_{f}(x)+\lambda(x-L)$ is a solution for data $g$ for the same combination of $\alpha$ and $\beta$ and vice versa. The optimality conditions \eqref{C1}, \eqref{C2}, \eqref{C3} read:
\begin{eqnarray*}
v_{f}''&=&f-u_{f}, \label{vf_cond1}\\
-v_{f}&\in& \alpha \,\mathrm{Sgn}(Du_{f}-w_{f}),\\
v_{f}&\in& \beta\,\mathrm{Sgn}(Dw_{f}).
\end{eqnarray*} 
Observing that we have $g(x)=f(x)+\lambda(x-L)$ we set $v_{g}=v_{f}$ and $w_{g}(x)=w_{f}(x)+\lambda$. Then we have
\begin{eqnarray*}
v_{g}''=v_{f}''
	   =f-u_{f}
	   =g-u_{g},
\end{eqnarray*}
\begin{eqnarray*}
-v'_{g}\in\alpha\, \mathrm{Sgn}(Du_{g}-w_{g}) \iff
-v'_{g}\in\alpha\,\mathrm{Sgn}(Du_{f}+\lambda -w_{f}-\lambda)\iff
-v'_{f}\in\alpha\,\mathrm{Sgn}(Du_{f}-w_{f}),
\end{eqnarray*}
 and also
\begin{eqnarray*}
v_{g}\in \beta\,\mathrm{Sgn}(Dw_{g}) \iff
v_{g}\in \beta \,\mathrm{Sgn}(Dw_{f}) \iff
v_{f} \in \beta\, \mathrm{Sgn}(Dw_{f}),
\end{eqnarray*}
thus \eqref{C1}, \eqref{C2}, \eqref{C3} hold for $u_{g}$, $v_{g}$ and $w_{g}$. Similarly, we show that if $u_{g}$ is a solution 
for data $g$ then $u_{f}(x):=u_{g}(x)-\lambda(x-L)$ is a solution for data $f$.

\end{proof}

\begin{figure}
\begin{center}
\subfloat[Solution of the type \newline\ref{NP1J}.]
{
\includegraphics[scale=0.27]{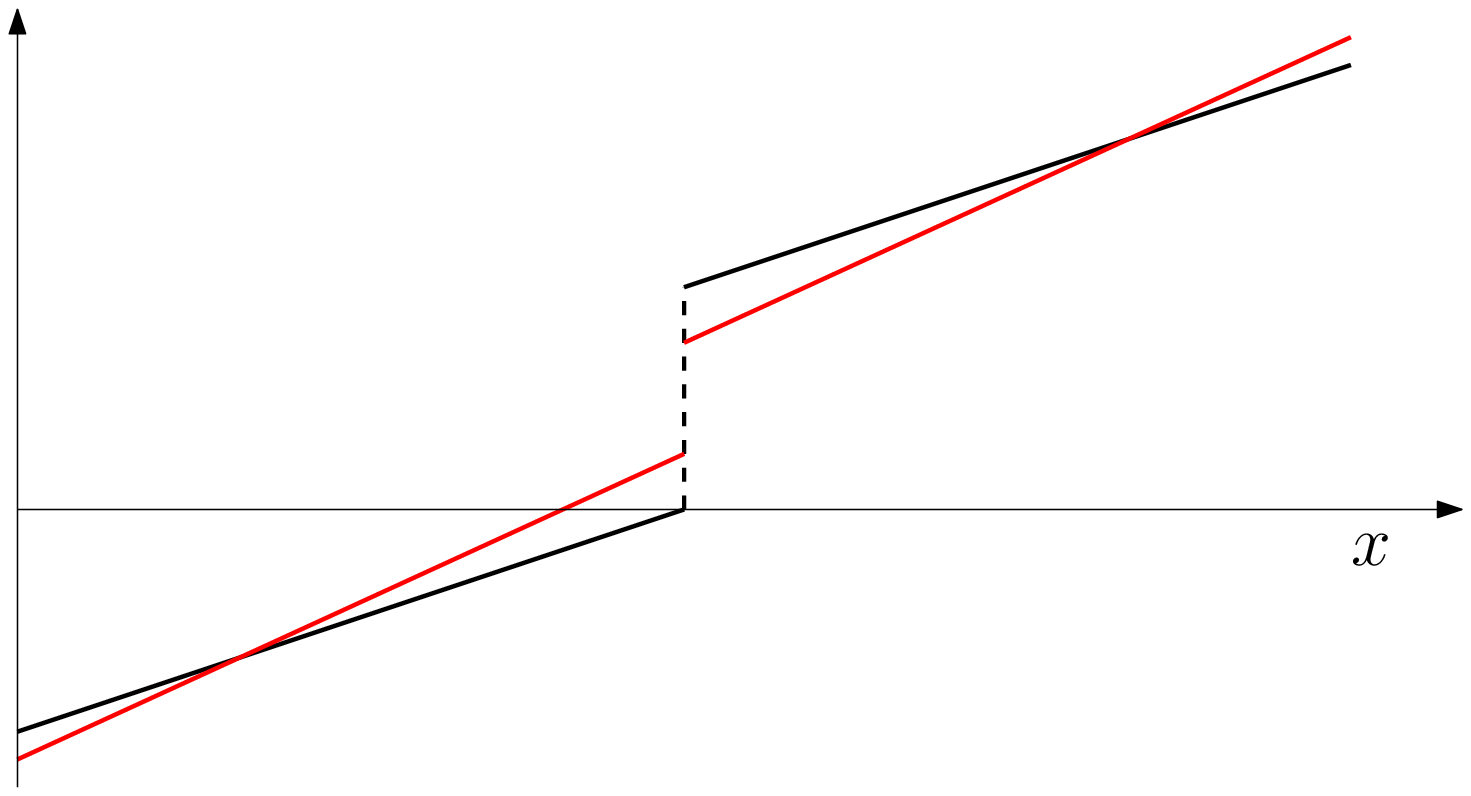}
}
\subfloat[Solution of the type \newline\ref{NP2J}.]{
\includegraphics[scale=0.27]{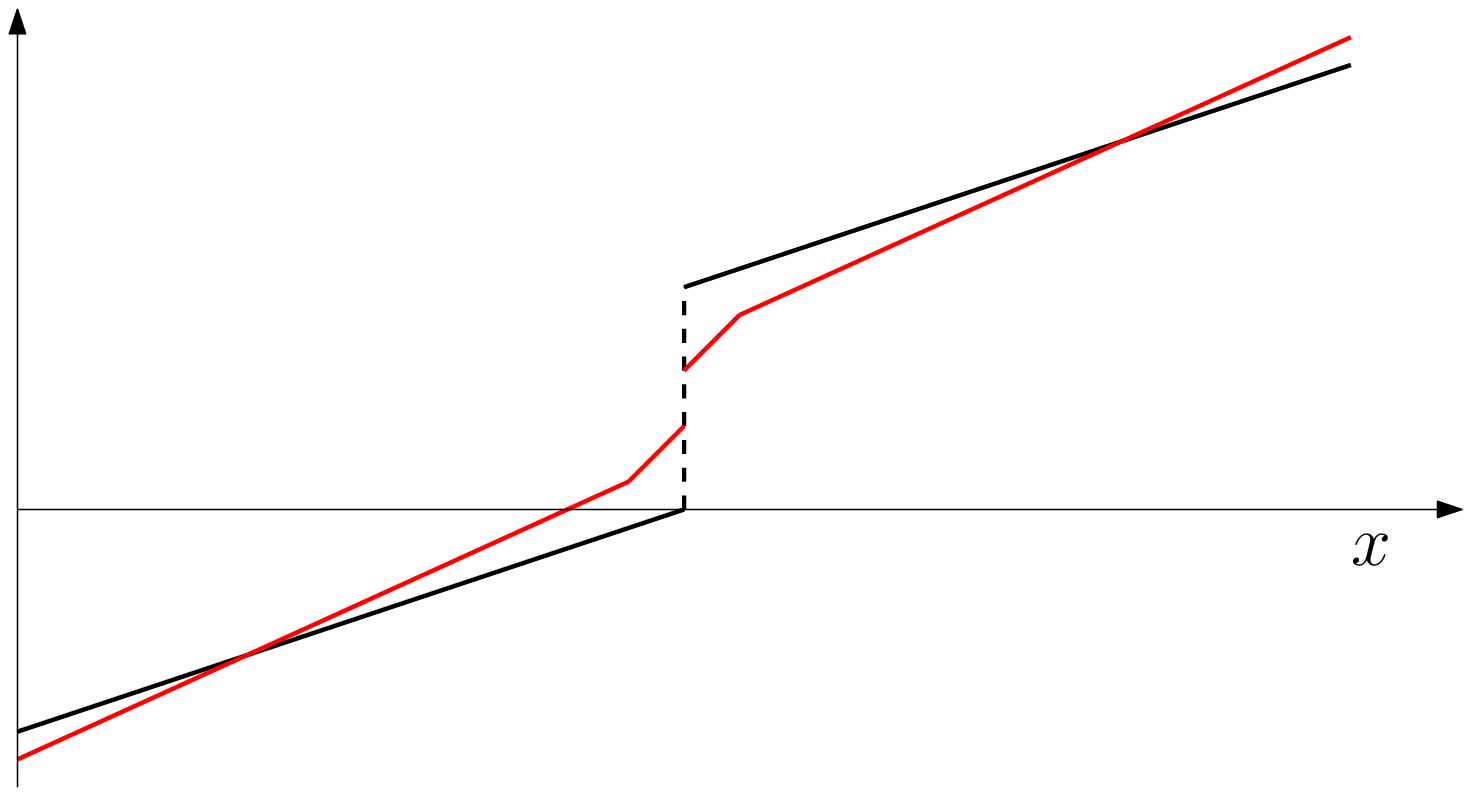}
}
\subfloat[Solution of the type \newline\ref{NP1C}.]{
\includegraphics[scale=0.27]{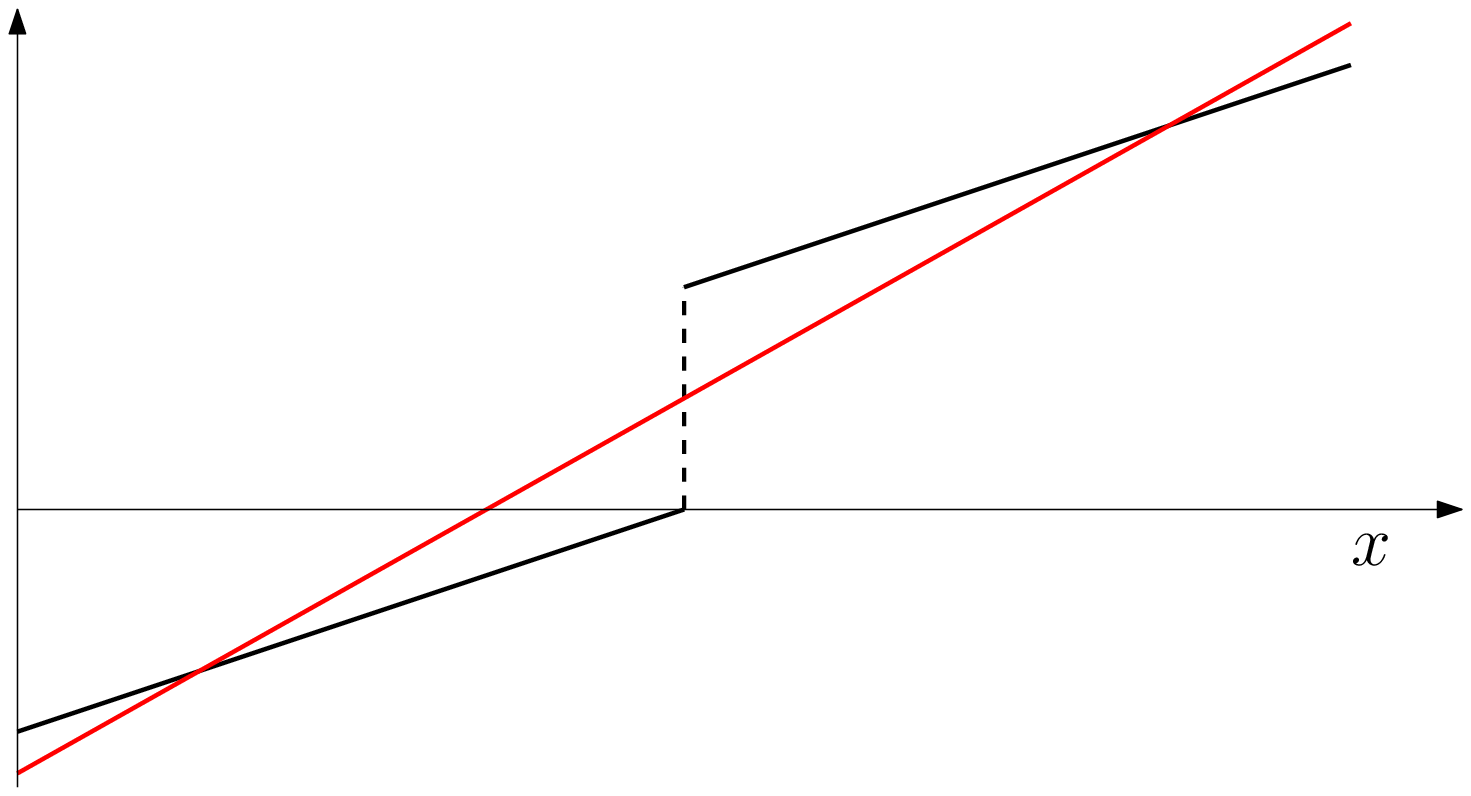}
}
\subfloat[Solution of the type \newline\ref{NP2C}.]{
\includegraphics[scale=0.27]{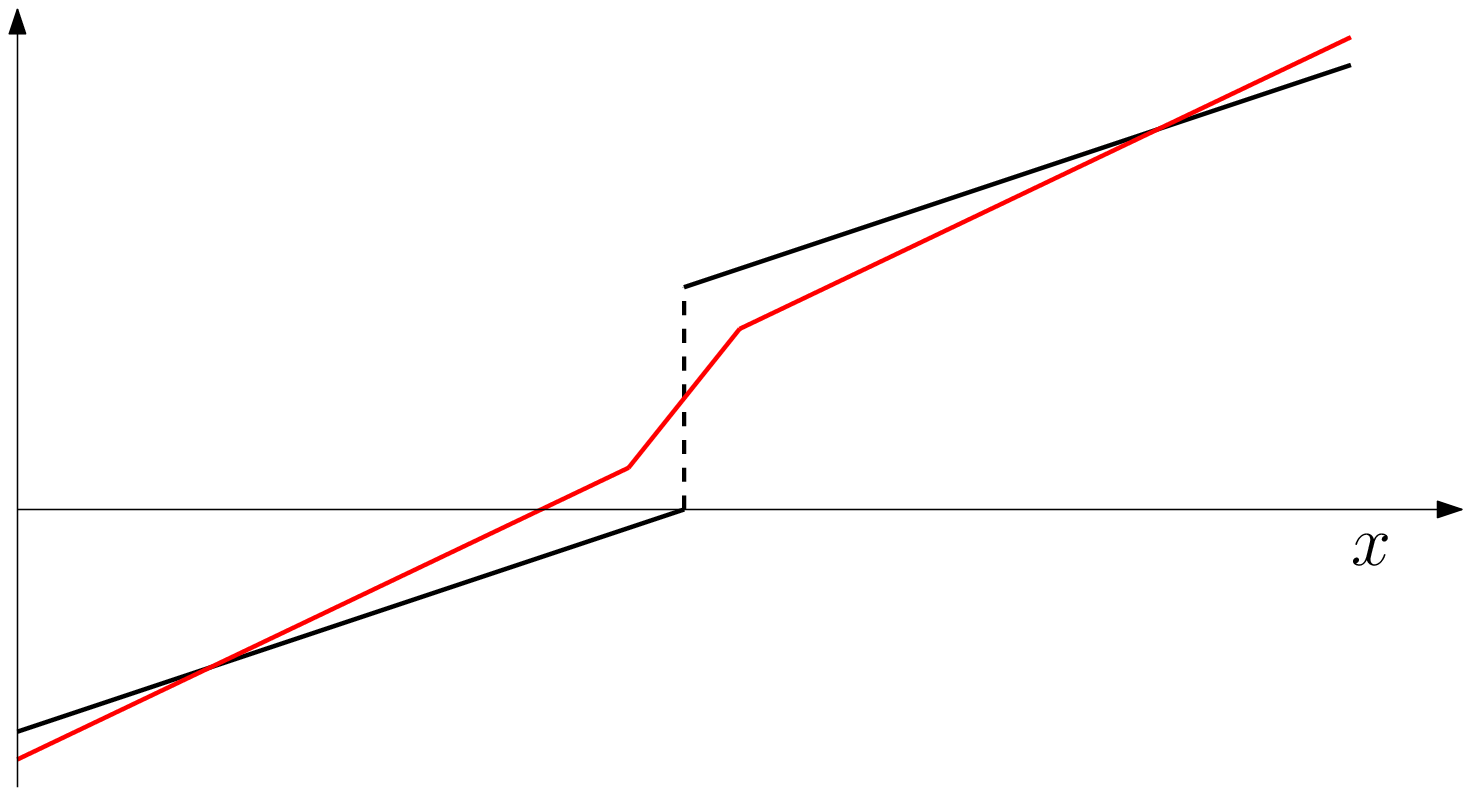}
}
\caption{All the possible types of solution $u$ to the minimisation problem \ref{P} for the data function $g$.}
\label{g_allowed}
\end{center}
\end{figure}

In Figure \ref{g_allowed} we show all such possible solutions. These solutions correspond to the same combinations of $\alpha$ and $\beta$ that are shown in Figure \ref{f_alpha_beta}. One can observe here the capability of TGV to preserve piecewise affine structures. This is in contrast to TV regularisation  which promotes piecewise constant reconstructions.

\subsection{Hat function}

\begin{figure}[h!]
\begin{center}
\includegraphics[scale=0.40]{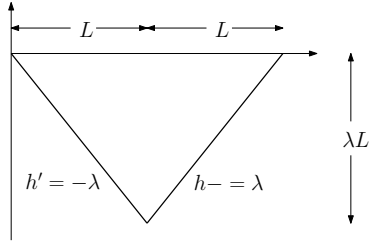}
\caption{The hat function h.}
\label{h_jump}
\end{center}
\end{figure}

In this section we are computing exact solutions for the hat function
\begin{equation}\label{definition_h}
h(x)=\lambda|x-L|-\lambda L, \quad x\in (0,2L),
\end{equation}
where $\lambda>0$, see Figure \ref{h_jump}. The study of this case gives an insight about how TGV deals with local extrema. Note again that the solutions will be symmetric, as $h$ is as well. Working similarly as we did for the function $f$, we conclude that the only possible types of solutions $u$ are the following: 
\begin{align}
\label{cec}\tag{C-E-C} & \text{$u$ is constant on $(0,x_{1})$, equal to $h$ on $(x_{1},x_{2})$ and constant on $(x_{1},L)$, where $(0<x_{1}<x_{2}<L)$,}\\
\label{aea}\tag{A-E-A} & \text{$u$ is affine on $(0,x_{1})$, equal to $h$ on $(x_{1},x_{2})$, and  affine on $(x_{1},L)$, where $(0<x_{1}<x_{2}<L)$,}\\
\label{c}\tag{C} & \text{$u$ is constant on $(0,L)$,}\\
\label{a}\tag{A} & \text{$u$ is affine on $(0,L)$,}
\end{align}

\begin{figure}
\begin{center}
\subfloat[Solution of the type \ref{cec}]
{
\hspace{-0.6cm}\includegraphics[scale=0.29]{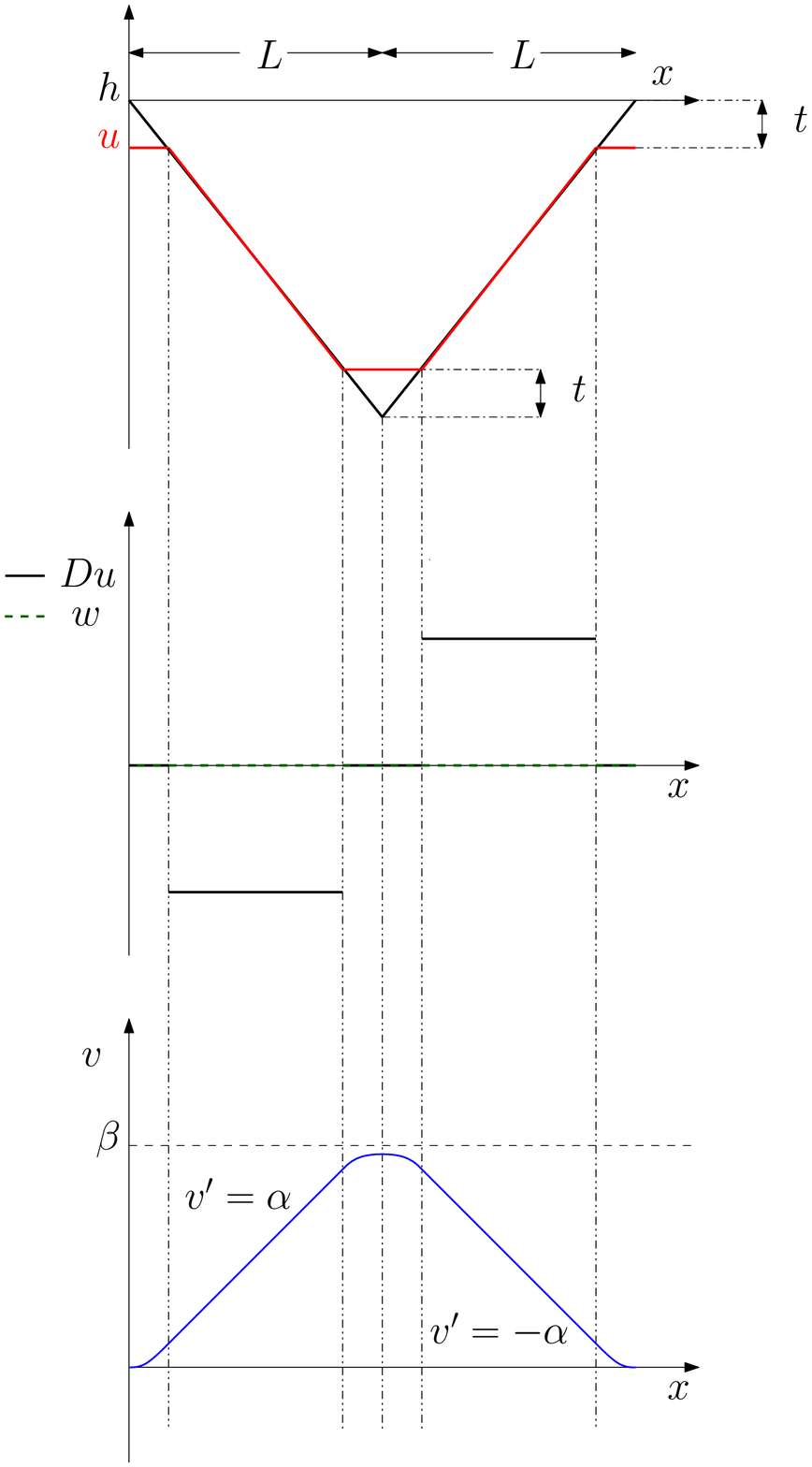}
}
\subfloat[Solution of the type \ref{aea}]{
\includegraphics[scale=0.29]{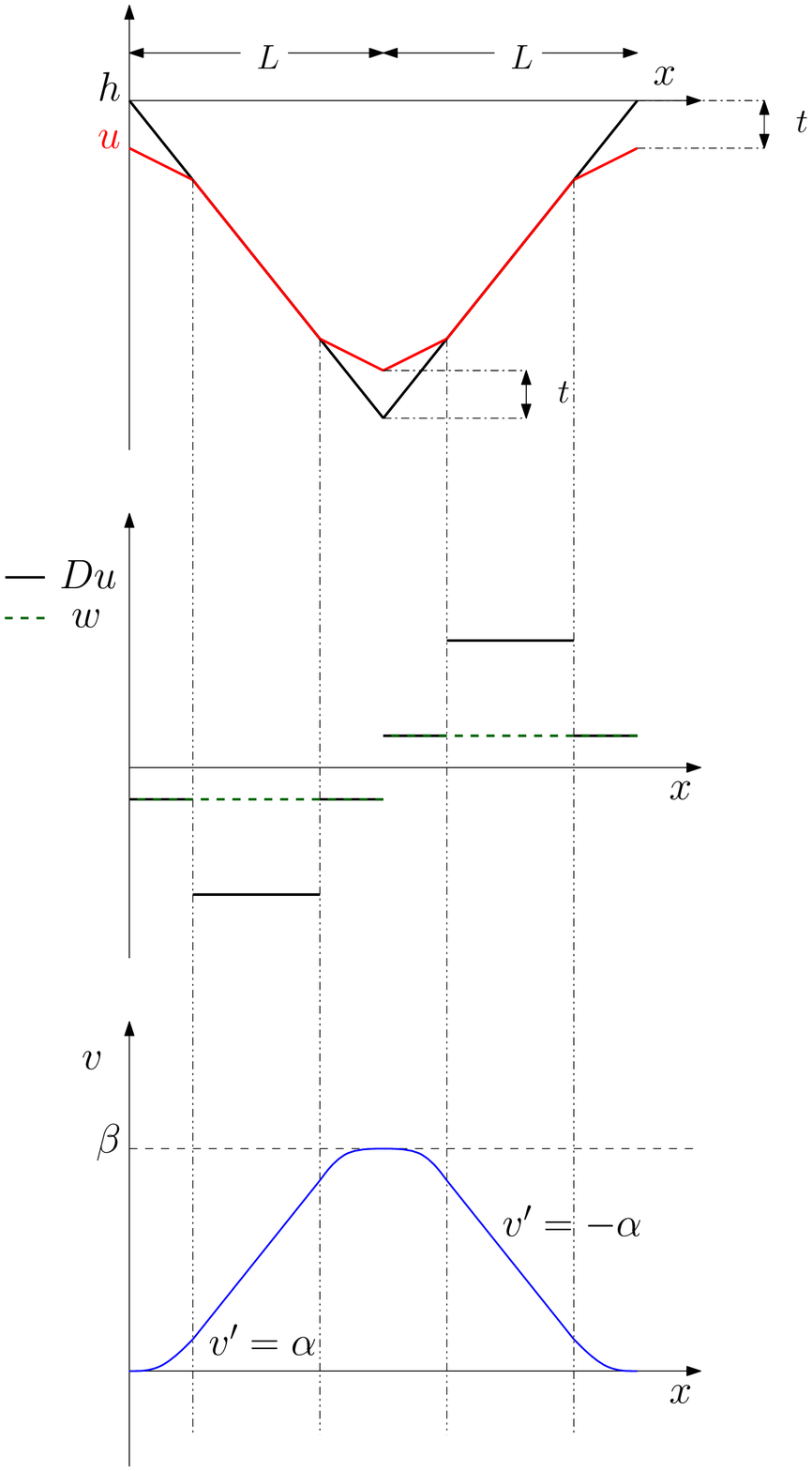}
}
\subfloat[Solution of the type \ref{c}]{
\includegraphics[scale=0.29]{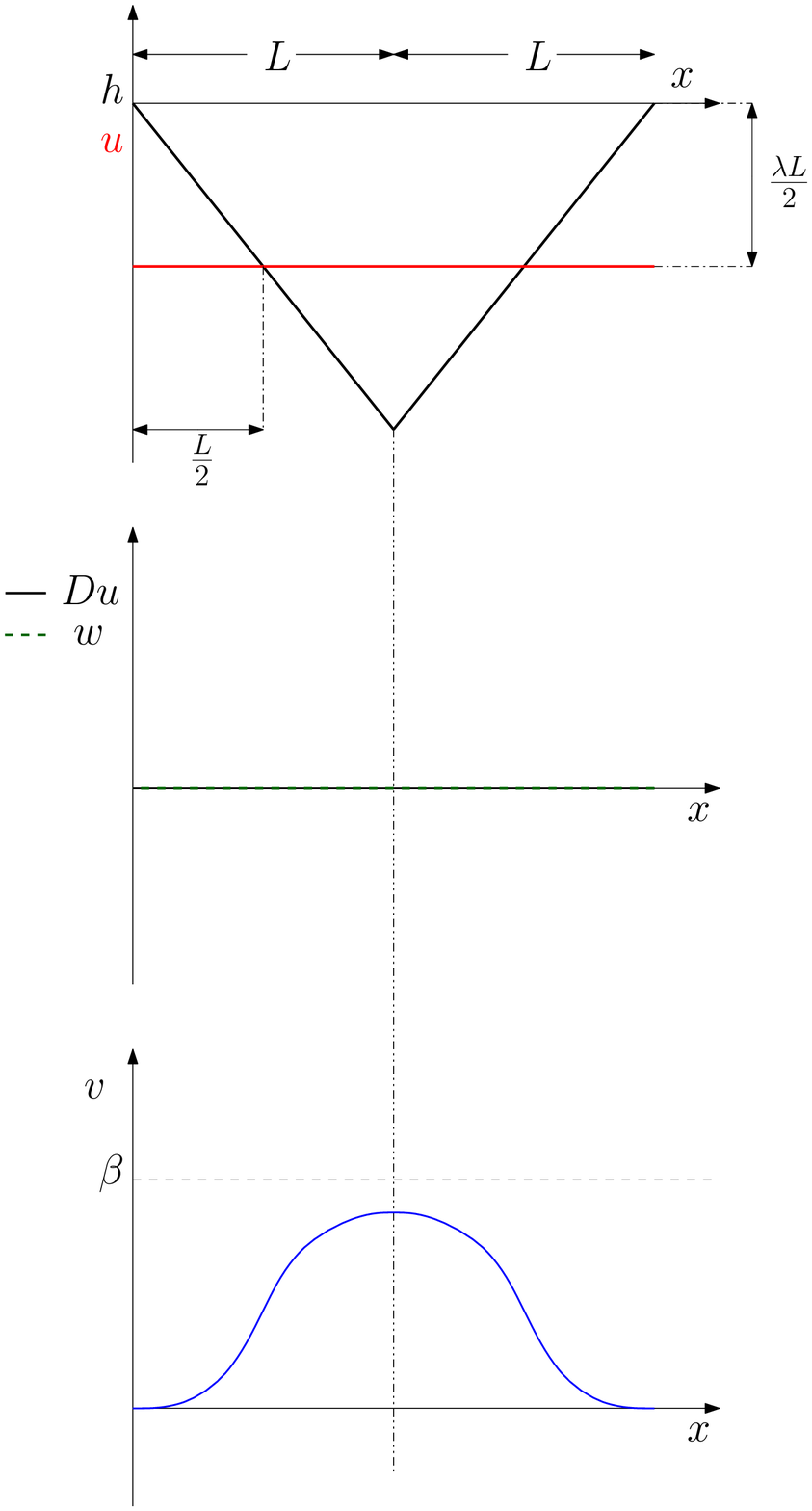}
}
\subfloat[Solution of the type \ref{a}]{
\includegraphics[scale=0.29]{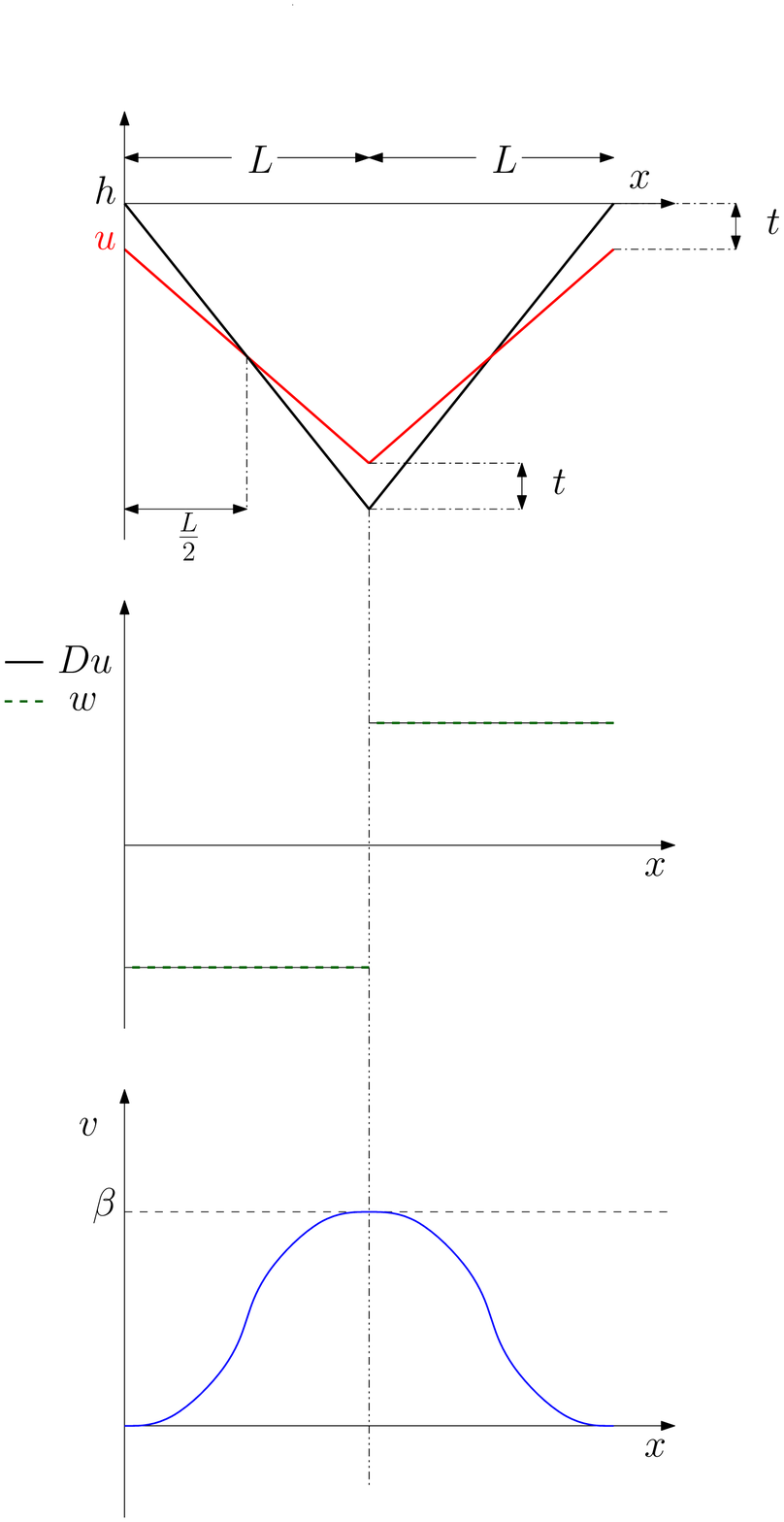}
}
\caption{All the possible types of solution $u$ to the minimisation problem \eqref{P} for the data function $h$ defined in \eqref{definition_h}. We also show the forms of the corresponding variables $w$ and $v$.}
\label{h_allowed}
\end{center}
\end{figure}

In Figure \ref{h_allowed} we show how these four types of solutions look like along with the corresponding variables $w$ and $v$. Notice that the solution of the type \ref{c} is the $L^{2}$--linear regression. Also, the solution of the type \ref{cec} corresponds to the TV--like solutions predicted by Theorem \ref{evenTV_label}. In the following, we investigate which combinations of the parameters $\alpha$ and $\beta$ correspond to each solution. Unlike the case with the piecewise constant function $f$, we are able to give necessary and sufficient conditions for every case.

\newtheorem{hat_main}[first]{Proposition}
\begin{hat_main}\label{hat_main_label}
Let $h$ be the hat function defined in \eqref{definition_h}. Then the solution $u$ of the problem \eqref{P} can be of the following type:
\begin{enumerate}[(1)]
\item The solution $u$ is of the type \emph{\textbf{\ref{cec}}}, Figure \ref{h_allowed}(a), if and only if
\begin{equation}\label{cec_cond}
\alpha<\frac{\lambda L^{2}}{8}\quad \text{ and }\quad \beta\ge \alpha L-\frac{2\alpha}{3}\sqrt{\frac{2\alpha}{\lambda}}.
\end{equation}
 In that case
\[u(x)=
\begin{cases}
-\sqrt{2\alpha\lambda} &  \text{if }\; x\in \left(0,\sqrt{\frac{2\alpha}{\lambda}}\right),\\
-\lambda x                   &  \text{if }\; x\in \left(\sqrt{\frac{2\alpha}{\lambda}},L-\sqrt{\frac{2\alpha}{\lambda}}\right),\\
-\lambda L+\sqrt{2\alpha\lambda} & \text{if }\; x\in \left(L-\sqrt{\frac{2\alpha}{\lambda}}),L\right).
\end{cases}
\]
\item  The solution $u$ is of the type \emph{\textbf{\ref{aea}}}, Figure \ref{h_allowed}(b), if and only if
\begin{equation}\label{aea_cond}
\beta>\frac{2L\alpha}{3}\quad \text{ and }\quad \beta< \alpha L-\frac{2\alpha}{3}\sqrt{\frac{2\alpha}{\lambda}}.
\end{equation}
 In that case
 \[u(x)=
\begin{cases}
-\mu x -\sqrt{2\alpha(\lambda-\mu)} &  \text{if }\; x\in \left(0,\frac{3}{2} \left(L-\frac{\beta}{\alpha}\right)\right),\\
-\lambda x                   & \text{if }\; x\in \left(\frac{3}{2} \left(L-\frac{\beta}{\alpha}\right),L-\frac{3}{2} \left(L-\frac{\beta}{\alpha}\right)\right),\\
-\mu x -(\lambda-\mu)L+2\sqrt{2\alpha(\lambda-\mu)} & \text{if }\; x\in \left(L-\frac{3}{2} \left(L-\frac{\beta}{\alpha}\right),L\right), 
\end{cases}
\;\text{ with }\; \mu=\lambda-\frac{8\alpha}{9\left(L-\frac{\beta}{\alpha}\right)^{2}}.
\]
\item The solution $u$ is of the type \emph{\textbf{\ref{c}}}, Figure \ref{h_allowed}(c), if and only if
\begin{equation}\label{c_cond}
\beta\ge \frac{\lambda L^{3}}{12} \quad \text{ and }\quad \alpha\ge \frac{\lambda L^{2}}{8}.
\end{equation}
 In that case
 \[u(x)=\frac{\lambda L}{2},\quad x\in (0,2L).\]
\item The solution $u$ is of the type \emph{\textbf{\ref{a}}}, Figure \ref{h_allowed}(d), if and only if
\begin{equation}\label{a_cond}
\beta<\frac{\lambda L^{3}}{12} \quad \text{ and }\quad \beta \le \frac{2L\alpha}{3}.
\end{equation}
 In that case
\[ u(x)=\mu|x-L|-(\lambda L-t), \quad \text{ with }\quad\mu=\lambda-\frac{12\beta}{L^{3}} \quad\text{ and }\quad t=\frac{6\beta}{L^{2}}.\]
\end{enumerate}
\end{hat_main}

\begin{proof}
(1) Since we are looking at solutions of the type \ref{cec}, we must find $0<x_{1}<x_{2}<L$ such that $u=t$ for a constant $t$ on $(0,x_{1})$, $u=h$ on $(x_{1},x_{2})$ and $u=-t-\lambda (x_{2}-x_{1})$ on $(x_{2},L)$. As far as the variable $w$ is concerned, we have that $w=Du=0$ on $(0,x_{1})\cup (x_{2},L)$. However since $u=f$ on $(x_{1},x_{2})$, we have that $v$ is affine in that interval, thus it cannot have an extremum there. Thus, condition \eqref{C3} forces $w=0$ on $(x_{1},x_{2})$ as well and \eqref{C2} forces $v'=\alpha$ there.
We conclude that the solution $u$ will be of the type \ref{cec} if and only if we can find  $0<x_{1}<x_{2}<L$ and a function $v$ with the following properties:
\begin{align}
\label{cec_v_cts} \text{$v$ and $v'$ are continuous},\qquad & \text{($v\in H_{0}^{2}$)},\\
\label{cec_v_start}\text{$v$ is a cubic polynomial on $(0,x_{1})$ and $(x_{2},L)$}, \qquad& \text{(condition \eqref{C1})},\\
\label{cec_v_bc} v(0)=0,\; v'(0)=0, \qquad& \text{(boundary conditions for $v$)},\\
\label{cec_v_middle}  \text{$v$ is affine and $v'=\alpha$ on $(x_{1},x_{2})$}, \qquad & (\text{as explained above}),\\
\label{cec_v_even} v'(L)=0,\qquad & \text{($v$ is an even function)},\\
\label{cec_v_b}|v|\le \beta,\qquad & \text{(condition \eqref{C3})},\\
\label{cec_v_a} |v'|\le \alpha,\qquad & \text{(condition \eqref{C2})}.
\end{align}
After some computations we find that
\[x_{1}=L-x_{2}=\sqrt{\frac{2\alpha}{\lambda}},\quad t=\sqrt{2\alpha\lambda},\]
\[v(x)=\begin{cases}
-\frac{\lambda}{6}x^{3}+ \sqrt{\frac{\alpha\lambda}{2}}x^{2}&  \text{if }\; x\in \left(0,\sqrt{\frac{2\alpha}{\lambda}}\right),\\
\alpha(x-x_{1})+\frac{\lambda x_{1}^{2}}{3}&  \text{if }\; x\in \left (\sqrt{\frac{2\alpha}{\lambda}},L-\sqrt{\frac{2\alpha}{\lambda}} \right ),\\
-\frac{\lambda}{6}(x+x_{1}-L)^{3}+\alpha(x+x_{1}-L)+\alpha (L-2x_{1})+\frac{\lambda x_{1}^{3}}{3} &  \text{if }\; x\in \left (L-\sqrt{\frac{2\alpha}{\lambda}},L\right).
\end{cases}
\]
From the equation $v''=h-u$ we can get an expression for $u$.
Moreover $x_{1}<x_{2}$ holds if and only if 
\[\alpha<\frac{\lambda L^{2}}{8}.\]
Finally one catch check that \eqref{cec_v_a} holds  and since $v$ is increasing on $(0,L)$, in order to satisfy \eqref{cec_v_b}, it suffices to have $v(L)\le \beta$ something that translates to
\[\beta\ge \alpha L-\frac{2\alpha}{3}\sqrt{\frac{2\alpha}{\lambda}}.\]

(2) The proof follows essentially the proof of $(1)$. Here we are looking for solutions of the type $u=-\mu x -t$ instead of $u=t$ on $(0,x_{1})$. In addition to the conditions \eqref{cec_v_cts}--\eqref{cec_v_a} for $v$ here we also have
\begin{equation}\label{aea_v_b}
v(L)=\beta,\qquad \text{(condition \eqref{C3})},
\end{equation}
because $w$ makes a positive jump at $x=L$, see also Figure \ref{h_allowed}(b). Again after some computations we find
\[x_{1}=L-x_{2}=\sqrt{\frac{2\alpha}{\lambda-\mu}},\quad t=\sqrt{2\alpha(\lambda-\mu)},\]
\[v(x)=\begin{cases}
-\frac{\lambda-\mu}{6}x^{3}+ \sqrt{\frac{\alpha(\lambda-\mu)}{2}}x^{2}&  \text{if }\; x\in \left(0,\sqrt{\frac{2\alpha}{\lambda-\mu}}\right),\\
\alpha(x-x_{1})+\frac{(\lambda-\mu)x_{1}^{2}}{3}&  \text{if }\; x\in \left (\sqrt{\frac{2\alpha}{\lambda-\mu}},L-\sqrt{\frac{2\alpha}{\lambda-\mu}} \right ),\\
-\frac{\lambda-\mu}{6}(x+x_{1}-L)^{3}+\alpha(x+x_{1}-L)+\alpha (L-2x_{1})+\frac{(\lambda-\mu) x_{1}^{3}}{3} &  \text{if }\; x\in \left (L-\sqrt{\frac{2\alpha}{\lambda-\mu}},L\right).
\end{cases}
\]
Again one can check that $|v'|<\alpha$. Moreover the condition $v(L)=\beta$ gives
\begin{equation}\label{aea_v_L}
\alpha L-\frac{2\alpha}{3}\sqrt{\frac{2\alpha}{\lambda--\mu}}=\beta \iff \mu=\lambda-\frac{8\alpha}{9\left (L-\frac{\beta}{\alpha} \right)^{2}}.
\end{equation}
Since we are looking into cases where $0<\mu<\lambda$ we must have $\sqrt{\frac{2\alpha}{\lambda-\mu}}>\sqrt{\frac{2\alpha}{\lambda}}$
and using \eqref{aea_v_L} this translates to
\begin{equation}\label{aea_condition1}
\beta<\alpha L-\frac{2\alpha}{3}\sqrt{\frac{2\alpha}{\lambda}}.
\end{equation}
Finally, it is easily checked that in order to impose $x_{1}<x_{2}$, we must have
\begin{equation}\label{aea_condition1}
\beta>\frac{2L\alpha}{3}.
\end{equation}

(3) In this case, we are looking for solutions of the type $u=-t$, $t>0$. The function $v$ will be a cubic polynomial that satisfy the conditions:
\[v(0)=0,\;v'(0)=0,\; v'(L)=0,\; v''=h+t,\;|v|\le \beta, \; |v'|\le \alpha.\]
and we easily compute
\[t=\frac{\lambda L}{2},\quad v(x)=-\frac{\lambda}{6}x^{3}+\frac{1}{2}t x^{2},\quad x\in(0,L).\]
We can also check that the conditions $|v|\le \beta$, $|v'|\le \alpha$ are equivalent to
\[\beta\ge \frac{\lambda L^{3}}{12},\quad \alpha\ge \frac{\lambda L^{2}}{8},\]
respectively.

(4) The proof is similar to (3). We are looking for a solution of the type $u(x)=-\mu x-t$, $x\in (0,L)$, $0<\mu<\lambda$ and $t>0$. As before, the function $v$ will be a cubic polynomial satisfying the conditions:
\[v(0)=0,\;v'(0)=0,\; v(L)=\beta,\; v'(L)=0,\; v''=h+t,\;|v|\le \beta, \; |v'|\le \alpha.\]
We get that
\[v(x)=-\frac{\lambda-\mu}{6}x^{3}+\frac{t}{2}x^{2},\quad x\in(0,L), \quad \text{with}\quad t=\frac{6\beta}{L^{2}},\quad \mu=\lambda-\frac{12\beta}{L^{3}}.\]
Thus $0<\mu<\lambda$ is equivalent to 
\[\beta<\frac{\lambda L^{3}}{12}.\]
Finally, we check easily that $|v|\le \beta$ holds and  $|v'|\le \alpha$ is equivalent to
\[\beta\le \frac{2L \alpha}{3}.\]
\end{proof}

\begin{figure}[h!]
\begin{center}
\includegraphics[scale=0.45]{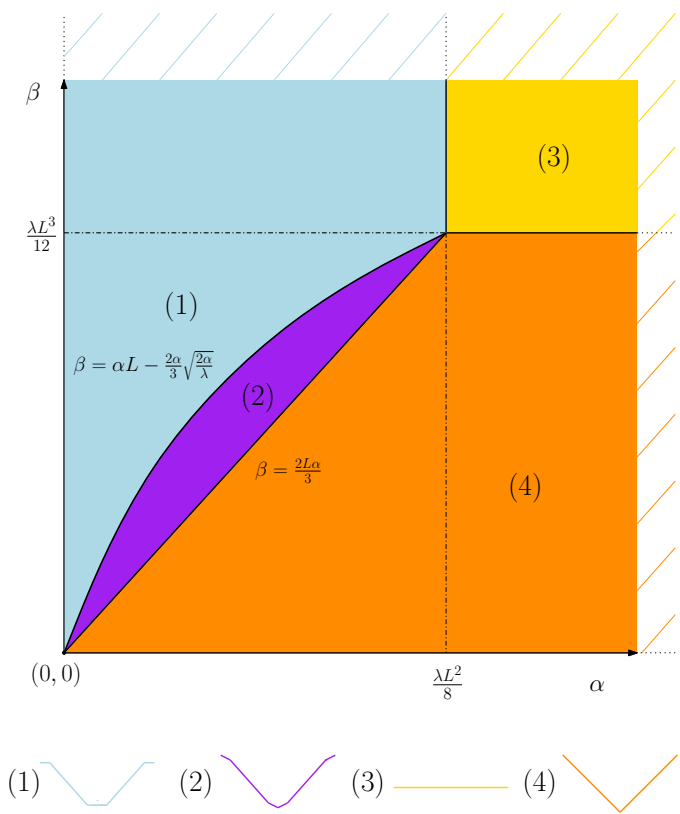}
\caption{Illustration of the four different types of solutions of \eqref{P} for the hat function $h$ that result for different combinations of the parameters $\alpha$ and $\beta$.}
\label{h_alpha_beta}
\end{center}
\end{figure}

As we did for the case of function $f$ we summarise how the solutions of the problem \eqref{P} with the data function $h$ are affected by $\alpha$ and $\beta$. In Figure \ref{h_alpha_beta} we have partitioned again  the set $\{\alpha>0,\;\beta>0\}$ into different areas that correspond to the four different possible solutions. We note again that in contrast to the cases of the functions $f$ and $g$, here we provide necessary and sufficient conditions for all the four different types of solutions. The corresponding areas are:
\begin{enumerate}[(1)]
\item Blue colour: $\left\{\alpha<\frac{\lambda L^{2}}{8},\; \beta\ge \alpha L -\frac{2\alpha}{3}\sqrt{\frac{2\alpha}{\lambda}}\right\}$,  \textbf{\ref{cec}}, see Figure \ref{h_allowed}(a).
\item Purple colour: $\left\{\beta>\frac{2L \alpha}{3},\;\beta<\alpha L -\frac{2\alpha}{3}\sqrt{\frac{2\alpha}{\lambda}}\right\}$, \textbf{\ref{aea}}, see Figure \ref{h_allowed}(b).
\item Yellow colour: $\left\{\alpha\ge \frac{\lambda L^{2}}{8},\; \beta\ge \frac{\lambda L^{3}}{12}\right\}$, \textbf{\ref{c}}, see Figure \ref{h_allowed}(c).
\item Orange colour: $\left\{\beta< \frac{\lambda L^{3}}{12},\; \beta\le\frac{2L \alpha}{3}\right\}$, \textbf{\ref{a}}, see Figure \ref{h_allowed}(d).
\end{enumerate}

\section{Numerical experiments}\label{section:numerics}
In this final section, we compare our theoretical results with numerical ones obtained by solving the discrete version of \eqref{P} with the primal-dual algorithm of Chambolle-Pock, \cite{chambolle2011first}. A description of the algorithm for TGV minimisation can be found in \cite{tgvcolour}. We also compute some numerical results with the presence of Gaussian noise that show that TGV regularisation is quite robust in the presence of noise. Let us note here that even though some sensitivity analysis can be done \cite{BredValk}, it is not an easy task to prove that the corresponding solutions of \eqref{P} with clean and corrupted data   have the same structure provided the noise is sufficiently small. Some relevant work has been done in \cite{benning2012ground} in terms of ground states of regularisation functionals.

In Figure \ref{f_aja_label}, we plot the exact and the numerical solutions of \eqref{P} using the function $f$ as a data function. We chose the parameters $\alpha$ and $\beta$ so that we have the solution of the type \ref{NP1J}. We observe that for clean data, the exact and the numerical solutions coincide, see Figure \ref{f_aja_label}(a). Moreover, even with the presence of noise, the numerical solution is not far away from the corresponding solution without the noise,  Figure \ref{f_aja_label}(b).

We observe similar results in Figures \ref{f_pajpa_label} and \ref{f_a_label} where we chose the parameters so that we have the solution of the type \ref{NP2J} and \ref{NP1C} respectively.

In Figure \ref{h_cec_label} we use the hat function $h$ as a data function. Again, without the presence of noise, the exact solution agrees with the numerical one, Figure \ref{h_cec_label}(a). However, when noise is added, even though the numerical solution is close to the  one that corresponds to the clean data, some staircasing is observed, Figure \ref{h_cec_label}(b). This is not surprising as with these combinations of $\alpha$ and $\beta$, TGV behaves like TV as it was shown in Theorem \ref{evenTV_label}. 

In Figure \ref{h_aea_label} the parameters are chosen so that the solution of the type \ref{aea} occurs. In the clean data case we have agreement between the exact and the numerical solution, Figure \ref{h_aea_label}(a), but in the noisy case a kind of ``affine'' staircasing effect appears in the area where the exact solution equals with the data, see detail of Figure \ref{h_aea_label}(b).

Finally, in Figure \ref{h_a_label}, the parameters are chosen so the solution of the type \ref{a} occurs. The numerical solution agrees with the exact one and deviates from it slightly in the presence of noise.

\begin{figure}[h!]
\begin{center}
\subfloat[Clean data, the numerical solution agrees with the exact one.]{
\hspace{-0.3 cm}\includegraphics[scale=0.35]{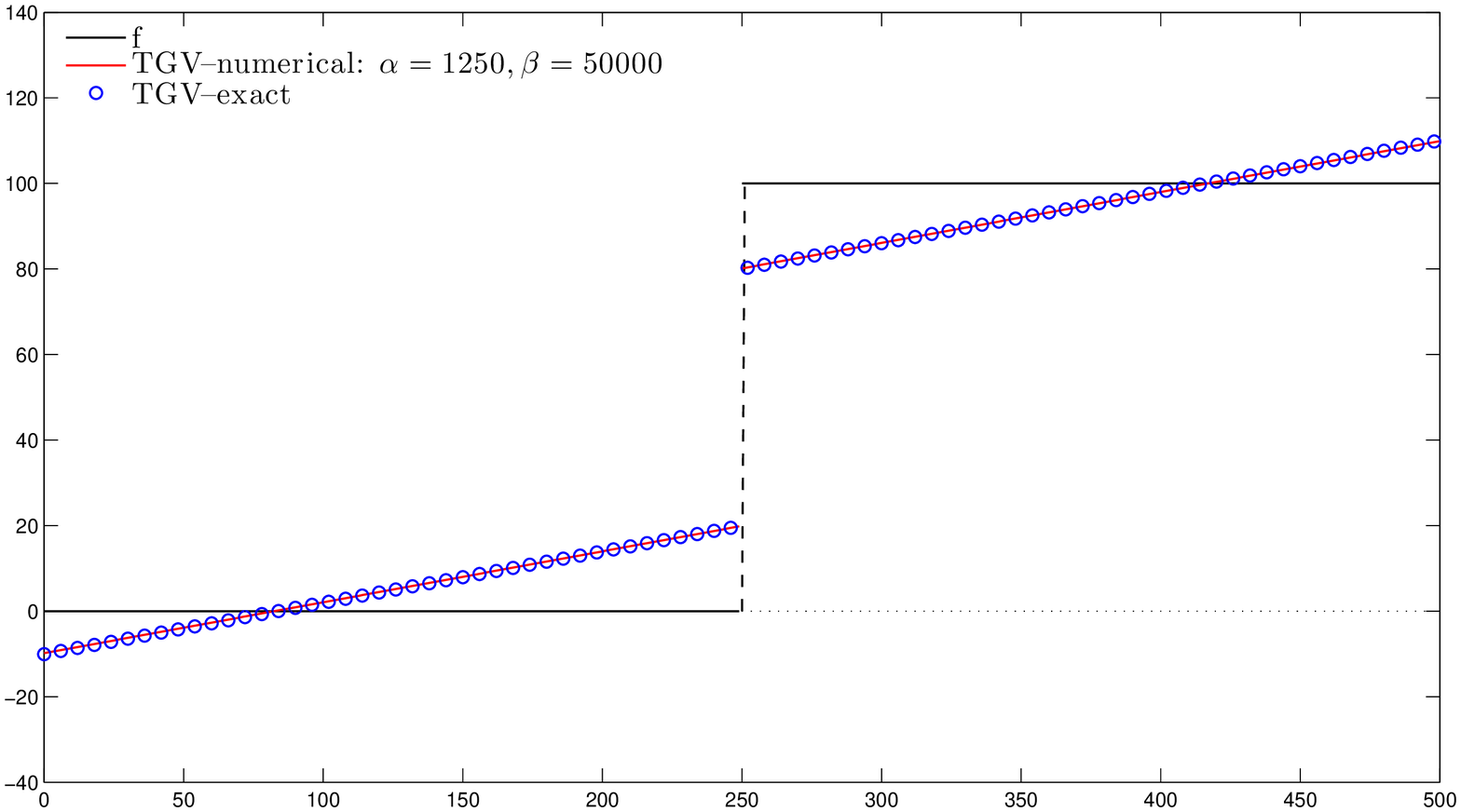}
}
\subfloat[Noisy data, the numerical solution deviates slightly from the corresponding exact solution with clean data.]{
\includegraphics[scale=0.35]{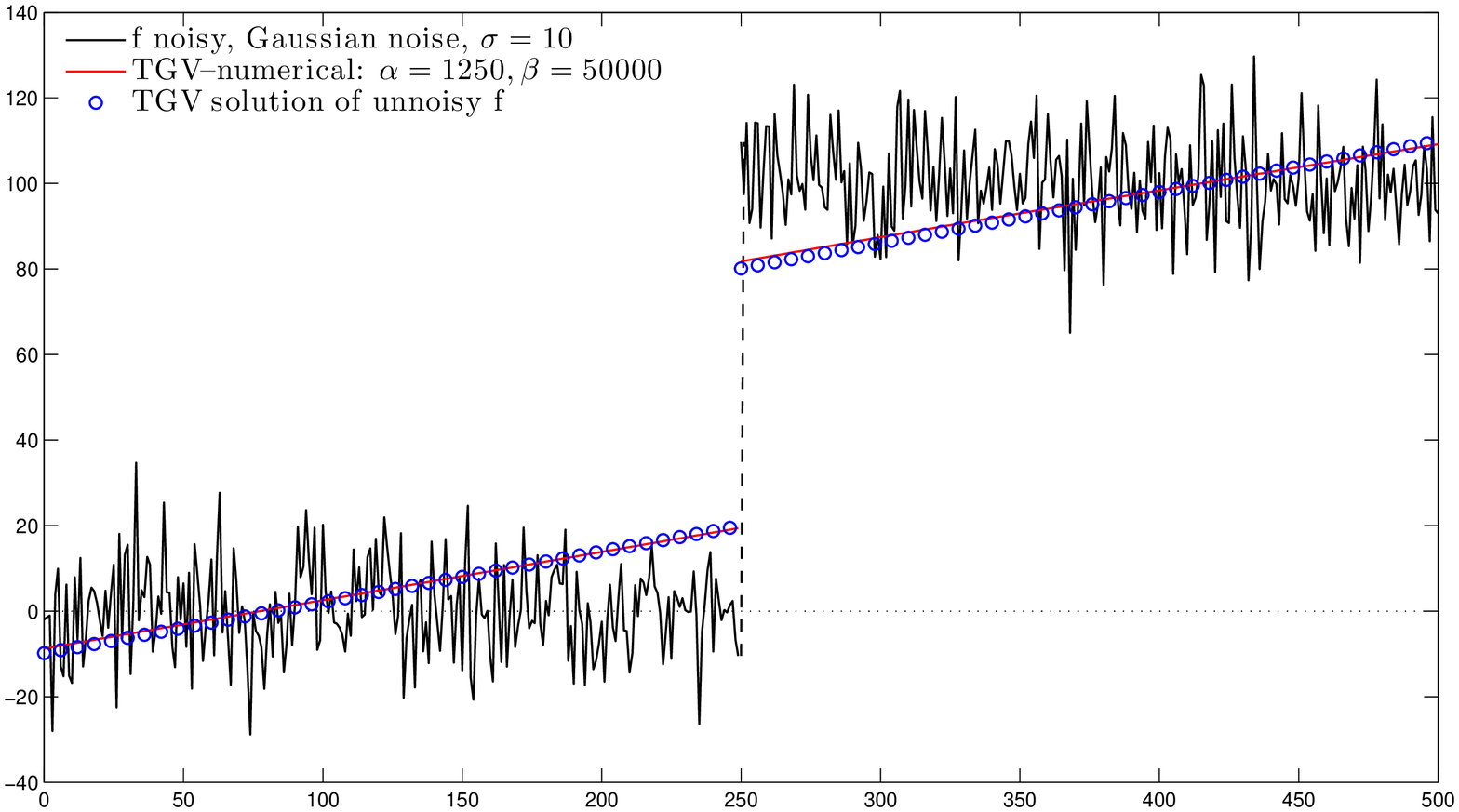}
}
\caption{Piecewise constant function $f$. The parameters $\alpha$ and $\beta$ satisfy the conditions \eqref{np1j_conditions}, thus the solution is of the type \ref{NP1J}. }
\label{f_aja_label}
\end{center}
\end{figure}

\begin{figure}[h!]
\begin{center}
\subfloat[Clean data, the numerical solution agrees with the exact one.]{
\hspace{-0.3 cm}\includegraphics[scale=0.35]{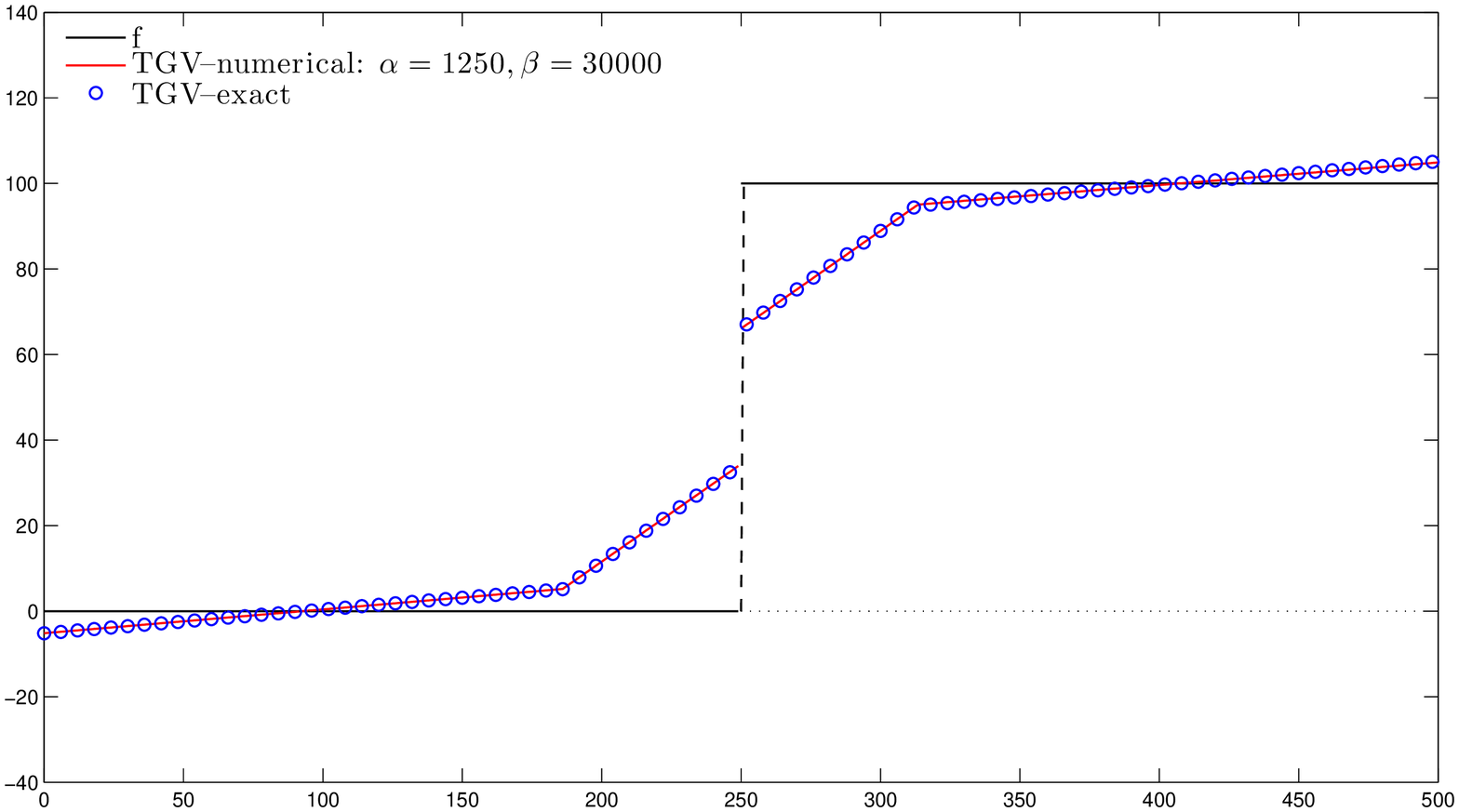}
}
\subfloat[Noisy data, the numerical solution deviates slightly from the corresponding exact solution with clean data.]{
\includegraphics[scale=0.35]{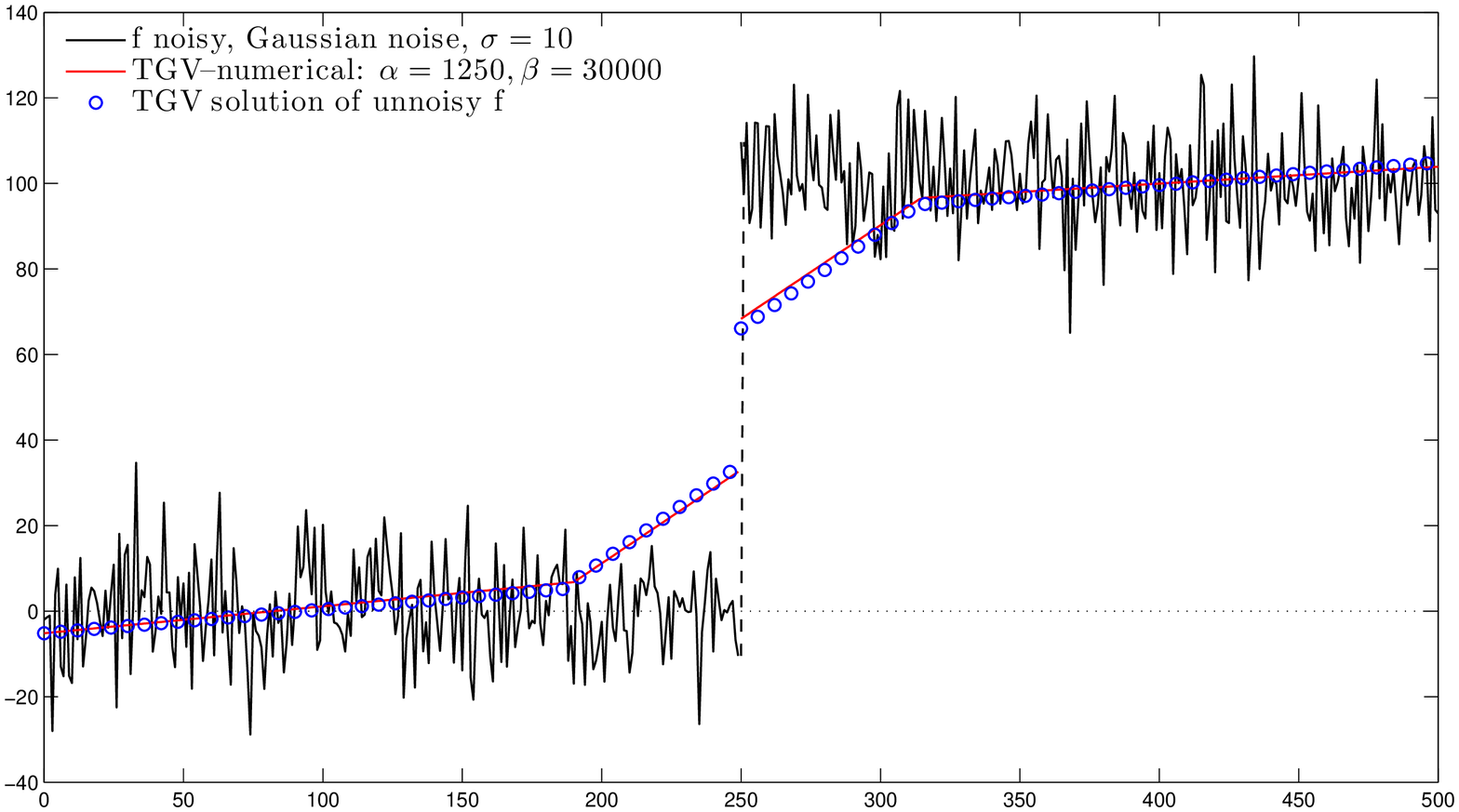}
}
\caption{Piecewise constant function $f$. The parameters $\alpha$ and $\beta$ satisfy the conditions \eqref{np2j_conditions}, thus the solution is of the type \ref{NP2J}.}
\label{f_pajpa_label}
\end{center}
\end{figure}

\begin{figure}[h!]
\begin{center}
\subfloat[Clean data, the numerical solution agrees with the exact one.]{
\hspace{-0.3 cm}\includegraphics[scale=0.35]{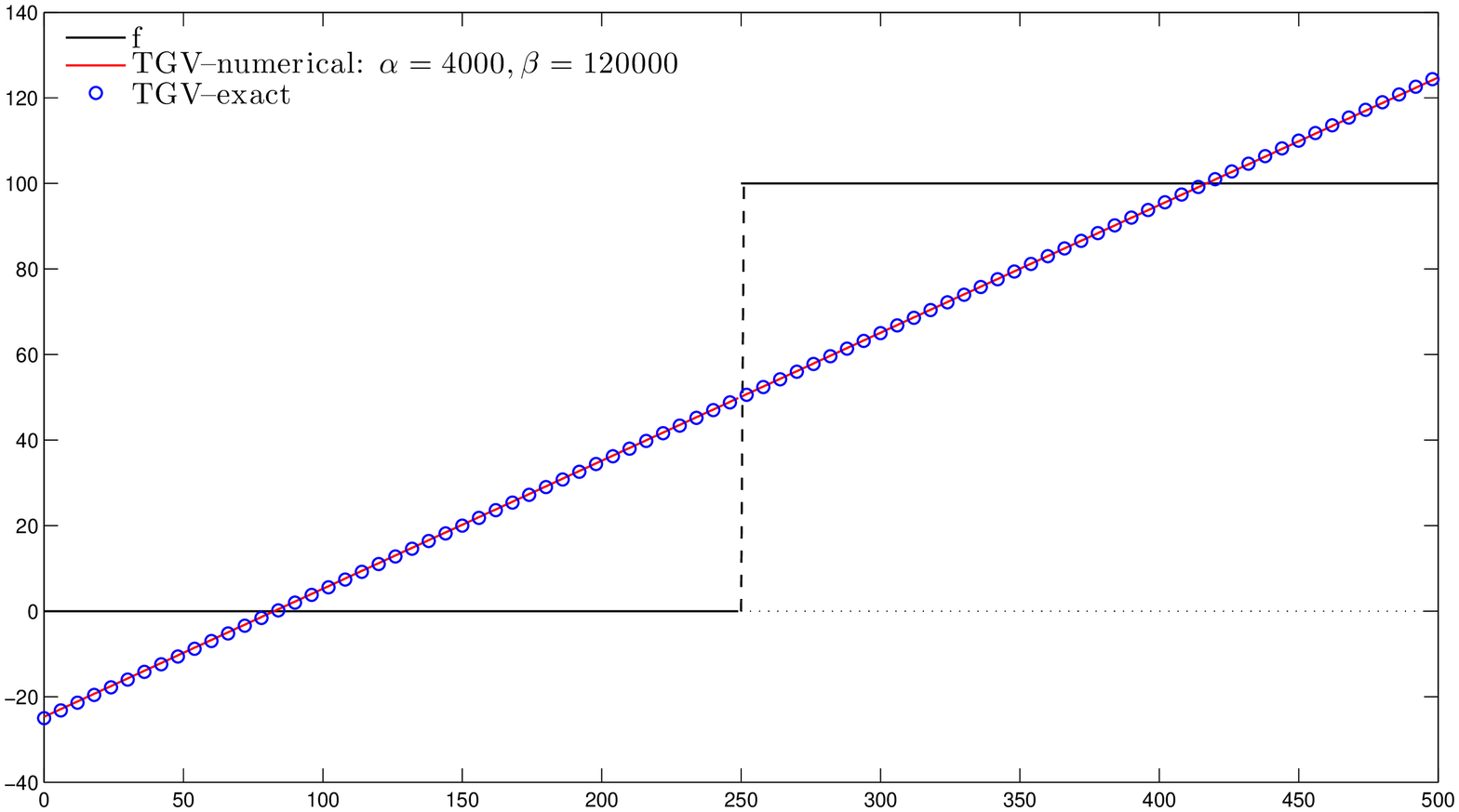}
}
\subfloat[Noisy data, the numerical solution deviates slightly from the corresponding exact solution with clean data.]{
\includegraphics[scale=0.35]{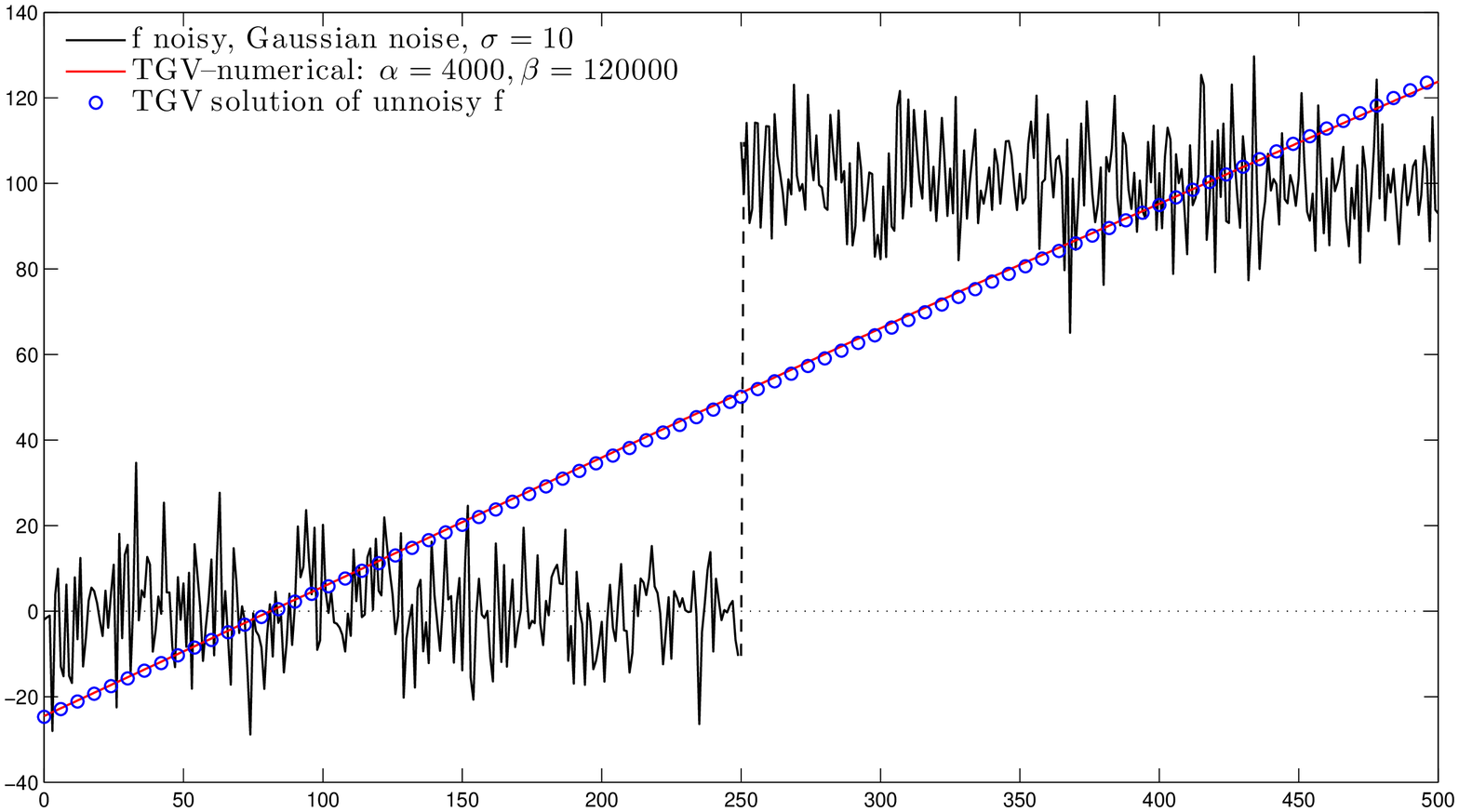}
}
\caption{Piecewise constant function $f$. The parameters $\alpha$ and $\beta$ satisfy the conditions \eqref{np1c_conditions}, thus the solution is of the type \ref{NP1C} ($L^{2}$--linear regression).}
\label{f_a_label}
\end{center}
\end{figure}

\begin{figure}[h!]
\begin{center}
\subfloat[Clean data, the numerical solution agrees with the exact one.]{
\hspace{-0.3 cm}\includegraphics[scale=0.45]{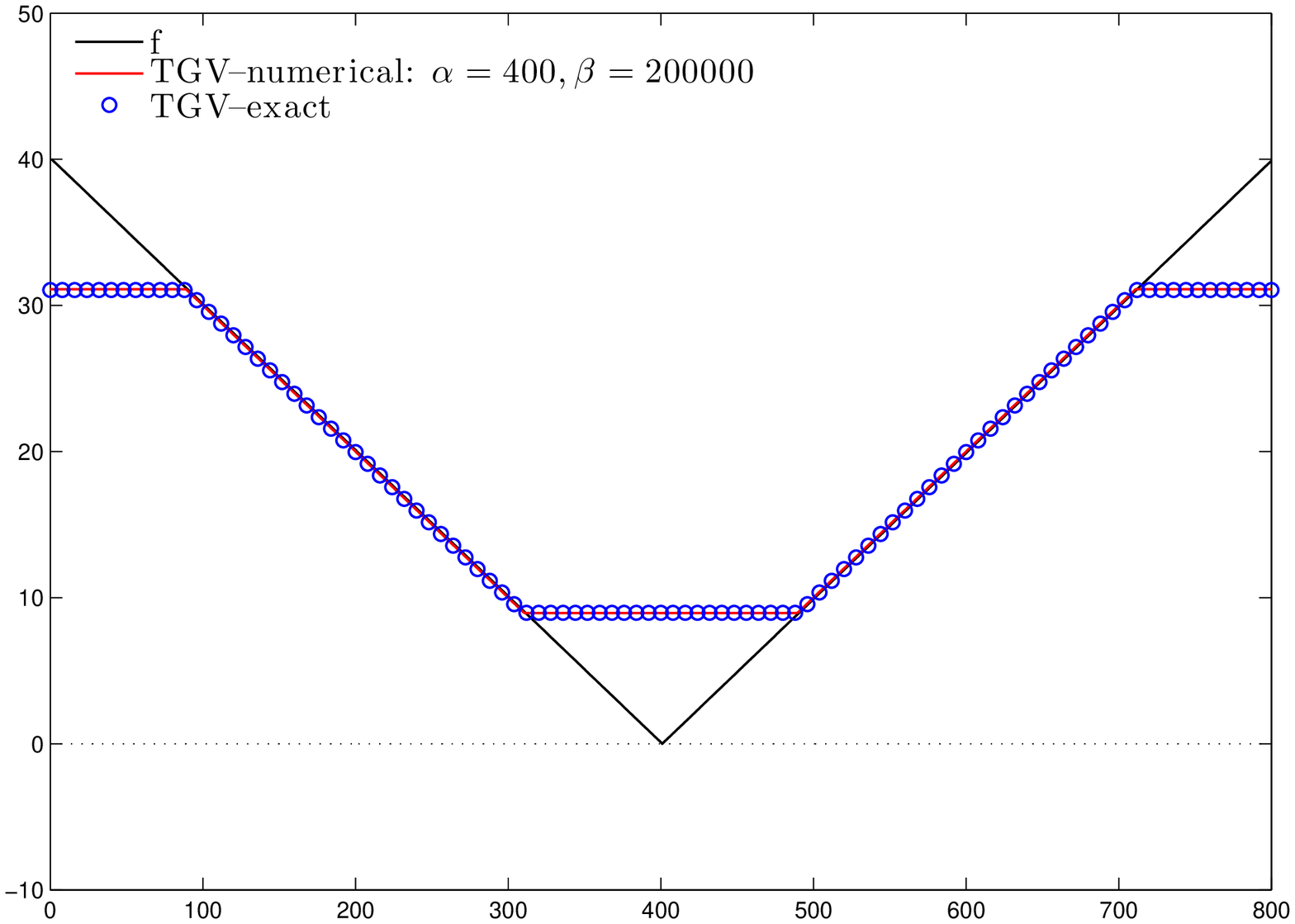}
}
\subfloat[Noisy data, appearance of staircasing effect.]{
\includegraphics[scale=0.45]{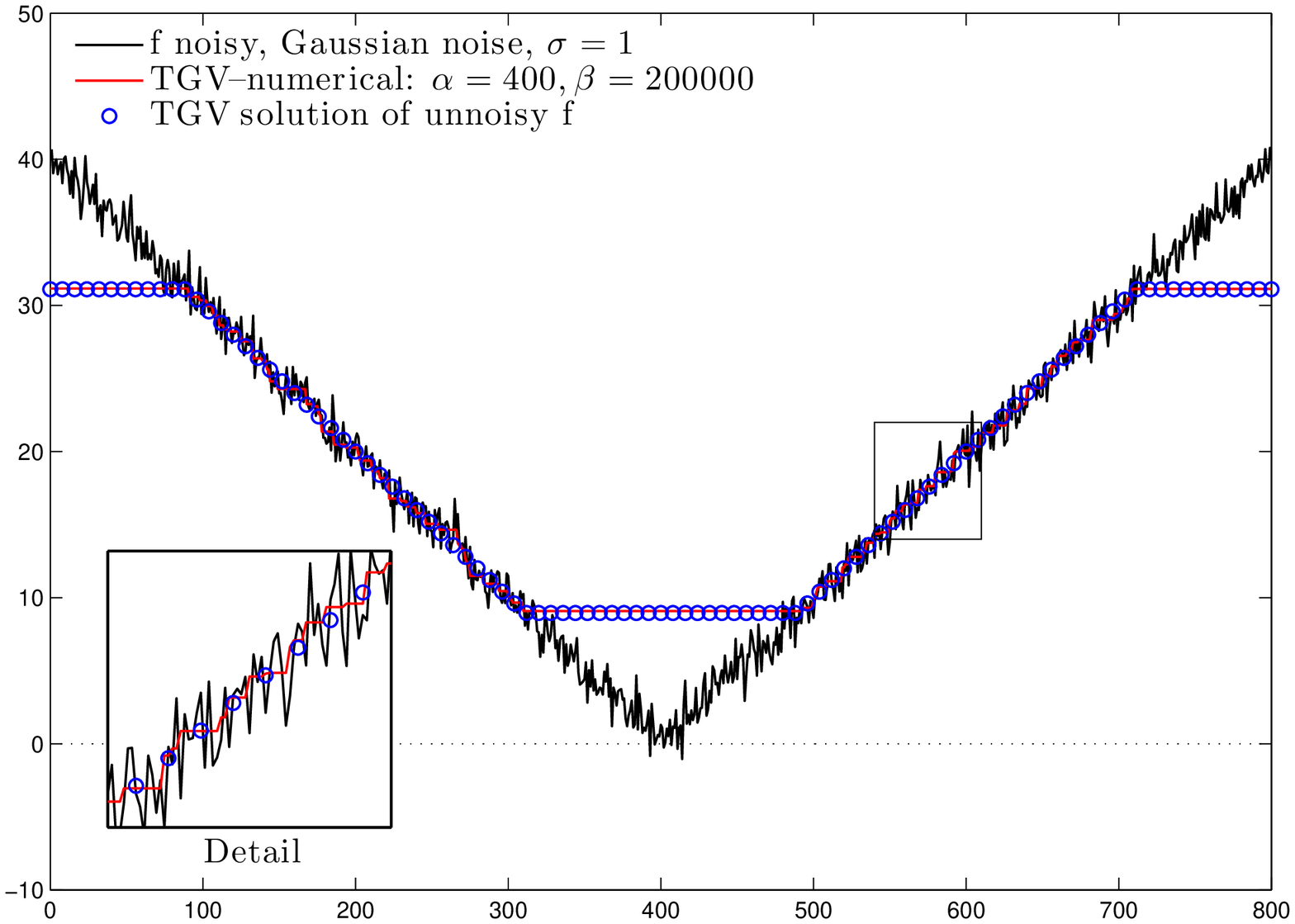}
}
\caption{Hat function $h$. The parameters $\alpha$ and $\beta$ satisfy the conditions \eqref{cec_cond}, thus the  solution is of the type \ref{cec}.}
\label{h_cec_label}
\end{center}
\end{figure}

\begin{figure}[h!]
\begin{center}
\subfloat[Clean data, the numerical solution agrees with the exact one.]{
\hspace{-0.3 cm}\includegraphics[scale=0.45]{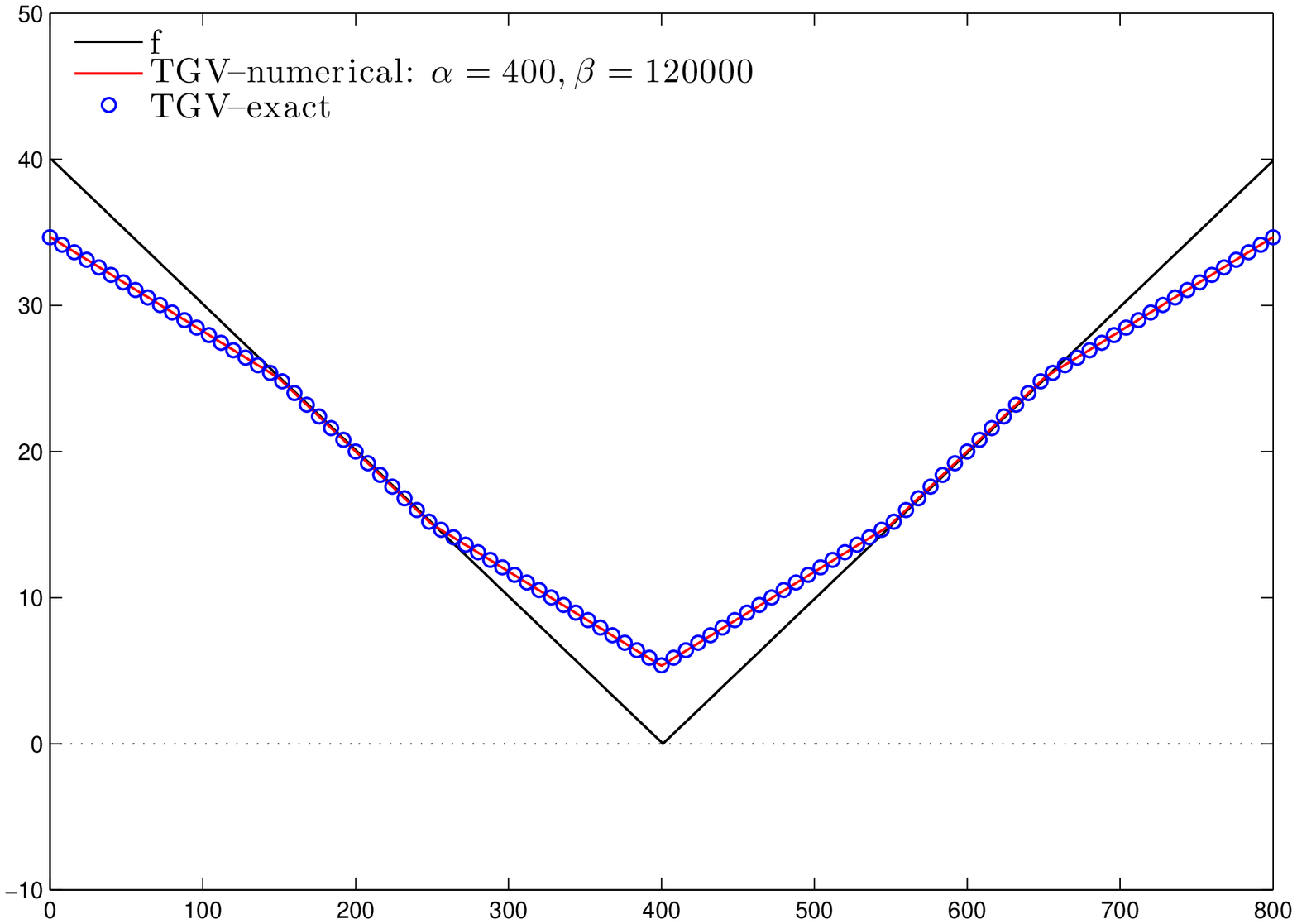}
}
\subfloat[Noisy data, appearance of an ``affine'' staircasing effect.]{
\includegraphics[scale=0.45]{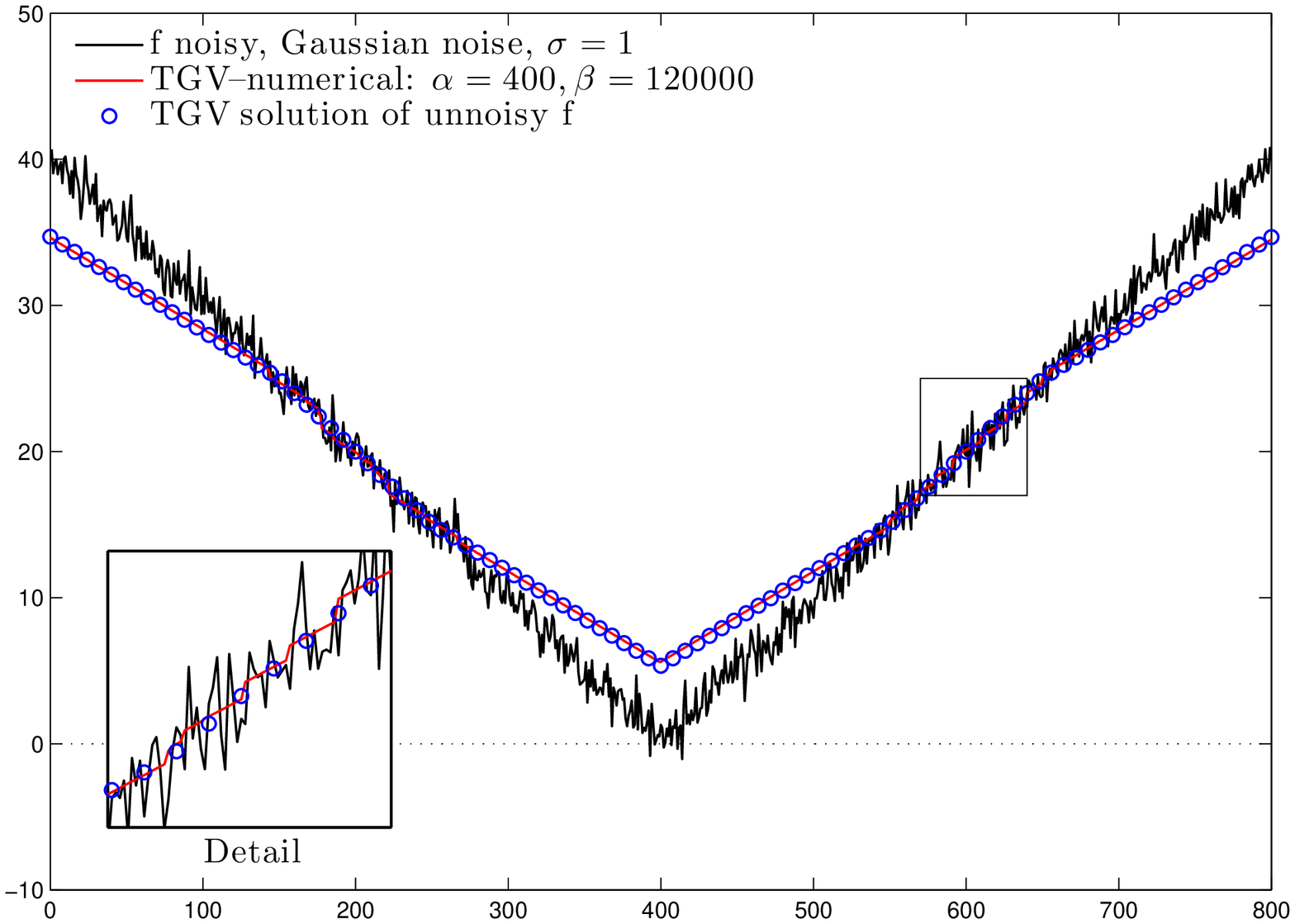}
}
\caption{Hat function $h$. The parameters $\alpha$ and $\beta$ satisfy the conditions \eqref{aea_cond}, thus the  solution is of the type \ref{aea}.}
\label{h_aea_label}
\end{center}
\end{figure}

\begin{figure}[h!]
\begin{center}
\subfloat[Clean data, the numerical solution agrees with the exact one.]{
\hspace{-0.3 cm}\includegraphics[scale=0.45]{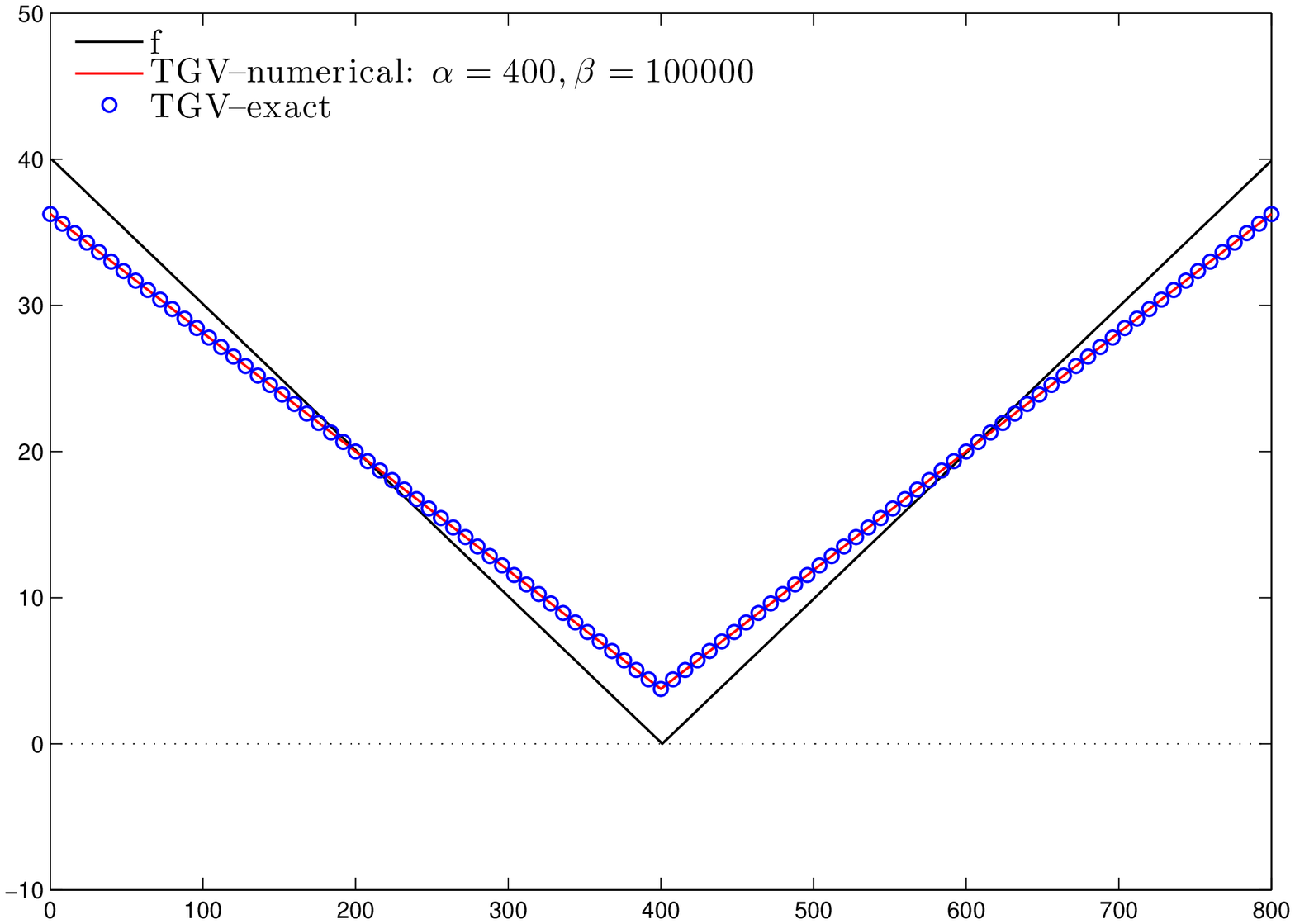}
}
\subfloat[Noisy data, the numerical solution deviates slightly from the corresponding exact solution with clean data.]{
\includegraphics[scale=0.45]{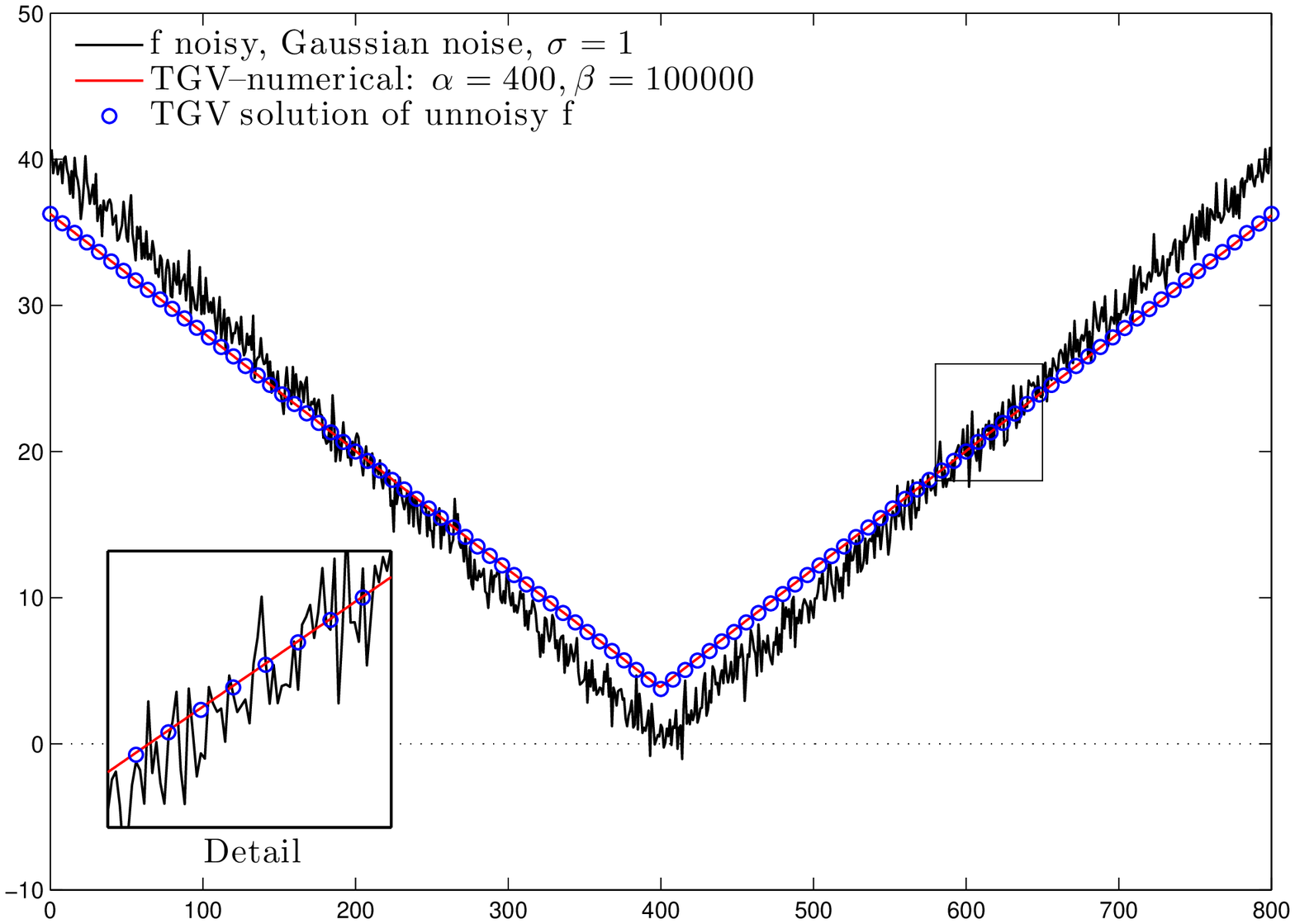}
}
\caption{Hat function $h$. The parameters $\alpha$ and $\beta$ satisfy the conditions \eqref{a_cond}, thus the  solution is of the type \ref{a}.}
\label{h_a_label} 
\end{center}
\end{figure}

\section{Conclusion}
We studied exact solutions to the one dimensional $L^{2}$--$\mathrm{TGV}_{\beta,\alpha}^{2}$ problem for simple piecewise constant, piecewise affine and hat functions as data terms. We used Fenchel--Rockafellar duality to derive the optimality conditions of the corresponding minimisation problem and we computed the exact solutions using these conditions. The relationship between the values of the parameters $\alpha$, $\beta$ and the structure of solutions was investigated. We performed numerical experiments in which the exact solutions agree with the numerical ones and having only a slight deviation when noise is added.

As far as further research is concerned, the analysis of the exact solutions of the corresponding 2D model should be the first priority as well as a more rigorous study of the problem under the presence of noise.

\subsection*{Acknowledgements}
The first author acknowledge the support of the UK Engineering and Physical Sciences Research Council (EPSRC) grant EP/H023348/1 for the University of Cambridge Centre for Doctoral Training, the Cambridge Centre for Analysis, the financial support provided by the EPSRC first grant Nr. EP/J009539/1 ``Sparse \& Higher-order Image Restoration'' and the Award No. \hspace{-4 pt}KUK-I1-007-43, made by King Abdullah University of Science and Technology (KAUST).

\bibliographystyle{amsalpha}

\bibliography{kostasbib}

\end{document}